\documentclass{amsart} 

\usepackage{mathtools}
\usepackage{lmodern} 
\usepackage{microtype} 
\usepackage[english]{babel} 
\usepackage[hidelinks]{hyperref} 
\usepackage{mathrsfs} 
\usepackage{xcolor} 
\usepackage{tikz} 
\usepackage{pgfplots} 
\pgfplotsset{compat=1.16}

\usepackage{tikz-cd} 
\tikzcdset{arrow style=tikz, arrows={thick}, diagrams={>=stealth}}

\numberwithin{equation}{section} 
\numberwithin{figure}{section}
\numberwithin{table}{section} 

\theoremstyle{plain}
\newtheorem{theorem}{Theorem}[section] 
\newtheorem{lemma}[theorem]{Lemma} 
\newtheorem{corollary}[theorem]{Corollary} 
\newtheorem{conjecture}[theorem]{Conjecture}

\theoremstyle{definition} 
\newtheorem*{definition}{Definition} 
\newtheorem{example}[theorem]{Example}

\DeclareMathOperator{\mre}{Re}

\DeclareMathOperator{\B}{B} 
\DeclareMathOperator{\sinc}{sinc} 
\DeclareMathOperator{\dist}{dist}

\begin{document} 
\title{Point evaluation in Paley--Wiener spaces} 
\date{\today} 

\author[O. F. Brevig]{Ole Fredrik Brevig} 
\address{Department of Mathematics, University of Oslo, 0851 Oslo, Norway} 
\email{obrevig@math.uio.no}

\author[A. Chirre]{Andr\'{e}s Chirre} 
\address{Department of Mathematics, University of Rochester, Rochester, NY 14627, USA} 
\email{cchirrec@math.rochester.edu}

\author[J. Ortega-Cerd\`{a}]{Joaquim Ortega-Cerd\`{a}} 
\address{Department de Matem\`{a}tiques i Inform\`{a}tica, Universitat de Barcelona, Barcelona, Spain and Centre de Recerca Matem\`{a}tica, Barcelona, Spain}
\email{jortega@ub.edu}

\author[K. Seip]{Kristian Seip} 
\address{Department of Mathematical Sciences, Norwegian University of Science and Technology (NTNU), 7491 Trondheim, Norway} 
\email{kristian.seip@ntnu.no}

\thanks{Ortega-Cerd\`{a} was supported in part by the Spanish Ministerio de Ciencia, Innovaci\'on, project PID2021-123405NB-I00 and by the AGAUR grant 2021 SGR 00087. Seip was supported in part by the Research Council of Norway grant 275113}

\begin{abstract}
	We study the norm of point evaluation at the origin in the Paley--Wiener space $PW^p$ for $0 < p < \infty$, i.~e., we search for the smallest positive constant $C$, called $\mathscr{C}_p$, such that the inequality $|f(0)|^p \leq C \|f\|_p^p$ holds for every $f$ in $PW^p$. We present evidence and prove several results supporting the following monotonicity conjecture: The function $p\mapsto \mathscr{C}_p/p$ is strictly decreasing on the half-line $(0,\infty)$. Our main result implies that $\mathscr{C}_p <p/2$ for $2<p<\infty$, and we verify numerically that $\mathscr{C}_p > p/2$ for $1 \leq p < 2$. We also estimate the asymptotic behavior of $\mathscr{C}_p$ as $p \to \infty$ and as $p \to 0^+$. Our approach is based on expressing $\mathscr{C}_p$ as the solution of an extremal problem. Extremal functions exist for all $0<p<\infty$; they are real entire functions with only real zeros, and the extremal functions are known to be unique for $1\leq p < \infty$. Following work of H\"{o}rmander and Bernhardsson, we rely on certain orthogonality relations associated with the zeros of extremal functions, along with certain integral formulas representing respectively extremal functions and general functions at the origin. We also use precise numerical estimates for the largest eigenvalue of the Landau--Pollak--Slepian operator of time--frequency concentration. A number of qualitative and quantitative results on the distribution of the zeros of extremal functions are established. In the range $1<p<\infty$, the orthogonality relations associated with the zeros of the extremal function are linked to a de Branges space. We state a number of conjectures and further open problems pertaining to $\mathscr{C}_p$ and the extremal functions. 
\end{abstract}


\maketitle

\setcounter{tocdepth}{1} 

\tableofcontents

\section{Introduction} This paper studies the following problem on time--frequency localization: What is the smallest positive constant $C$, to be called $\mathscr{C}_p$ in what follows, such that the inequality 
\begin{equation}\label{eq:pointeval} 
	|f(0)|^p \leq C \|f\|_p^p 
\end{equation}
holds for every $f$ in the Paley--Wiener space $PW^p$ with $0 < p < \infty$? Here $PW^p$ is the subspace of $L^p(\mathbb{R})$ consisting of entire functions of exponential type at most $\pi$, and $\|f\|_p^p$ denotes the $L^p$ integral of $f$. By the translation invariance of $PW^p$, we may replace $|f(0)|$ by $|f(x)|$ for any real number $x$ in \eqref{eq:pointeval} without affecting the value of $\mathscr{C}_p$. Consequently, the constant $\mathscr{C}_p^{1/p}$ also represents the norm of the embedding of $PW^p$ into $PW^\infty$.

We observe that $\mathscr{C}_2 = 1$ from the reproducing kernel formula 
\begin{equation}\label{eq:2rpk} 
	f(0) = \int_{-\infty}^\infty f(x)\,\sinc{\pi x}\,dx, 
\end{equation}
valid for all $f$ in $PW^2$, by use of the Cauchy--Schwarz inequality. In \eqref{eq:2rpk} and in what follows, we employ the notation
\[\sinc{x} := \frac{\sin{x}}{x}.\]
There seem to be few nontrivial results for $p\neq 2$, with a notable exception for the case $p=1$ which has been studied by several authors (see e.g. \cite{AKP96,CMS19,Gorbachev05}) owing largely to its relevance for problems in analytic number theory. The best numerical estimates in this case, 
\begin{equation}\label{eq:hb} 
	0.5409288219 \leq \mathscr{C}_1 \leq 0.5409288220, 
\end{equation}
were found by H\"ormander and Bernhardsson \cite{HB93}, whose interest in $\mathscr{C}_1$ was motivated by its appearance in a Bohr-type estimate for the Cauchy--Riemann operator in $\mathbb R^2$. 

We will present evidence and prove several results supporting what we will refer to as the \emph{monotonicity conjecture}, namely that $p\mapsto \mathscr{C}_p/p$ is a strictly decreasing function on the half-line $(0,\infty)$. Our first result to that effect reads as follows.
\begin{theorem}\label{thm:korevaar} 
	If $2\leq p \leq 4$, then
	\[\mathscr{C}_p \leq \frac{p}{2}\left(1- 2(p-2)\int_1^\infty (\sinc{ \pi x})^2 \, \frac{4x+p-2}{(2x+p-2)^2}\,dx\right).\]
\end{theorem}

We have not been able to find a similar explicit bound for $\mathscr{C}_p$ in the range $1<p<2$. We will however offer some numerical evidence for the monotonicity conjecture also here. To this end, let $\B$ denote the beta function and set 
\begin{equation}\label{eq:fp} 
	f_p(z) := \frac{2}{\B(1/2,2/p)} \int_{-\pi}^\pi \left(1-\frac{\xi^2}{\pi^2}\right)^{\frac{2}{p}-1}\,e^{iz\xi}\,\frac{d\xi}{2\pi}. 
\end{equation}
The normalization factor is chosen so that $f_p(0)=1$. The lower bound in \eqref{eq:hb} is obtained by a small perturbation of
\begin{equation} \label{eq:HBR} f_1(z) = \frac{3\sin{\pi z}-3\pi z\cos{\pi z}}{\pi^3 z^3}\end{equation}
while $f_2(z) = \sinc{\pi z}$ is the extremal function for $p=2$. In general, $f_p$ can be expressed as a confluent hypergeometric function or in terms of Bessel functions of the first kind. 

Figure~\ref{fig:plot} exhibits numerical evidence for the monotonicity conjecture in the range $1\leq p \leq 4$, obtained using the \texttt{integrate} and \texttt{special.hyp0f1} packages from SciPy. At $p=1$, the lower bound arising from \eqref{eq:HBR} differs from H\"{o}rmander and Bernhardsson's bounds only in the fourth decimal place. We believe that the blue curve in Figure~\ref{fig:plot} in essence represents the true value of $\mathscr{C}_p$ for $1\leq p \leq 4$. In Section~\ref{subsec:conv}, we will present further evidence for this claim in the case $p=4$. However, as we will see below, the case $p=1$ is particularly favorable, and it seems hard to get numerical bounds of similarly high precision for $p>1$.

\begin{figure}
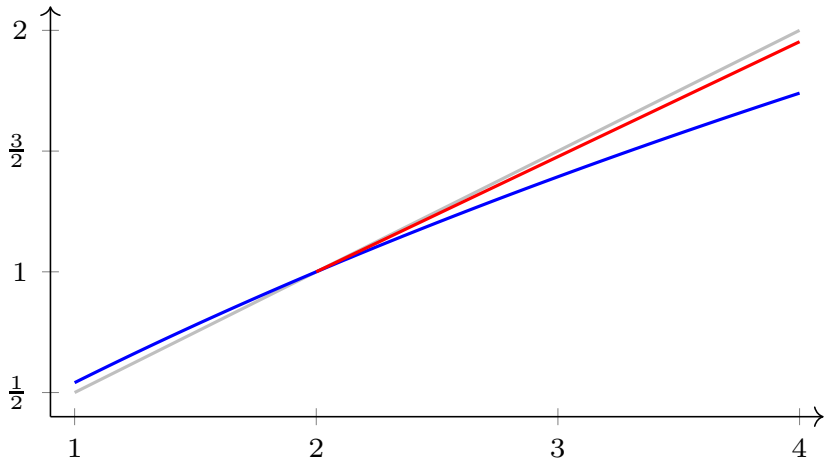

	\centering
	\begin{tikzpicture}[scale=1.5]
		\begin{axis}
			[axis equal image,
			axis lines=middle,
			axis line style=thin,
			xmin=1-2*0.05,
			xmax=4+2*0.05,
			xtick={1,2,3,4},
			xticklabels={$\scriptstyle 1$,$\scriptstyle 2$,$\scriptstyle 3$,$\scriptstyle 4$},
			ymin=0.5-0.1,
			ymax=2+0.1,
			ytick={0.5,1,1.5,2,2.5,3.5,4},
			yticklabels={$\scriptstyle \frac{1}{2}$,$\scriptstyle 1$,$\scriptstyle \frac{3}{2}$, $\scriptstyle 2$ },
			every axis x label/.style={ at={(ticklabel* cs:1.025)}, 
				anchor=west,},
			every axis y label/.style={ at={(ticklabel* cs:1.025)}, 
				anchor=south,},
			axis line style={->}]
			\addplot[thick,draw opacity=0.5,color=gray] coordinates {(1,0.5) (4,2)};
			\input{bessel.tex}
			\input{upper2p4.tex}
		\end{axis}
	\end{tikzpicture}
	\caption{Plot of the {\color{blue} lower bound for $\mathscr{C}_p$} obtained numerically by testing with the function $f_p$ from \eqref{eq:fp} for $1 \leq p \leq 4$ and the {\color{red} upper bound for $\mathscr{C}_p$} from Theorem~\ref{thm:korevaar} for $2 \leq p \leq 4$. The line {\color{gray} $p/2$} can be seen in the background.}
	\label{fig:plot}
\end{figure}

By the power trick (see Lemma~\ref{lem:powertrick} below), which implies that $\mathscr{C}_{2p} \leq 2\mathscr{C}_p$, Theorem~\ref{thm:korevaar} and these numerics yield bounds for $\mathscr{C}_p$ in the whole range $0<p<\infty$. In particular, we find that $\mathscr{C}_p<p/2$ for all $p>2$. This is an improvement on the best previously known estimate $\mathscr{C}_p \leq \lceil p/2 \rceil$ for all $p>2$, which follows from $\mathscr{C}_2=1$ and the power trick (see e.g. \cite{Gorbachev21,Ibragimov60,Korevaar49}). We may also obtain a uniform bound
\[\mathscr{C}_p\leq p/2-A(p-2)\]
for $2<p<\infty$ by an elaboration of the method used to prove Theorem~\ref{thm:korevaar}. Of greater interest, however, is the asymptotic behavior of $\mathscr{C}_p$ as $p\to \infty$ and $p \to 0^+$, respectively. In the former case, we obtain the following precise result. 
\begin{theorem}\label{thm:Cpinfty} 
	There exist two positive constants $A$ and $B$ such that
	\[- A \frac{\log p}{\sqrt{p}} \leq \mathscr{C}_p - \sqrt{\frac{\pi p}{2}} \leq - B \frac{\log p}{\sqrt{p}}\]
	for all sufficiently large $p$. 
\end{theorem}
In the case $p\to 0^+ $ which is of a rather different nature, our result reads as follows: 

\pagebreak 

\begin{theorem}\label{thm:Cp0} 
	\mbox{} 
	\begin{enumerate}
		\item[(a)] There is a positive number $c_0$ such that
		\[\lim_{p \to 0^+} \frac{2}{p}\, \mathscr{C}_p = c_0.\]
		\item[(b)] The number $c_0$ lies in the interval $[1.1393830,1.1481785]$. 
	\end{enumerate}
\end{theorem}

Our approach is based on expressing $\mathscr{C}_p$ as the solution of an extremal problem, namely 
\begin{equation}\label{eq:extremalproblem} 
	\frac{1}{\mathscr{C}_p} = \inf_{f \in PW^p} \left\{\|f\|_p^p \,:\, f(0) = 1\right\}. 
\end{equation}
This extremal problem was studied by Levin and Lubinsky \cite{LL15}, who found that the quantity $\mathscr{C}_p^{-1}$ appears as a scaling limit for certain Christoffel functions. See also Lubinsky's survey \cite{Lubinsky17}, where Problems 4 and 5 specifically ask for estimates of $\mathscr{C}_p$ and for information on the solutions of \eqref{eq:extremalproblem}. Our upper bounds for $\mathscr{C}_p$ imply new estimates for the Nikolskii constants for trigonometric polynomials by \cite[Theorem~1]{GM18}.

A compactness argument shows that the extremal problem \eqref{eq:extremalproblem} has solutions for any $0<p<\infty$, and a rescaling argument shows that the type of these extremal functions is exactly $\pi$. The analysis of the particularly simple case $p=2$ above also shows that the corresponding extremal problem has the unique solution $\varphi_2(z) = \sinc{\pi z}$. In this case, all zeros are simple, and the zero set consists of the nonzero integers. Our first main result pertaining to the solutions of \eqref{eq:extremalproblem} is a qualitative analogue of this observation. 

To state that result, we introduce the following notation and terminology. Given an entire function $\varphi$, we let $\mathscr{Z}(\varphi)$ denote its zero set. We define the \emph{separation constant} of a set of real numbers $\Lambda$ as
\[\sigma(\Lambda) := \inf\big\{ |\lambda-\mu|\,:\, \lambda, \mu\in \Lambda\,\text{ and }\, \lambda\neq \mu\big\}.\]
We say that $\Lambda$ is \emph{uniformly discrete} if $\sigma(\Lambda)>0$ and that $\Lambda$ is \emph{uniformly dense} if there exists a positive number $L$ such that every interval of length $L$ contains at least one element from $\Lambda$.
\begin{theorem}\label{thm:zeroset} 
	Fix $0<p<\infty$ and suppose that $\varphi$ is a solution of the extremal 
problem \eqref{eq:extremalproblem}. The zeros of $\varphi$ are all real. Moreover,
	\begin{enumerate}
		\item[(a)] $\mathscr{Z}(\varphi)$ is a finite union of uniformly discrete sets; 
		\item[(b)] if $p\geq \frac{1}{2}$, then $\mathscr{Z}(\varphi)$ is uniformly discrete; 
		\item[(c)] if $p\geq1$, then $\mathscr{Z}(\varphi)$ is uniformly dense. 
	\end{enumerate}
\end{theorem}

It is not difficult to see that the extremal problem \eqref{eq:extremalproblem} has a unique solution for fixed $p$ in the convex range $1 \leq p < \infty$. This extremal function $\varphi_p$ must then necessarily be even and consequently of the form
\[\varphi_p(z) = \prod_{n=1}^\infty \left(1-\frac{z^2}{t_n^2}\right),\]
where $(t_n)_{n\geq1}$ is a strictly increasing sequence of positive numbers. Our second main result on solutions of \eqref{eq:extremalproblem} and a crucial ingredient in the proof of Theorem~\ref{thm:korevaar} is the following quantitative version of Theorem~\ref{thm:zeroset}~(b).

\pagebreak 

\begin{theorem}\label{thm:sep} 
	If $2 \leq p \leq 4$, then 
	\begin{enumerate}
		\item[(a)] $t_1 \geq \frac{2}{\pi}$; 
		\item[(b)] $t_{n+1} - t_n \geq \frac{2}{3}$ for every $n\geq1$. 
	\end{enumerate}
\end{theorem}
We have verified that our methods allow us to improve both these bounds slightly; we could for instance have replaced $2/3$ by $0.6778$ in part (b). Here the point is, however, that $2/3$ is the exact bound required in the proof of Theorem~\ref{thm:korevaar}. The weaker bounds $t_1\geq1/2$ and $t_{n+1}-t_n\geq3/5$ will be shown to hold respectively in the full range $0<p<\infty$ and for $2<p<\infty$.

\subsection*{Outline of the paper} 
As a guide to the reading of this rather long paper, we now give an outline of the various sections and the main ideas involved in them. The reader may find it useful to consult the dependence relations between these sections shown in Figure~\ref{fig:dependence}.

The next section---the prologue of our paper---addresses what appears to be a main challenge to establish the monotonicity conjecture: How can we extend the inequality $\mathscr{C}_{2p}\leq 2\mathscr{C}_p$ to all pairs $p<q$ to get $\mathscr{C}_{q}/q\leq \mathscr{C}_p/p$? This question is reminiscent of the problems discussed in \cite{BOS18}, one of which was solved in a striking way in \cite{Ku22}. See also the preprint \cite{Ll22} which solves another problem from \cite{BOS18}. The basic question in all these problems is how to circumvent the obstacle that the familiar tools of interpolation theory are unavailable. 

Our partial remedy for this impasse is the topic of Section~\ref{sec:prologue}; it is an explicit integral formula for $|\varphi(0)|^q$ for any extremal function $\varphi$ and any $q>0$. This formula will be the starting point for the proof of Theorem~\ref{thm:korevaar}, similarly to how \eqref{eq:2rpk} was the starting point for the proof that $\mathscr{C}_2=1$. It also clarifies the need for precise information about the zeros of $\varphi$ and thus serves as a motivation for our detailed study of those zeros. 

The three subsections of Section~\ref{sec:prelim} present some preliminary results. Section~\ref{subsec:ortho} deals with the most basic properties of the zeros of extremal functions, including the important orthogonality relations that will be extensively used in subsequent sections. Roughly speaking, this subsection enunciates that the majority of the results of \cite{HB93} for $p=1$ extend to the full range $0<p<\infty$. Section~\ref{subsec:convex} deduces various consequences of convexity in the range $1\leq p < \infty$. Of particular importance is the following counterpart of \eqref{eq:2rpk},
\begin{equation} \label{eq:fat0} f(0) = \int_{-\infty}^\infty f(x) \,\frac{|\varphi_p(x)|^{p-2} \varphi_p(x)}{\|\varphi_p\|_p^p} \,dx \end{equation}
which holds for all $f$ in $PW^p$ when $1 \leq p < \infty$. The fact that $|\varphi_p|^{p-2}\varphi$ only takes values $\pm 1$ 
when $p=1$ played a crucial role in \cite{HB93}, and we will also see it used in Section~\ref{sec:zeroset}. The final Section~\ref{subsec:PS} discusses briefly the Landau--Pollak--Slepian operator of time--frequency concentration \cite{SP61, LP61, LP62}; the numerical value of its largest eigenvalue will be required in the proofs of both Theorem~\ref{thm:zeroset} (c) and Theorem~\ref{thm:sep}. 

Sections~\ref{sec:zeroset} and \ref{subsec:zeroset24}, giving the proofs of Theorem~\ref{thm:zeroset} and Theorem~\ref{thm:sep}, study respectively the geometry of the zeros sets of extremal functions and numerical bounds on the separation of such zero sets in the range $2<p<\infty$. In both cases, the main difficulty is that the zero sets are only indirectly accessible through the orthogonality relations alluded to above. Sections~\ref{sec:zeroset} and \ref{subsec:zeroset24} aim at bringing out as much explicit information as possible from these relations.

Section~\ref{sec:2p4} gives the proof of Theorem~\ref{thm:korevaar}. The idea is to put into effect the numerical bounds of Theorem~\ref{thm:sep} in the formula of Section~\ref{sec:prologue}. In carrying out this idea, we rely on somewhat intricate combinatorial arguments. 

Section~\ref{sec:asymp} presents the proof of Theorem~\ref{thm:Cpinfty} and Theorem~\ref{thm:Cp0}. The example functions
\[ g_{\alpha}(z):=\frac{\Gamma^2(\alpha)}{\Gamma(\alpha-z)\Gamma(\alpha+z)} \]
play a central role in our asymptotic analysis of the lower bounds for $\mathscr{C}_p$, because these functions seem to mimick the extremal functions for suitable choices of $\alpha$. We believe that $g_{1/2+1/p}$ is ``close'' to the extremal function $\varphi_p$ when $p$ is ``close'' to $2$. However, when $p\to \infty $ or $p\to 0^+$, we will see that this connection is more subtle. 

The two final sections of the paper present ideas for further developments. Section~\ref{sec:epilogue}---the epilogue of our paper---revisits two central notions used throughout the preceding sections, namely those of duality and orthogonality. We first prove that \eqref{eq:fat0} extends in a distributional sense beyond the convex regime to $1/2\le p<1$. We then show that there is a natural Hilbert space---more specifically a de Branges space---induced by the orthogonality relations associated with the zeros of $\varphi_p$ for $1<p<\infty$. The basic questions suggested by this section are both related to convexity: How to extend results more generally \emph{beyond} the convex regime, and how to take advantage of the theory of de Branges spaces \emph{in} the strictly convex case? The final Section~\ref{sec:conj} lists a number of conjectures and further open problems, suggested by our analysis and numerical experiments. 

We hope the final two sections of our paper may inspire further work.

\begin{figure}
	\centering
	\begin{tikzcd}
	 & & & 4 \arrow[r] & 8 \arrow[dr, bend left] &    \\
	 1 \arrow[r] & 2 \arrow[r] & 3 \arrow[ur, bend left] \arrow[r] \arrow[dr, bend right] & 5 \arrow[r] & 6 \arrow[r] &  9 \\
	 & & & 7 \arrow[urr, bend right] & & 
	\end{tikzcd}
	\caption{Dependence relations between sections of the present paper.}
	\label{fig:dependence}
\end{figure}
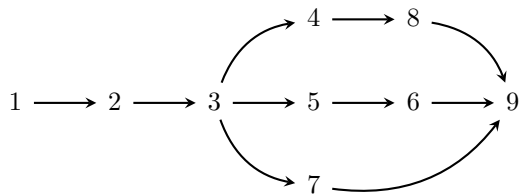

\subsection*{Acknowledgements} We thank the anonymous referee for making us aware of Lemma~\ref{lem:hormander}, which led to an improvement in Theorem~\ref{thm:sepall}~(a).

\section{Prologue: The power trick and its extension to positive powers} \label{sec:prologue} 
We will set the stage for our study by scrutinizing the well known \emph{power trick}, which in our context reads as follows. 
\begin{lemma}\label{lem:powertrick} 
	Fix $0<p<\infty$. If there is a positive integer $k$ such that $q = kp$, then
	\[\frac{\mathscr{C}_q}{q}\leq \frac{\mathscr{C}_p}{p}.\]
\end{lemma}
\begin{proof}
	Consider $f$ in $PW^q$. Clearly, $f^k$ is entire and of exponential type at most $k\pi$. The function $g(z) = f^k(z/k)$ is then of exponential type at most $\pi$. Consequently,
	\[|f(0)|^q = |g(0)|^p \leq \mathscr{C}_p \int_{-\infty}^\infty |g(x)|^p \,dx = k\mathscr{C}_p \int_{-\infty}^\infty |f(x)|^q\,dx = k\mathscr{C}_p \|f\|_q^q.\]
	This implies that $\mathscr{C}_q \leq k \mathscr{C}_p$, which is equivalent to the asserted estimate. 
\end{proof}

We have already seen that $\mathscr{C}_2 = 1$, so the power trick implies that $\mathscr{C}_4 \leq 2$. It turns out that we can do better by a slight twist on the power trick. Indeed, suppose that $f$ is a function in $PW^4$. Then the function $g(z):=f(z/2)f^\ast(z/2)$ is in $PW^2$, where $f^\ast(z):=\overline{f(\overline{z})}$. Replacing $f$ by $g$ in the reproducing formula \eqref{eq:2rpk}, we find that 
\begin{equation}\label{eq:inspiration} 
	|f(0)|^2 = g(0) = \int_{-\infty}^\infty g(x) \, \sinc{\pi x}\,dx = 2 \int_{-\infty}^\infty |f(x)|^2 \,\sinc{2\pi x}\,dx. 
\end{equation}
We now obtain the following result. 
\begin{theorem}\label{thm:PW4} 
	$\displaystyle \mathscr{C}_4 \leq 2- \frac{1}{12}$. 
\end{theorem}
\begin{proof}
	Starting from \eqref{eq:inspiration}, we obtain an upper bound for the right-hand side after replacing $\sinc{2 \pi x}$ by $\max\left(\sinc{2 \pi x},0\right)$. Squaring both sides and using the Cauchy--Schwarz inequality, we find that
	\[|f(0)|^4 \leq 4 \|f\|_4^4 \int_{-\infty}^\infty \max\left(\sinc{2 \pi x},0\right)^2 \,dx.\]
	We estimate
	\begin{align*}
		4 \int_{-\infty}^\infty \max\left(\sinc{2 \pi x},0\right)^2 \,dx &= 2 - 2\sum_{k=0}^\infty \int_{k+1/2}^{k+1} \frac{\sin^2(2\pi x)}{\pi^2 x^2}\,dx \\
		&\leq 2 - 2\sum_{k=0}^\infty \frac{1}{\pi^2 (k+1)^2} \int_{k+1/2}^{k+1} \sin^2(2\pi x)\,dx = 2-\frac{1}{12}
	\end{align*}
	to obtain the desired bound.
\end{proof}
In our context, we may in fact replace $g$ by $f^2(x/2)$ in the above argument, in view of the following observation. 
\begin{lemma}\label{lem:real} 
	Fix $0<p<\infty$. Any solution of \eqref{eq:extremalproblem} is real-valued on $\mathbb{R}$. 
\end{lemma}
\begin{proof}
	Suppose that $\varphi$ is a solution of \eqref{eq:extremalproblem} and define
	\[\psi(z) := \frac{\varphi(z)+\varphi^\ast(z)}{2},\]
	which is in $PW^p$ and satisfies $\psi(0)=1$ since $\varphi(0)=1$. If $x$ is a real number, then clearly $\psi(x) = \mre{\varphi(x)}$. Hence if $\varphi$ is not real-valued on $\mathbb{R}$, then $\|\psi\|_p < \|\varphi\|_p$ which contradicts the assumption that $\varphi$ is an extremal function. 
\end{proof}
We may now interpret what we gained in Theorem~\ref{thm:PW4} as follows. The map $f(x) \mapsto f^2(x/2)$ from the set of real-valued functions in $PW^4$ fails to be onto the set of real-valued functions in $PW^2$. This is quite obvious since for example $\sinc \pi x$ has alternating signs. Another way to understand this is to think of squaring $f$ as ``smoothing'' its Fourier transform, and there is no way we may arrange this so that such a ``smoothed'' Fourier transform equals a characteristic function.

To get access to $\mathscr{C}_p$ in the range $2<p<4$, we may ask if there is a way to make sense of the power trick for non-integer powers of our function. If this can be done, we may again think of powers larger than $1$ as corresponding to suitable ``smoothings'' of the Fourier transform and thus hope for a saving in accordance with the monotonicity conjecture.

Taking non-integer powers of entire functions can in general only be done locally. However, if the zeros are real, then we may make sense of such powers in respectively the lower and upper half-planes. Pursuing the idea of extending the power trick, we are thus led to the following representation formula. (Here and elsewhere we account for multiplicities in the usual way when listing the zeros of a given entire function.)
\begin{theorem}\label{thm:intrep} 
	Fix $0<p<\infty$. Suppose that $f$ is a function in $PW^p$ which does not vanish at the origin and has only real zeros $(t_n)_{n\in\mathbb{Z}\setminus\{0\}}$ ordered such that
	\[ \cdots \leq t_{-2}\leq t_{-1}<0<t_1 \leq t_2 \leq \cdots .\]
	For any $0<q<\infty$,
	\[|f(0)|^q = \sum_{n=0}^\infty \int_{t_{-n-1}}^{t_{-n}} |f(x)|^q \,\frac{\sin{\pi q(x+n)}}{\pi x}\,dx +\sum_{n=0}^\infty \int_{t_n}^{t_{n+1}} |f(x)|^q \,\frac{\sin{\pi q(x-n)}}{\pi x}\,dx,\]
	where we use the convention that $t_0:=0$. 
\end{theorem}

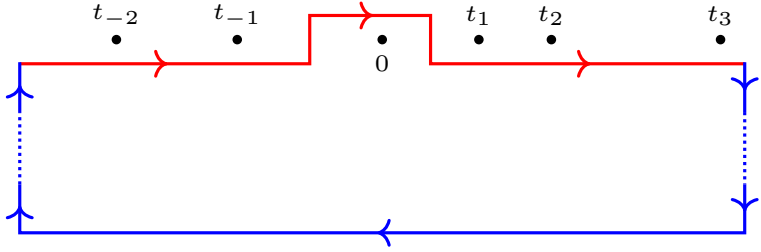
\begin{figure}
	\centering
	\begin{tikzpicture}[scale=1.5]
		\begin{axis}[
			axis equal image,
			axis lines = none,
			xmin = -16, 
			xmax = 16,
			ymin = -15, 
			ymax = 5]
			
			\addplot[thick,solid,color=red] coordinates {(-15,-1) (-3,-1) (-3,1) (2,1) (2,-1) (15,-1)};	
			\addplot[thick,solid,color=blue] coordinates {(15,-0.9355) (15,-3)};	
			\addplot[thick,densely dotted,color=blue] coordinates {(15,-3) (15,-6)};
			\addplot[thick,solid,color=blue] coordinates {(15,-6) (15,-8) (-15,-8) (-15,-6)};
			\addplot[thick,densely dotted,color=blue] coordinates {(-15,-6) (-15,-3)};
			\addplot[thick,solid,color=blue] coordinates {(-15,-3) (-15,-0.9355)};	

			\def\s{0.1} 	
			\draw[thick,solid,color=red,->] (axis cs:-9.01+\s,-1) -- (axis cs:-8.99+\s,-1);
			\draw[thick,solid,color=red,->] (axis cs:-0.51+\s,1) -- (axis cs:-0.49+\s,1);
			\draw[thick,solid,color=red,->] (axis cs:8.49+\s,-1) -- (axis cs:8.51+\s,-1);
			\draw[thick,solid,color=blue,->] (axis cs:15,-1.99-\s) -- (axis cs:15,-2.01-\s);
			\draw[thick,solid,color=blue,->] (axis cs:15,-6.99-\s) -- (axis cs:15,-7.01-\s);
			\draw[thick,solid,color=blue,->] (axis cs:0.1-\s,-8) -- (axis cs:-0.1-\s,-8);
			\draw[thick,solid,color=blue,->] (axis cs:-15,-2.01+\s) -- (axis cs:-15,-1.99+\s);
			\draw[thick,solid,color=blue,->] (axis cs:-15,-7.01+\s) -- (axis cs:-15,-6.99+\s);
			
			\addplot[only marks,mark=*,color=black,mark size=1pt] coordinates {(-11,0) (-6,0) (0,0) (4,0) (7,0) (14,0)};
			\node at (axis cs:-11,1) {\tiny $t_{-2}$};
			\node at (axis cs:-6,1) {\tiny $t_{-1}$};
			\node at (axis cs:0,-1) {\tiny $0$};
			\node at (axis cs:4,1) {\tiny $t_1$};
			\node at (axis cs:7,1) {\tiny $t_2$};
			\node at (axis cs:14,1) {\tiny $t_3$};			
			
		\end{axis}
	\end{tikzpicture}
	\caption{The contour $\Gamma_{T,\varepsilon}$ from the proof of Theorem~\ref{thm:intrep}. The contribution from the {\color{blue} lower part} vanishes as $T\to\infty$, while the {\color{red} upper part} becomes the contour $\Gamma_\varepsilon$ in the same limit.}
	\label{fig:contour}
\end{figure}

\begin{proof}
	Inspecting the stated formula, it is clear that we may assume that $f(0)=1$ without loss of generality. Combined with the assumption that $f$ is $PW^p$, we have (see e.g. \cite[Lecture~17]{Levin96}) the representation 
	\begin{equation}\label{eq:lucrep} 
		f(z) = \lim_{T\to\infty} \prod_{|t_n|<T}\left(1-\frac{z}{t_n}\right). 
	\end{equation}
	For every $q>0$, we may define $f^q$ in the slit plane $\mathbb{C} \setminus \big((-\infty,t_{-1}]\cup [t_1, \infty)\big)$. Let $T$ be a large positive number and $\varepsilon$ a small positive number. Consider the contour $\Gamma_{T,\varepsilon}$ defined as the line segments connecting the points
	\[-T-i \varepsilon, \quad \frac{t_{-1}}{2}-i\varepsilon, \quad \frac{t_{-1}}{2}+i\varepsilon, \quad \frac{t_1}{2}+i\varepsilon, \quad \frac{t_1}{2}-i\varepsilon, \quad T-i\varepsilon, \quad T-iT, \quad -T-iT,\]
	oriented clockwise. See Figure~\ref{fig:contour} for an illustration. From Cauchy's formula, we conclude that
	\[1 = f^q(0) = -\frac{1}{2\pi i} \int_{\Gamma_{T,\varepsilon}} f^q(z) \frac{e^{-q \pi i z}}{z}\,dz.\]
	Our goal is next to show that the contribution of the part of the contour in the lower half-plane vanishes as $T\to\infty$. We begin with the horizontal part of the contour. If $q \geq p$, then $f$ is in $PW^q$ and we estimate
	\[\left|\int_{T-iT}^{-T-i T} f^q(z) \frac{e^{-q \pi i z}}{z}\,dz \right| \leq \frac{e^{-q \pi T}}{T} \int_{-\infty}^\infty |f(x-i T)|^q \, dx \leq \frac{\|f\|_q^q}{T},\]
	where the final inequality is the Plancherel--P\'{o}lya theorem (see e.g. \cite[Lecture~7]{Levin96}). If $q<p$, then we first use H\"older's inequality to the effect that
	\[\left|\int_{T-iT}^{-T-i T} f^q(z) \frac{e^{-q \pi i z}}{z}\,dz \right| \leq \left(e^{-p \pi T} \int_{-\infty}^\infty |f(x-iT)|^p\,dx\right)^{\frac{q}{p}} \left(\int_{-\infty}^\infty \frac{dx}{|x-iT|^r} \right)^{\frac{1}{r}},\]
	where $r = p/(p-q)>1$ so the final term vanishes as $T\to \infty$. The rest of the argument is similar. For the vertical part of the contour, we argue as follows. If $q \geq p$, then $f$ is in $PW^q$ and consequently $g(z) := f^q (z) e^{-q \pi i z}$ is in the Hardy space $H^1$ of the lower half-plane. Since integration along a vertical line is a Carleson measure for the latter space (see e.g.~\cite[Chapter~2]{Garnett07}), we conclude that there is a constant $C>0$ such that
	\[\left|\int_{T-i\varepsilon}^{T-iT} f^q(z) \frac{e^{-q \pi i z}}{z}\,dz\right| \leq \frac{1}{T} \int_0^\infty |g(T-iy)|\,dy \leq \frac{C}{T} \|g\|_1 = \frac{C}{T} \|f\|_q^q.\]
	If $q<p$, then we first use H\"older's inequality as above and argue similarly. We conclude from this that 
	\begin{equation}\label{eq:cauchy2} 
		1 = -\frac{1}{2\pi i } \int_{\Gamma_{\varepsilon}} f^q(z) \frac{e^{-q \pi i z}}{z}\,dz, 
	\end{equation}
	where $\Gamma_\varepsilon$ is the contour obtained from the line segments connecting
	\[-\infty-i \varepsilon, \qquad \frac{t_{-1}}{2}-i\varepsilon, \qquad \frac{t_{-1}}{2}+i\varepsilon, \qquad \frac{t_1}{2}+i\varepsilon, \qquad \frac{t_1}{2}-i\varepsilon, \qquad \infty-i\varepsilon.\]
	A similar argument involving a rectangular contour of integration oriented counter-clockwise in the upper half-plane shows that 
	\begin{equation}\label{eq:cauchy3} 
		0 = \frac{1}{2\pi i} \int_{-\infty+i\varepsilon}^{\infty+i\varepsilon} f^q(z) \frac{e^{q i \pi z}}{z}\,dz. 
	\end{equation}
	Adding \eqref{eq:cauchy2} and \eqref{eq:cauchy3}, we find that 
	\begin{align*}
		1 = \int_{t_{-1}/2+i\varepsilon}^{t_1/2+i\varepsilon} f^q(z) \frac{\sin{\pi q z}}{\pi z}\,dz &+ \left(\int_{-\infty+i\varepsilon}^{t_{-1}/2+i\varepsilon}+\int_{t_1/ 2+i\varepsilon}^{\infty+i \varepsilon} \right) f^q(z) \frac{e^{q i \pi z}}{2 \pi i z}\,dz \\
		&-\left(\int_{-\infty-i \varepsilon}^{t_{-1}/2-i \varepsilon}+\int_{t_1/2-i \varepsilon}^{\infty-i \varepsilon} \right)f^q(z) \frac{e^{-q i \pi z}}{2 \pi i z}\,dz + O(\varepsilon). 
	\end{align*}
	The remainder term $O(\varepsilon)$ accounts for the contribution from integration along the two segments $[t_{-1}/2-i\varepsilon,t_{-1}/2+i\varepsilon]$ and $[t_1/2+i\varepsilon,t_1/2-i\varepsilon]$ in \eqref{eq:cauchy2}. We now wish to take the limit $\varepsilon \to 0^+$. Let us define
	\[f_{\pm}^q(x) := \lim_{\varepsilon \to 0^+} f^q(x\pm i\varepsilon).\]
	If $t_{-1}<x<t_1$, then evidently $f_+^q(x) = f_-^q(x) = f^q(x) = |f(x)|^q$. This means that the total contribution from these $x$ as $\varepsilon\to0^+$ is
	\[\int_{t_{-1}}^{t_1} f^q(x)\, \frac{e^{i\pi qx}-e^{-i\pi qx}}{2\pi i x }\,dx = \int_{t_{-1}}^{t_1} |f(x)|^q\, \frac{\sin{q \pi x}}{\pi x}\,dx.\]
	If $t_n<x<t_{n+1}$ for some $n\geq1$, then $f_{\pm}^q(x) = |f(x)|^q e^{\mp i \pi q n}$ in view of \eqref{eq:lucrep}. Consequently, for these $x$ we get as $\varepsilon\to0^+$ the contribution
	\[\int_{t_n}^{t_{n+1}} f_+^q(x) \, \frac{e^{q i \pi x}}{2i \pi x}\,dx - \int_{t_{n}}^{t_{n+1}} f_-^q(x) \,\frac{e^{-q i \pi x}}{2i \pi x}\,dx= \int_{t_n}^{t_{n+1}} |f(x)|^q \,\frac{\sin{q\pi (x-n)}}{\pi x}\,dx.\]
	The case when $t_{-n-1}<x<t_{-n}$ for some $n\geq1$ is similar. 
\end{proof}
We will see in the next section that the zeros of any solution of \eqref{eq:extremalproblem} are real, so Theorem~\ref{thm:intrep} could possibly replace the power trick. Indeed, setting $q=p/2$ in the theorem, we could hope to proceed in a similar way as in the proof of Theorem~\ref{thm:PW4}, using the Cauchy--Schwarz inequality and accounting suitably for both the size and the sign of the kernel that $|f|^q$ is integrated against. Note that when $q=2$, the location of the zeros plays no role, and we are back to \eqref{eq:inspiration}.

To succeed with this approach and thus prove Theorem~\ref{thm:korevaar}, we need to have quite detailed information about the location of the zeros of the extremal functions. A good part of our paper will therefore study the zero sets of those functions, and we will be particularly interested in precise results in the range $2<p\leq 4$. The reader may notice that if it were known that the zeros lie in suitable neighborhoods of the nonzero integers, then Theorem~\ref{thm:korevaar} would be a trivial consequence of Theorem~\ref{thm:intrep}. While this kind of location is likely to be true, our results are quite far from verifying that. Fortunately, however, Theorem~\ref{thm:sep} is sufficiently strong, though barely so, for the above proof idea to work. 

The representation formula of Theorem~\ref{thm:intrep} was, as just described, established to prove Theorem~\ref{thm:korevaar} by a non-integer version of the power trick. We will see, however, that it will also be a convenient tool in the study of the asymptotic behavior of $\mathscr{C}_p$ when $p\to 0^+$. As to the asymptotics when $p\to \infty$, we observe that the integrals in Theorem~\ref{thm:intrep} become increasingly oscillating as $q=p/2$ gets large, in accordance with the ``phase shift'' in the order of $\mathscr{C}_p/p$ exhibited by Theorem~\ref{thm:Cpinfty}.

\section{Preliminaries} \label{sec:prelim} 
This section compiles some preliminary results, most of which are already known. In Section~\ref{subsec:ortho} we work in the full range $0<p<\infty$, while in Section~\ref{subsec:convex}, we restrict our attention to the convex range $1 \leq p < \infty$. Section~\ref{subsec:PS} is concerned with the Landau--Pollak--Slepian problem, where the relevant theory has only been developed for $p=2$.

\subsection{Zeros and associated orthogonality relations} \label{subsec:ortho} What follows is largely inspired by the variational arguments used in the proof of \cite[Theorem~2.6]{HB93}. Here and in what follows, we say that $f$ is a \emph{real} entire function if $f$ is real-valued on $\mathbb{R}$. Similarly, we say that a rational function is \emph{real} if it is the ratio of two real polynomials. 
\begin{lemma}\label{lem:variational} 
	Fix $0<p<\infty$ and let $\varphi$ be a solution of \eqref{eq:extremalproblem}. If $r=r_1/r_2$ is a real rational function such that $\deg(r_1) \leq \deg(r_2)$, $r(0)=0$, $\varphi r$ is an entire function, and $|\varphi|^pr$ is integrable, then
	\begin{equation} \label{eq:zerocond} \int_{-\infty}^\infty |\varphi(x)|^p r(x) \, dx = 0. \end{equation}
\end{lemma}
\begin{proof}
	For every real number $\varepsilon$, the function $\psi (z) := \varphi(z) + \varepsilon \varphi(z) r(z)$ belongs to $PW^p$ and satisfies $\psi(0)=1$. If we set
	\[F(\varepsilon) := \|\psi\|_p^p = \int_{-\infty}^\infty \left|\varphi(x)+ \varepsilon \varphi(x) r(x)\right|^p \,dx,\]
	then the assumption that $\varphi$ is a solution of the extremal problem implies that $F(\varepsilon) \geq F(0)$ for all $\varepsilon$. We will show that this can only hold if \eqref{eq:zerocond} is true. 
	
	Let $x_1,x_2,\ldots, x_k$ be the real poles of $r$ with multiplicities $m_1,m_2,\ldots,m_k$. Then the zero of $\varphi$ at $x_j$ has multiplicity $n_j\geq m_j$, and we also have
	\begin{equation} \label{eq:multip} 
		pn_j>m_j-1
	\end{equation} 
	since $|\varphi|^p r$ is assumed to be integrable. Fix a small number $\delta>0$ and set 
	\[ E_j\coloneqq \big\{x\,:\, \ |x-x_j|\leq \delta |\varepsilon|^{1/m_j}\big\}\]
	for $j=1,2,\ldots,k$, and let $E\coloneqq E_1 \cup E_2 \cup \cdots \cup E_k$. Then 
	\[ \int_{\mathbb R \setminus E} \left|\varphi(x)+ \varepsilon \varphi(x) r(x)\right|^p \,dx = \|\varphi\|_p^p+ \left( \int_{-\infty}^{\infty} |\varphi(x)|^pr(x) \, dx +O(\delta^m)\right)\varepsilon+o(|\varepsilon|).\]
	On the other hand, we find that 
	\[\int_{E_j} \left|\varphi(x)+ \varepsilon \varphi(x) r(x)\right|^p \,dx=O\big(|\varepsilon|^{p+\frac{1}{m_j}(p(n_j-m_j)+1)}\big)=O(|\varepsilon|^{1+\eta_j})\]
	with $\eta_j\coloneqq  (pn_j+1)/m_j-1>0$ by \eqref{eq:multip}. Fixing a sufficiently small $\delta$, we see that we may obtain
$F(\varepsilon)<F(0)$ for some small $\varepsilon$ should 
\[ \int_{-\infty}^\infty |\varphi(x)|^p r(x) \, dx \neq 0. \qedhere\]
\end{proof}

We have two applications of Lemma~\ref{lem:variational}. The first reads as follows.

\begin{lemma}\label{lem:zeros} 
	Fix $0<p<\infty$. The zeros of a solution of \eqref{eq:extremalproblem} are real. Moreover,
	\begin{enumerate}
		\item[(a)] if $p\geq 1/2$, then the zeros are simple;
		\item[(b)] if $p<1/2$, then the zeros have multiplicity at most $1/p$.
	\end{enumerate}
\end{lemma}
\begin{proof}
	Let $\varphi$ be a solution of \eqref{eq:extremalproblem}. It follows from Lemma~\ref{lem:real} that if $a+bi$ is a zero of $\varphi$ and $b \neq0$, then $a-bi$ is also a zero of $\varphi$. Pick $\varepsilon>0$ so small that
	\[(x-a)^2 + b^2 > (x-a)^2 +b^2 - \varepsilon x^2 > 0\]
	for every real number $x\neq0$. The function
	\[\psi(z) := \frac{(z-a)^2 +b^2-\varepsilon z^2}{(z-a)^2 + b^2} \varphi(z)\]
	is in $PW^p$ and satisfies $\psi(0)=1$. However, since $|\psi(x)|<|\varphi(x)|$ for every real number $x\neq0$ by our choice of $\varepsilon$, this contradicts the assumption that $\varphi$ is an extremal function. Consequently, all zeros of $\varphi$ must be real.
	Suppose that $p>1/2$ and that $t$ is a real zero of $\varphi$ of order $2$ or more. Since $\varphi(0)=1$, we can exclude the possibility that $t=0$. We apply Lemma~\ref{lem:variational} to the function $r(z) = \frac{z^2}{(z-t)^2}$, where $|\varphi|^p r$ is integrable since $p>1/2$, and we obtain that
	\[ \int_{-\infty}^\infty \frac{|\varphi(x)|^p x^2}{(x-t)^2}\, dx = 0, \]
	which yields a contradiction. The same proof works if $0<p<1/2$ and the multiplicity of the zero of $\varphi$ at $t$ is strictly larger than $1/p$ because in this case $|\varphi|^p r$ is again integrable.

	Next we deal with the case $p=1/2$. We can no longer apply Lemma~\ref{lem:variational} as above, since $|\varphi|^{1/2} r$ may not be integrable. Suppose nevertheless that $\varphi$ has a double zero at $t$ and introduce
	\[\varphi_\varepsilon(x) := \frac{\big(1-\frac{x}{t-\varepsilon}\big)\big(1-\frac{x}{t+\varepsilon}\big)}{(1-x/t)^2} \varphi(x)\]
	for $0<\varepsilon<1$. Our plan is to show that there exists a positive constant $c$ such that
	\[\|\varphi_{\varepsilon}\|_{1/2}^{1/2}=\|\varphi\|_{1/2}^{1/2}-c\varepsilon\log\frac{1}{\varepsilon}+O(\varepsilon)\]
	as $\varepsilon \to 0^+$, contradicting the assumption that $\varphi$ solves the extremal problem. We set $x=t+\xi$ and see that
	\begin{equation} \label{eq:perturb} 
		|\varphi_\varepsilon(t+\xi)|=\frac{|\xi^2-\varepsilon^2|}{\xi^2|1- \varepsilon^2/t^2|}|\varphi(t+\xi)|. 
	\end{equation}
	This entails that
	\[ \int_{|\xi|>1} |\varphi_{\varepsilon}(t+\xi)|^{1/2}\, dx= \int_{|\xi|>1} |\varphi (t+\xi)|^{1/2}\, dx+O(\varepsilon^2). \]
	Since $\varphi(t+\xi)=O(\xi^2)$ for small $\xi$, we also find that both 
	\[ \int_{|\xi|<\varepsilon} |\varphi_{\varepsilon}(t+\xi)|^{1/2} \, dx=O(\varepsilon^2) \qquad \text{and} \qquad  \int_{|\xi|<\varepsilon} |\varphi(t+\xi)|^{1/2}\, dx=O(\varepsilon^2). \]
	Since $\varphi(t+\xi)=a\xi^2 + O(\xi^3)$, it is however clear from \eqref{eq:perturb} that 
	\[\int_{\varepsilon\leq |\xi|\leq 1} |\varphi_{\varepsilon}(t+\xi)|^{1/2}\, dx= \int_{\varepsilon\leq |\xi|\leq 1} |\varphi(t+\xi)|^{1/2}\, dx - c \varepsilon \log \frac{1}{\varepsilon}+O(\varepsilon). \qedhere\]
\end{proof}

The second application of Lemma~\ref{lem:variational} is the following orthogonality relations between the zeros of the extremal functions, which will be the workhorse of Sections~\ref{sec:zeroset}, \ref{subsec:zeroset24}, and \ref{sec:epilogue}. 
\begin{lemma}\label{lem:orthogonality} 
	Fix $0<p<\infty$, and let $\varphi$ be a solution of \eqref{eq:extremalproblem}. 
	\begin{enumerate}
		\item[(a)] If $t$ is a zero of $\varphi$, then
		\[\int_{-\infty}^\infty \frac{|\varphi(x)|^p x}{x-t}\,dx = 0.\]
		\item[(b)] If $t$ and $s$ are distinct zeros of $\varphi$, then
		\[\int_{-\infty}^\infty \frac{|\varphi(x)|^p x^2}{(x-t)(x-s)}\,dx = 0.\]
	\end{enumerate}
\end{lemma}
\begin{proof}
	The zeros are real by Lemma~\ref{lem:zeros}. To prove (a), we apply Lemma~\ref{lem:variational} with $r(z) = \frac{z}{z-t}$, and to establish (b), we use Lemma~\ref{lem:variational} with $r(z) = \frac{z^2}{(z-t)(z-s)}$. 
\end{proof}

We will also use the following estimate, which is due to H\"ormander \cite[p.~26]{Hormander55}.
\begin{lemma}\label{lem:hormander}
	Suppose that $f$ is a real function in $PW^\infty$ and that $|f(\xi)|=\|f\|_\infty$. Then
	\[|f(x)| \geq \|f\|_\infty \cos{\pi (x-\xi)}\]
	for $|x-\xi|\leq 1/2$.
\end{lemma}

\subsection{Consequences of convexity} \label{subsec:convex} The convexity of $L^p(\mathbb{R})$ for $1 \leq p < \infty$ can be employed to obtain additional information about the solutions of \eqref{eq:extremalproblem} compared to what we obtained for the full range $0<p<\infty$ above. 

The strict convexity of $L^p(\mathbb{R})$ for $1<p<\infty$ implies that the extremal problem \eqref{eq:extremalproblem} has a unique solution. It is also known that \eqref{eq:extremalproblem} has a unique solution for $p=1$. We refer to \cite[Lemma~6.1]{LL15} for a clean proof of this fact. A slightly different proof can be found in \cite[Section~3.1]{CMS19}. The same argument can also be extracted from the proof of \cite[Theorem~2.6]{HB93}.

It is plain that if $\varphi$ is a solution of \eqref{eq:extremalproblem}, then so is $\psi(z) = \varphi(-z)$. Hence the uniqueness of the extremal function forces it to be even. Consequently, we obtain the following result from Lemma~\ref{lem:zeros}~(a) and the canonical product representation of functions in Paley--Wiener spaces (see e.g.~\cite[Lecture~17]{Levin96}). 
\begin{lemma}\label{lem:unique} 
	Fix $1 \leq p < \infty$. The extremal problem \eqref{eq:extremalproblem} has a unique solution of the form
	\[\varphi_p(z) = \prod_{n=1}^\infty \left(1-\frac{z^2}{t_n^2}\right),\]
	where $(t_n)_{n\geq1}$ is a strictly increasing sequence of positive numbers. 
\end{lemma}

In the range $1 \leq p < \infty$, we have in Lemma~\ref{lem:unique} adopted the notation $\varphi_p$ for the unique solution of \eqref{eq:extremalproblem}. Before stating our next result, we set 
\begin{equation}\label{eq:holderfriend} 
	\Phi(x) := \frac{|\varphi(x)|^{p-2}}{\|\varphi\|_p^p} \varphi(x) 
\end{equation}
for any nontrivial function $\varphi$ in $PW^p$. Another variational argument shows that the unique extremal function gives rise to a reproducing kernel formula for $PW^p$. In the case $p=1$, the same argument may be found in \cite[Section~3.2]{CMS19} or in \cite[Section~2]{HB93}.
\begin{theorem}\label{thm:prpk} 
	If $1 \leq p < \infty$, then $\varphi = \varphi_p$ if and only if 
	\begin{equation}\label{eq:prpk} 
		f(0) = \int_{-\infty}^\infty f(x)\, \Phi(x)\,dx 
	\end{equation}
	holds for every $f$ in $PW^p$. 
\end{theorem}
\begin{proof}
	Suppose first that \eqref{eq:prpk} holds for every $f$ in $PW^p$. If $f(0)=1$, then H\"{o}lder's inequality applied to \eqref{eq:prpk} yields $\|\varphi\|_p \leq \|f\|_p$. This shows that $\varphi=\varphi_p$ by Lemma~\ref{lem:unique}. Conversely, suppose that $\varphi=\varphi_p$. Let $g$ be a real function in $PW^p$ which satisfies $g(0)=0$. Mimicking the proof of Lemma~\ref{lem:variational}, we begin by setting $\psi := \varphi_p + \varepsilon g$ and
	\[G(\varepsilon) := \|\psi\|_p^p = \int_{-\infty}^\infty |\varphi_p(x)+\varepsilon g(x)|^p\,dx.\]
	Since $1 \leq p < \infty$, we may differentiate under the integral sign without worrying about the zeros of $\varphi_p$ (which we know are simple by Lemma~\ref{lem:zeros}~(a)). The assumption that $\varphi_p$ is an extremal function shows that 
	\begin{equation}\label{eq:G0} 
		0 = G'(0) = \int_{-\infty}^\infty |\varphi_p(x)|^{p-2} \varphi_p(x) g(x)\,dx. 
	\end{equation}
	For $f$ in $PW^p$, we set $g(z):=f(z)-\varphi_p(z) f(0)$ and write
	\[\int_{-\infty}^\infty |\varphi_p(x)|^{p-2} \varphi_p(x) f(x)\,dx = \int_{-\infty}^\infty |\varphi_p(x)|^{p-2} \varphi_p(x) g(x)\,dx + \int_{-\infty}^\infty |\varphi_p(x)|^p f(0)\,dx .\]
	Since $\varphi_p(0)=1$ and since $\varphi_p$ is real by Lemma~\ref{lem:real}, we see that \eqref{eq:G0} implies \eqref{eq:prpk} under the additional assumption that $f$ is real. This additional assumption can be removed by the linearity of the integral and the fact that every $f$ in $PW^p$ can be written as $f = g+ih$ for functions $g$ and $h$ in $PW^p$ that are real. 
\end{proof}

We will discuss a further interesting consequence of convexity in Section~\ref{subsec:deBranges}, where the following simple consequence of Theorem~\ref{thm:prpk} will have a role to play.
\begin{corollary}\label{cor:x2infty} 
	If $1 \leq p < \infty$, then 
	\begin{equation}\label{eq:growth} 
		\int_{-\infty}^{\infty} x^2|\varphi_p(x)|^p\,dx =\infty. 
	\end{equation}
\end{corollary}
\begin{proof}
	We prove this by first noting that if $0<\varepsilon<1/2$, then
	\[ \int_{-\infty}^{\infty} x^2 \sinc^2(\pi \varepsilon x) \varphi_p((1-2\varepsilon)x) |\varphi_p(x)|^{p-2}\varphi_p(x)\, dx=0 \]
	by Theorem~\ref{thm:prpk}. If the integral on the left-hand side of \eqref{eq:growth} is finite, then we may let $\varepsilon \to 0^+$ and reach the absurd conclusion that this integral vanishes. 
\end{proof}
Let us sketch a different proof of Theorem~\ref{thm:prpk}. Consider the bounded linear functional $\mathscr{L}$ defined on $PW^p$ by $\mathscr{L}(f) = f(0)$. By the Hahn--Banach theorem and the Riesz representation theorem, we deduce that there is some $\Psi$ in $L^q(\mathbb{R})$ such that $\mathscr{L}(f)= \langle f, \Psi \rangle$. By Lemma~\ref{lem:unique}, we know that $\|\mathscr{L}\| = \|\varphi_p\|_p^{-1}$. It follows from this and H\"older's inequality that $\Psi = \Phi_p$ (in the case $p=1$, we use that $\varphi_1$ vanishes on a set of measure $0$). This means that the bounded linear functional generated by $\Phi_p$ on $L^p(\mathbb{R})$ when restricted to the subspace $PW^p$ yields the reproducing kernel formula at the origin. In Section~\ref{subsec:traces}, we will explore what traces of this duality remain in the range $p<1$.

In the special case $p=2$, we may use Theorem~\ref{thm:prpk} to solve the extremal problem \eqref{eq:extremalproblem} without knowledge of the reproducing kernel formula \eqref{eq:2rpk}. We will give yet another proof based on Theorem~\ref{thm:intrep} in Example~\ref{ex:2} below. Note that Corollary~\ref{cor:C2} in combination with \eqref{eq:holderfriend} and Theorem~\ref{thm:prpk} recovers the reproducing kernel formula \eqref{eq:2rpk}, which we have then established by a variational argument.
\begin{corollary}\label{cor:C2} 
	$\mathscr{C}_2=1$, and the unique extremal function is $\varphi_2(z)=\sinc{\pi z}$. 
\end{corollary}
\begin{proof}
	By Lemma~\ref{lem:unique} we know that there is a strictly increasing sequence of positive numbers $(t_n)_{n\geq1}$ such that
	\[\varphi_2(z) = \prod_{n=1}^\infty \left(1-\frac{z^2}{t_n^2}\right).\]
	Suppose that $g$ is a function in $PW^2$ such that $g(t_1)=0$ and apply Theorem~\ref{thm:prpk} to the function $g(z) \frac{z}{z-t_1}$. Multiplying both sides by $\|\varphi_2\|_2^2$, we infer that 
	\begin{equation}\label{eq:phi2int} 
		0 = \int_{-\infty}^\infty g(x) \frac{x}{x-t_1} \, \varphi_2(x)\,dx. 
	\end{equation}
	Let $f$ be any function in $PW^2$ and define
	\[g(z) := f(z) - \frac{f(t_1)}{t_1 \varphi_2'(t_1)} \frac{x \varphi_2(x)}{x-t_1},\]
	where $\varphi_2'(t_1)\neq0$ by Lemma~\ref{lem:zeros}~(a). Since $\varphi_2(t_1)=0$, we find that $g(t_1)=0$. It therefore follows from \eqref{eq:phi2int} that
	\[0 = \int_{-\infty}^\infty f(x) \frac{x}{x-t_1} \, \varphi_2(x)\,dx - \frac{f(t_1)}{t_1 \varphi_2'(t_1)} \int_{-\infty}^\infty \left| \frac{x}{x-t_1}\varphi_2(x)\right|^2\,dx,\]
	which we rewrite as 
	\begin{equation}\label{eq:rpklambda1} 
		f(t_1) = \int_{-\infty}^\infty f(x) \frac{\psi_2(x)}{\|\psi_2\|_2^2}\,dx, 
	\end{equation}
	where $\psi_2(z) = \frac{1}{t_1 \varphi_2'(t_1)} \frac{z}{z-t_1} \varphi_2(z)$. Since $f$ is any function in $PW^2$, we conclude from Theorem~\ref{thm:prpk} and \eqref{eq:rpklambda1} that
	\[\frac{\psi_2(z)}{\|\psi_2\|_2^2} = \frac{\varphi_2(z-t_1)}{\|\varphi_2\|_2^2} \qquad \iff \qquad z \varphi_2(z) = C (z-t_1) \varphi_2(z-t_1),\]
	where $C$ is a nonzero constant. Since these two entire functions are equal, their zero set must be $(nt_1)_{n\in \mathbb{Z}}$ and consequently $\varphi_2(z) = \sinc{\pi z /t_1}$. Since the extremal function is of exponential type $\pi$ and $t_1>0$, we must have $t_1=1$. We conclude that $\mathscr{C}_2 = 1$ by the well known fact that the $L^2$ integral of $\sinc{\pi x}$ is $1$. 
\end{proof}

It is of interest to interpret Theorem~\ref{thm:prpk} in terms of the Fourier transform of $|\varphi|^{p-2}\varphi$. 
To this end, we begin by fixing the following normalization: 
\[\widehat{f}(\xi) = \int_{-\infty}^\infty f(x) \,e^{-ix\xi}\,dx \qquad \text{and} \qquad f(x) = \int_{-\infty}^\infty \widehat{f}(\xi) \,e^{i \xi x}\,\frac{d\xi}{2\pi}.\]
By the $L^p$ version of the Paley--Wiener theorem (see e.~g. \cite[Theorem~4]{NAB14}), we have another formula for the functional of point evaluation in $PW^p$, namely 
\begin{equation}\label{eq:PWpoint} 
	f(0) = \int_{-\pi}^\pi \widehat{f}(\xi)\,\frac{d\xi}{2\pi}. 
\end{equation}
If $1 \leq p \leq 2$, then it follows from the Hausdorff--Young inequality that in fact $\widehat{f}$ is a function in $L^q([-\pi,\pi])$. In the range $2<p<\infty$, we interpret $\widehat{f}$ in \eqref{eq:PWpoint} as a tempered distribution which is supported on $[-\pi,\pi]$. If we define $\Phi$ as in \eqref{eq:holderfriend}, a similar argument shows that $\widehat{\Phi}$ is in $L^p(\mathbb{R})$ if $\varphi$ is in $PW^p$ for $2 \leq p < \infty$, while in the range $1 \leq p < 2$ we can only interpret $\widehat{\Phi}$ as a tempered distribution. 

At any rate, we obtain the following result from Theorem~\ref{thm:prpk}, \eqref{eq:PWpoint}, and an approximation argument.
\begin{corollary}\label{cor:otherside} 
	If $1 \leq p < \infty$, then $\varphi = \varphi_p$ if and only if
	\[\widehat{\Phi} \chi_{(-\pi,\pi)} = \chi_{(-\pi,\pi)}\]
	in the sense of distributions and where $\Phi$ is as in \eqref{eq:holderfriend}.
\end{corollary}

By the convolution theorem for Fourier transforms, the condition of Corollary~\ref{cor:otherside} can be reformulated as $\Phi \ast \varphi_2 = \varphi_2$, where $\varphi_2(z) = \sinc{\pi z}$. In the range $1<p<\infty$, this reformulation makes sense pointwise. The fact that $\varphi_2$ is not in $PW^1$, means that for $p=1$, the reformulation must also be interpreted in the sense of distributions. This means that if $p=1$ in the following result, then $\|\varphi_2-\Phi \ast \varphi_2\|_\infty$ should be interpreted as the norm of the (possibly unbounded) linear functional on $\mathscr{S}\cap L^1(\mathbb{R})$, where $\mathscr{S}$ denotes the Schwartz class.
\begin{theorem}\label{thm:prehormander} 
	Fix $1 \leq p < \infty$. If $\varphi$ is a function in $PW^p$ with $\varphi(0)=1$, then 
	\begin{equation}\label{eq:prehormander} 
		\frac{1}{\|\varphi\|_p} \leq \mathscr{C}_p^{1/p} \leq \frac{1}{\|\varphi\|_p} + \|\varphi_2-\Phi \ast \varphi_2\|_q, 
	\end{equation}
	where $\Phi$ is as in \eqref{eq:holderfriend} and $p^{-1}+q^{-1}=1$. 
\end{theorem}
\begin{proof}
	The lower bound in \eqref{eq:prehormander} is trivial. For the upper bound, consider the function $F := \varphi_2 + (\Phi-\Phi\ast \varphi_2)$ and let $f$ be a function in $\mathscr{S} \cap PW^2$. Computing using \eqref{eq:2rpk}, we find that $F \ast f = f$. By H\"older's inequality and density, it follows that
	\[\mathscr{C}_p^{1/p} \leq \|F\|_q \leq \|\Phi\|_q + \|\varphi_2-\Phi\ast\varphi_2\|_q = \frac{1}{\|\varphi\|_p} + \|\varphi_2-\Phi\ast\varphi_2\|_q. \qedhere\]
\end{proof}
The above result should be compared to the second part of \cite[Theorem~2.6]{HB93}. Indeed, if we add the assumption that
the Fourier transform of $\Phi$ restricted to $(-\pi,\pi)$ be in $L^p(-\pi,\pi)$, then we get from \eqref{eq:prehormander} to  
\begin{equation} \label{eq:hormHY} \frac{1}{\|\varphi\|_p} \leq \mathscr{C}_p^{1/p} \leq \frac{1}{\|\varphi\|_p} + \left(\int_{-\pi}^\pi \big|1-\widehat{\Phi}(\xi)\big|^p\,\frac{d\xi}{2\pi}\right)^{\frac{1}{p}}\end{equation}
by the Hausdorff--Young inequality when $1 \leq p \leq 2$. Weaker than \eqref{eq:prehormander}, \eqref{eq:hormHY} is more manageable from a numerical point view as we integrate over a finite interval. Since $|\Phi|$ is constant and $\Phi$ changes sign at the zeros of $\varphi$ when $p=1$, we may compute $\Phi$ and thus $\widehat{\Phi}$ once we know the zeros of a given function $\varphi$. This fact enabled H\"{o}rmander and Bernhardsson to use \eqref{eq:hormHY} rather than \eqref{eq:prehormander} to bound $\mathscr{C}_1$ from above.

If $p$ is a positive even integer, then we can compute $\widehat{\Phi}$ by the convolution theorem (see Section~\ref{subsec:conv}). However, in this case we cannot use the Hausdorff--Young inequality and Theorem~\ref{thm:prehormander} to obtain an upper bound for $\mathscr{C}_p$.

\subsection{Eigenvalues of prolate spheroidal wave functions} \label{subsec:PS} In a series of seminal papers, Landau, Pollak, and Slepian studied in the early 1960s the eigenvectors and eigenvalues of the time--frequency concentration operator. We will only need certain specific values of these eigenvalues, but we would like to place our application in context by recapitulating the basic results of \cite{SP61, LP61, LP62}. 

The concentration operator is defined on $PW^2_\Omega$, which consists of the $L^2$ functions with Fourier transform supported on $[-\Omega,\Omega]$. We define two operators, the band-limiting operator $\mathbf{B}_\Omega$ which is the orthogonal projection from $L^2(\mathbb{R})$ to $PW^2_\Omega$ and the time-limiting operator $\mathbf{D}_T$ which is the orthogonal projection of functions in $L^2(\mathbb{R})$ to the subspace of functions supported in $[-T/2,T/2]$. Finally, the concentration operator $\mathbf{C}_{\Omega, T} \colon PW^2_\Omega \to PW^2_\Omega$ is the composition $\mathbf{C}_{\Omega, T}:=\mathbf{B}_{\Omega} \mathbf{D}_T$. 

Landau, Pollak, and Slepian studied the eigenfunctions and eigenvalues of $\mathbf{C}_{\Omega, T}$ with a view to applications in signal processing. The concentration operator $\mathbf{C}_{\Omega, T}$ is compact and self-adjoint, and by the spectral theorem, $PW^2_\Omega$ has an orthonormal basis of eigenfunctions with real eigenvalues. Clearly, since $\mathbf{C}_{\Omega, T}$ is a composition of two projection operators, all its eigenvalues lie between $0$ and $1$. The eigenvalues are simple and denoted by $\lambda_0, \lambda_1, \ldots$ and ordered by descending magnitude so that $1>\lambda_0 > \lambda_1> \cdots$. What is important for us and in many applications, is that the largest eigenvalue $\lambda_0$ satisfies the extremal property 
\begin{equation} \label{eq:exLPS}
	\lambda_0 = \sup_{f\in PW^2_\Omega} \left\{\int_{-T/2}^{T/2} |f(x)|^2\, dx\,:\, \int_{-\infty}^\infty |f(x)|^2\,dx = 1\right\}
\end{equation}
and the unimodular scalar multiples of the corresponding eigenfunction $\psi_0$ are the only functions that achieve this supremum. The eigenvalues depend only on the product of $T$ and $\Omega$, and the same holds for the eigenfunctions, up to a trivial scaling. It is therefore customary to introduce the parameter $c := \Omega T/2$ and write $\lambda_j = \lambda_j(c)$ and $\psi_j(t) = \psi_j(c,t)$ for the eigenfunctions associated with $\Omega=c$ and $T=2$.
 
Specifically, we will apply this as follows: If $f$ is a function in $PW^2 = PW^2_\pi$, then
\[\int_{I} |f(x)|^2 \,dx \leq \lambda_0\left(\frac{\pi |I|}{2}\right) \|f\|_2^2.\]
It is therefore important for us to have good estimates for $\lambda_0$. A remarkable fact proved by Landau, Pollak, and Slepian is that the concentration operator $\mathbf{C}_{\Omega, T}$ commutes with the differential operator 
\[\mathbf{L}(x) := \frac{d}{dx}(1 -x^2)\frac{d}{dx} -c x^2.\]
This ``lucky accident'' shows that the eigenfunctions of the concentration operator are also eigenfunctions of the corresponding Sturm--Liouville operator associated with $\mathbf{L}$. This helps identifying them because this equation arises in the study of the three-dimensional wave equation in prolate spheroidal coordinates, and there is an abundant literature on them (see e.g.~\cite{ORX13}). The eigenfunctions are the prolate spheroidal angular functions, and the corresponding eigenvalues are given by the prolate spheroidal radial function $R_{0,n}$. In particular,
\[\lambda_0(c) = \frac{2\pi}{c} R_{0,0}^2(c,1).\]
The prolate spheroidal functions $R_{0,n}$ can be computed by an expansion in Bessel functions. We have used a Fortran implementation of the algorithm by Zhang and Jin in \cite[p.~536]{ZhangJin} to compute the values in Table~\ref{table:lambda0} that we require in our analysis of $\mathscr{C}_p$. We have also verified the table using Mathematica.

\begin{table}
	\begin{center}
		\begin{tabular}
			{c|c} $c$ & $\lambda_0(c)$ \\
			\hline $2\pi/3$ & $0.896107188059$ \\
			$2$ & $0.880559922317$ \\
			$3\pi/5$ & $0.858990907475$ \\
			$1.080420803046 \pi/4$ & $0.500000000028$ 
		\end{tabular}
	\end{center}
	\caption{Principal eigenvalue of $\mathbf{C}_{c,2}$.}
	\label{table:lambda0} 
\end{table}

\section{Geometric properties of the zero sets of extremal functions} \label{sec:zeroset} 
The purpose of this section is to establish Theorem~\ref{thm:zeroset}. We will enumerate the zero set $\mathscr{Z}(\varphi)=(t_n)_{n\in\mathbb{Z}\setminus\{0\}}$ according to the convention that
\[ \cdots < t_{-2} < t_{-1} < 0 < t_1 < t_2 < \cdots .\]
The proof splits naturally into two parts that will be treated in the subsequent subsections.

\subsection{Proof of Theorem~\ref{thm:zeroset}~(a) and~(b)} Recall that the first assertion, namely that the zeros of any solution of \eqref{eq:extremalproblem} are real, has already been established in Lemma~\ref{lem:zeros} above. The parts (a) and (b) have similar proofs, but the latter is somewhat more refined (and relies on the former). 

\begin{proof}[Proof of Theorem~\ref{thm:zeroset} (a)]
	By symmetry, it is sufficient to consider zeros $t$ satisfying $t\geq1$. Consider $k \geq \max(2,4/p)$ consecutive zeros repeated according to their 
multiplicity, say
	\[1 \leq t_n \leq t_{n+1}\leq \cdots \leq t_{n+k-1}.\]
	We may assume that $t_{n+k-1} - t_n \leq 1/3$, because otherwise we would be done and $t_{n+k-1} > t_n$ because, by Lemma~\ref{lem:zeros}~(a) and (b), the multiplicity of $t_n$ is bounded above by $1/p$. Using Lemma~\ref{lem:orthogonality}~(b) with $t = t_n$ and $s=t_{n+k-1}$, we find that 
	\begin{equation}\label{eq:interval} 
		\int_{-\infty}^\infty \frac{|\varphi(x)|^p x^2}{(x-t_n)(x-t_{n+k-1})}\,dx = 0. 
	\end{equation}
	Setting
	\[\psi(x) := \varphi(x) \prod_{j=0}^{k-1} \frac{1}{x-t_{n+j}} \quad \text{and} \quad \Psi(x) := \frac{|\psi(x)|^p \, x^2}{|(x-t_n)(x-t_{n+k-1})|}\prod_{j=0}^{k-1} |x-t_{n+j}|^p,\]
	we rewrite \eqref{eq:interval} as 
	\begin{equation}\label{eq:equ} 
		\int_{\mathbb{R}\setminus I} \Psi(x)\,dx = \int_I \Psi(x)\,dx, 
	\end{equation}
	where $I:=[t_n,t_{n+k-1}]$. Estimating the right-hand side of \eqref{eq:equ} crudely, we find that 
	\begin{align}\label{eq:11:30} 
		\begin{split}
			\int_I \Psi(x)\,dx & \leq \|\psi\|_\infty^p t_{n+k-1}^2 |I|^{p(k-2)} \int_I (x-t_n)^{p-1} (t_{n+k-1}-x)^{p-1}\,dx \\
			& \qquad\qquad\qquad\qquad\qquad\qquad\qquad= \|\psi\|_\infty^p t_{n+k-1}^2 |I|^{pk-1} \B(p,p). 
		\end{split}
	\end{align}
	On the other hand, if $\xi$ is a point at which $|\psi|$ attains its maximum, then
	\[ |\psi(x)| \geq\frac{\| \psi \|_{\infty}}{2},\]
	when $|x-\xi|\leq 1/3$ by Lemma~\ref{lem:hormander}. Hence, since $|I|\leq 1/3$, there exists an interval of length $1/12$ in which all points $x$ are at distance at least $1/12$ from $I$ and $|\psi(x)|\geq\| \psi \|_\infty/2$. Restricting the integration over $\mathbb R \setminus I$ to this interval and using that $kp-4\geq0$, we find by an elementary argument that 
	\begin{equation}\label{eq:otherside} 
		\int_{\mathbb{R}\setminus I} \Psi(x)\,dx\geq\frac{2^{-p}}{12} C^{kp-2} t_{n+k-1}^2 \| \psi \|_\infty^p 
	\end{equation}
	for an absolute constant $C>0$. Now plugging \eqref{eq:11:30} and \eqref{eq:otherside} into \eqref{eq:equ}, we find that $|I| \geq \delta$ for an absolute constant $\delta>0$. This leads to the conclusion that $\mathscr{Z}(\varphi)$ can be expressed as the union of $\lfloor 4/p \rfloor+2$ or fewer uniformly discrete sets. 
\end{proof}

In the proof of Theorem~\ref{thm:zeroset}~(a), we used Lemma~\ref{lem:hormander} to extract information about the behavior of certain functions in $PW^\infty$ near their global maxima. The following result allows us to say something about the local behavior of $\psi$ under an assumption that any extremal function will satisfy in light of Theorem~\ref{thm:zeroset}~(a). The basic idea is to consider the local symmetrization of $\psi$ in a point. The result will always be employed in conjunction with the arithmetic--geometric mean inequality.
\begin{lemma}\label{lem:local} 
	Suppose that $\psi$ is in $PW^p$ and that $\mathscr{Z}(\psi)$ is a finite union of uniformly discrete subsets of $\mathbb R$. Let $\xi$ be point in $\mathbb{R}\setminus \mathscr{Z}(\psi)$ and let $0<\eta\leq \dist(\xi,\mathscr{Z}(\psi))$. Then, for $|y|\leq \eta/2$ we have
	\[|\psi(\xi-y)\psi(\xi+y)| \geq \frac{3}{4}\left(\frac{2}{\pi}\right)^{2m} |\psi(\xi)|^2,\]
	where $m$ denotes the maximal number of zeros that $\psi$ has in an interval of length $\eta$. 
\end{lemma}
\begin{proof}
	Using the canonical factorization of $\psi$, we find that 
	\begin{equation*}
		\frac{\psi(\xi-y)\psi(\xi+y)}{\psi^2(\xi)} = \prod_{t \in \mathscr{Z}(\psi)} \left(1-\frac{y^2}{(\xi-t)^2}\right). 
	\end{equation*}
	For each $j$ in $ \mathbb{Z}$, we set $I_j:=\mathscr{Z}(\psi)\cap [\xi+\eta j, \xi+\eta(j+1))$. Note that $I_{0}=\emptyset$, $I_{-1}\subseteq \{\xi-\eta\}$, and each set $I_j$ contains at most $m$ points. Therefore, if $|y| \leq \eta$, then we get 
	\begin{multline*}
		\prod_{t \in \mathscr{Z}(\psi)} \left(1-\frac{y^2}{(\xi-t)^2}\right)=\prod_{j\in \mathbb{Z}\setminus \{0\}}\prod_{\,\,t\in I_j} \left(1-\frac{y^2}{(\xi-t)^2}\right) \\
		\geq \left(1-\frac{y^2}{\eta^2}\right) \prod_{j \in \mathbb{Z} \setminus \{0\}} \left(1-\frac{y^2}{\eta^2 j^2} \right)^{m}. 
	\end{multline*}
	Restricting further to $|y| \leq {\eta}/{2}$, we obtain the asserted estimate. 
\end{proof}

Lemma~\ref{lem:local} will find several applications in the proof of Theorem~\ref{thm:zeroset}~(c). In the proof of Theorem~\ref{thm:zeroset}~(b), we will use a slightly different version of the same argument. The same idea will also be used in the proof of Theorem~\ref{thm:sep}~(b) in Section~\ref{sec:2p4}.

When we estimated the right-hand side of \eqref{eq:equ} in the proof of Theorem~\ref{thm:zeroset}~(a), we got the factor $|I|^{kp-1}$. For the conclusion at the end of the argument to hold, we must have $kp-1 \geq 0$. To establish that the zero set is uniformly discrete, we must be able to take $k=2$ here. This is what forces us to require $p \geq 1/2$ in Theorem~\ref{thm:zeroset}~(b).

\begin{proof}[Proof of Theorem~\ref{thm:zeroset}~(b)] 
	Our plan is to assume that the zero set fails to be uniformly discrete and show that this leads to a contradiction. To this end, let $(t_{n_j})_{j\geq1}$ be a sequence of zeros of $\varphi_p$ such that $\delta_j:=t_{n_j+1}-t_{n_j} \to 0$. We may assume without loss of generality that $t_{n_j}\to \infty$ and also that $\delta_j\leq 1/4$ for all $j\geq1$. Since $\varphi$ has exponential type $\pi$, we know that $t_n/n \to 1$ as $n\to\infty$. We may therefore also assume that both 
	\begin{equation}\label{eq:mono} 
		t_{n_j}-t_{n_j-1} \geq \delta_j \qquad \text{and} \qquad t_{n_j+2}-t_{n_j+1}\geq \delta_j. 
	\end{equation}
	We will again invoke Lemma~\ref{lem:orthogonality}~(b), now with $t=t_{n_{j}}$ and $s=t_{n_j+1}$ so that $k=2$ in the argument used in the proof of Theorem~\ref{thm:zeroset}~(a). We set
	\[\psi_j(x) := \frac{\varphi(x)}{(x-t_{n_j})(x-t_{n_j+1})} \qquad \text{and} \qquad \Psi_j(x) := \frac{|\psi_j(x)|^p x^2}{|x-t_{n_j}|^{1-p} |x-t_{n_j+1}|^{1-p}},\]
	to find that 
	\begin{equation}\label{eq:eqpsik} 
		\int_{\mathbb{R}\setminus I_j} \Psi_j(x)\,dx = \int_{I_j} \Psi_j(x)\,dx, 
	\end{equation}
	where $I_j := [t_{n_j},t_{n_j+1}]$. Let $\xi_j$ a point where $|\psi_j|$ takes its maximum on $I_j$.
	
	The right-hand side of \eqref{eq:eqpsik} is estimated as before, so that we get 
	\begin{equation}\label{eq:sepup1} 
		\int_{I_j} \Psi_j(x)\,dx \leq t_{n_j+1}^2 \delta_j^{2p-1} \B(p,p) |\psi_j(\xi_j)|^p. 
	\end{equation}
	
	To estimate the left-hand side of \eqref{eq:eqpsik}, we first use the arithmetic--geometric mean inequality to see that
	\[\Psi_j(\xi_j-y)+\Psi_j(\xi_j+y) \geq 2 \sqrt{\Psi_j(\xi_j-y)\Psi_j(\xi_j+y)}.\]
	If we restrict our attention to $2\delta_j \leq y \leq 1$, then
	\[\sqrt{\Psi_j(\xi_j-y)\Psi_j(\xi_j+y)} \geq (t_{n_j}^2-1) \frac{y^{2(p-1)}}{2^p} |\psi_j(\xi_j)|^p |g_j(y)|^{p/2},\]
	for
	\[g_j(y):=\frac{\psi_j(\xi_j-y)\psi_j(\xi_j+y)}{\psi_j^2(\xi_j)} = \prod_{t \in \mathscr{Z}(\psi_j)} \left(1-\frac{y^2}{(\xi_j-t)^2}\right).\]
	It follows that 
	\begin{equation}\label{eq:lowerdis} 
		\int_{\mathbb{R}\setminus I_j} \Psi_j(x)\,dx \geq 2^{1-p} (t_{n_j}^2-1) |\psi_j(\xi_j)|^p \int_{2\delta_j}^1 y^{2(p-1)} |g_j(y)|^{p/2}\,dy. 
	\end{equation}
	We separate out the zeros of $\varphi$ that are close to our pair of zeros by defining 
	\begin{equation}\label{eq:Tj} 
		T_j := \left\{t \in \mathscr{Z}(\varphi) \setminus \{t_{n_j},t_{n_j+1}\} :\, \dist(t,\{t_{n_j},t_{n_j+1}\}) \leq 2 \right\}. 
	\end{equation}
	It follows from Theorem~\ref{thm:zeroset} (a) that there is a positive integer $d$ such that $T_j$ has at most $d$ elements for every $j \geq1$. Mimicking the proof of Lemma~\ref{lem:local}, we find that there is a constant $C$ which only depends on $\varphi$ such that
	\[|g_j(y)|\geq C \prod_{t \in T_j} \left|1-\frac{y^2}{(\xi_j-t)^2}\right|\]
	for $0\leq y \leq 1$. For every $t$ in $T_j$, we have that $\delta_j \leq |\xi_j-t| \leq 2+\delta_j$ by, respectively, \eqref{eq:mono} and \eqref{eq:Tj}. Set
	\[E_j = \bigcup_{t \in T_j} \bigg\{y \in [0,1]\,:\, 1-\frac{1}{10d} \leq \frac{y}{|\xi_j-t|} \leq 1 + \frac{1}{10d}\bigg\}.\]
	By construction, $|g_j|^p$ is bounded below by a constant which does not depend on $j$ when $y$ is in $[0,1] \setminus E_j$. Thus restricting the integration to $[2\delta_j,1]\setminus E_j$ when we estimate the right-hand side of \eqref{eq:lowerdis} from below, we find that
	\[\int_{I_j} \Psi_j(x)\,dx \geq (t_{n_j}^2-1) |\psi_j(\xi_j)|^p 
	\begin{cases}
		C_p, & \text{if } \frac{1}{2}<p<\infty; \\
		C_{\frac{1}{2}} \log{\frac{1}{\delta_j}}, & \text{if } p=\frac{1}{2}, 
	\end{cases}
	\]
	for a constant $C_p$ which only depends on $p$ and $\varphi$. We reach a contradiction as $j\to \infty$, since these bounds are in conflict with \eqref{eq:eqpsik} and \eqref{eq:sepup1}. 
\end{proof}

\subsection{Proof of Theorem~\ref{thm:zeroset} (c)} \label{sec:uppersep} We will now employ the reproducing formula of Theorem~\ref{thm:prpk}. As observed both in the introduction and in connection with the work of 
H\"{o}rmander and Bernhardsson \cite{HB93} discussed in Section~\ref{subsec:convex}, this formula is particularly transparent when $p=1$ since then $|\varphi_p|^{p-2} \varphi_p$ takes the values $\pm 1$ and changes sign at the zeros $t_n$. This makes the proof of part (c) rather simple when $p=1$. We therefore begin with this case, which still embodies the idea to be used for general $p\geq 1$.

\begin{proof}[Proof of Theorem~\ref{thm:zeroset}~(c): The case $p=1$]
	If $f$ is a function in $PW^1$ which is nonnegative on $\mathbb{R}$, then the case $p=1$ of Theorem~\ref{thm:prpk} implies that
	\begin{equation} \label{eq:cp1}
		\int_{I_n} f(x)\,dx \leq \|\varphi_1\|_1 f(0) + \int_{\mathbb{R}\setminus I_n} f(x)\,dx,
	\end{equation}
	where as usual $I_n := [t_n,t_{n+1}]$. Suppose that $0<T \leq |I_n|$ and let $\psi_0$ be the solution of the extremal problem \eqref{eq:exLPS} with $\Omega=\pi/2$. Choosing \[f(z) := \psi_0(z-\mu_n) \psi_0^\ast(z-\mu_n)\] for $\mu_n := (t_n+t_{n+1})/2$ in \eqref{eq:cp1}, we see that
	\[\lambda_0\left(\frac{\pi T}{4}\right) \leq \|\varphi_1\|_1|\psi_0(-\mu_n)|^2 + \left(1-\lambda_0\left(\frac{\pi T}{4}\right)\right).\]
	Since $|\psi_0(x)| \to 0$ when $|x|\to \infty$, we are led to the bound
	\[\limsup_{n\to\infty} \big(t_{n+1}-t_n\big) \leq T_0,\]
	where $T_0$ is the number such that $\lambda_0(\pi T_0/4)=1/2$. We extract from Table~\ref{table:lambda0} that $T_0 \leq 1.0805$.
\end{proof}

We will make use of the fact that $|\varphi_p|^{p-2} \varphi_p$ has constant sign on $(t_n,t_{n+1})$ also when $1<p<\infty$, but now the size of $|\varphi_p|$ matters as well. This makes the argument more complicated, and a somewhat elaborate preparation is required before we can proceed as in the case $p=1$. We state the first result in a slightly more general form than needed in the sequel.
\begin{lemma}\label{lem:monotonicity} 
	Let $(t_n)_{n\in \mathbb{Z}}$ be a strictly increasing sequence of real numbers with no finite accumulation point and let $(\alpha_n)_{n\in \mathbb{Z}}$ be an associated sequence of positive numbers. If
	\[ \sum_{n=-\infty}^{\infty} \frac{\alpha_n}{1+t_n^2} < \infty \]
	and $a$ is any real number, then the function
	\[ \Psi(x):=e^{a x} \prod_{|t_n|\leq 1} |x-t_n|^{\alpha_n} \prod_{|t_n|>1} \left|1-\frac{x}{t_n}\right|^{\alpha_n} e^{\alpha_n x/t_n} \]
	is well defined on $\mathbb{R}$. The function $\Psi'/\Psi$ is strictly decreasing on each interval $(t_n,t_{n+1})$. In particular, $\Psi$ has exactly one maximum on each interval $(t_n,t_{n+1})$. 
\end{lemma}
\begin{proof}
	The convergence of the infinite product to a continuous function is immediate. We clearly have $\Psi(t_n)=0$, and the remaining part of the lemma follows by logarithmic differentiation: 
	\begin{equation}\label{eq:logderiv} 
		\frac{d}{dx} \frac{\Psi'(x)}{\Psi(x)}=- \sum_{n=-\infty}^{\infty} \frac{\alpha_n}{(x-t_n)^2}. \qedhere 
	\end{equation}
\end{proof}

We will use both the orthogonality relation of Lemma~\ref{lem:orthogonality}~(b) and Theorem~\ref{thm:zeroset}~(b), which we have just established. As before, we consider the function
\[\Psi_n(x) := \frac{x^2 |\varphi_p(x)|^p }{|x-t_{n}|\,|x-{t_{n+1}}|}.\]
Since now $\Psi_n(t_n)=\Psi_n(t_{n+1})=0$ for $1<p<\infty$, Lemma~\ref{lem:monotonicity} shows that there exists a unique point $\xi_n$ in $(t_n,t_{n+1})$ such that $\Psi_n$ increases on $[t_n,\xi_n]$ and decreases on $[\xi_n,t_{n+1}]$. We begin by showing that $\xi_n$ cannot be arbitrarily close to the endpoints of this interval.
\begin{lemma}\label{lem:xiendpoints} 
	Fix $1 < p < \infty$ and let $(t_n)_{n\geq1}$ and $(\xi_n)_{n\geq1}$ be as above. Then
	\[\inf_{n\geq1} \dist\big(\xi_n,\{t_n,t_{n+1}\}\big)>0.\]
\end{lemma}
\begin{proof}
	We set $\varepsilon_n := \dist\big(\xi_n,\{t_n,t_{n+1}\}\big)$ for a fixed $n\geq1$. We have $t_{n+1}-t_n\geq\sigma$, where $\sigma=\sigma(\mathscr{Z}(\varphi_p))>0$ by Theorem~\ref{thm:zeroset}~(b). No argument is needed when $\varepsilon_n > \sigma/4$, so we assume that $\varepsilon_n\leq \sigma/4$. We will also assume that $\xi_n - t_n = \varepsilon_n$, the converse case being virtually identical. Our starting point is as before Lemma~\ref{lem:orthogonality}~(b). Hence we have 
	\begin{equation}\label{eq:Psixi} 
		\int_{\mathbb{R} \setminus I_n} \Psi_n(x)\,dx = \int_{I_n} \Psi_n(x)\,dx, 
	\end{equation}
	where $I_n = [t_n,t_{n+1}]$, and our goal is to estimate the right-hand side from above and the left-hand side from below.
	
	We first make a preliminary observation. Setting $\Theta_n(x): = |x-t_n|^{1-p} \Psi_n(x)$, we have
	\[\frac{\Psi_n'(x)}{\Psi_n(x)} = \frac{p-1}{x-t_n} + \frac{\Theta_n'(x)}{\Theta_n(x)}.\]
	Since $\Psi_n'(\xi_n)=0$, it is clear that 
	\begin{equation}\label{eq:xin} 
		\frac{\Theta_n'(\xi_n)}{\Theta_n(\xi_n)} = -\frac{p-1}{\varepsilon_n}.
	\end{equation}
	We turn to the upper bound for the right-hand side of \eqref{eq:Psixi}. We split the interval into two parts. On the interval $[t_n,t_n+2\varepsilon_n]$, we use the trivial estimate $\Psi_n(x) \leq \Psi_n(\xi_n)$. On the interval $[t_n+2\varepsilon_n,t_{n+1}]$, we find that 
	\begin{equation}\label{eq:preint1} 
		\frac{\Psi_n'(x)}{\Psi_n(x)} \leq \frac{p-1}{x-t_n} + \frac{\Theta_n'(\xi_n)}{\Theta_n(\xi_n)} = \frac{p-1}{x-t_n} - \frac{p-1}{\varepsilon_n} \leq -\frac{p-1}{2\varepsilon_n}. 
	\end{equation}
	Here the first inequality holds because $x>\xi_n$ and $x\mapsto \Theta_n'(x)/\Theta_n(x)$ is a decreasing function by Lemma~\ref{lem:monotonicity}. Integrating \eqref{eq:preint1}, we obtain the pointwise estimate
	\[\Psi_n(x) \leq \Psi_n(\xi_n)\exp\left({-\frac{p-1}{2\varepsilon_n} (x-\xi_n)}\right).\]
	Combining these bounds for $\Psi_n$, we find that 
	\begin{equation}\label{eq:Psixiupper} 
		\begin{split}
			\int_{I_n} \Psi_n(x)\,dx &\leq \Psi_n(\xi_n) \left(2\varepsilon_n + \int_{t_n+2\varepsilon_n}^{t_{n+1}} \exp\left({-\frac{p-1}{2\varepsilon_n} (x-\xi_n)}\right)dx\right)\\
			&\qquad\qquad\qquad\qquad\qquad\qquad\qquad\qquad\qquad\leq \Psi_n(\xi_n) \frac{2 p \varepsilon_n}{p-1}. 
		\end{split}
	\end{equation}
	To get a lower bound for the right-hand side of \eqref{eq:Psixi}, we restrict to $x$ that lie in the interval $[t_n-3\sigma/4,t_n-\sigma/4]$. Since $\xi_n\leq t_n+\sigma/4$, we will then have
	\[\frac{\Psi_n(x)}{\Psi_n(\xi_n)} = \frac{|x-t_n|^{p-1} \Theta_n(x)}{|\xi_n -t_n|^{p-1}\Theta_n(\xi_n)} \geq \frac{\Theta_n(x)}{\Theta_n(\xi_n)} \geq c_p.\]
	Here the final inequality follows from \eqref{eq:xin} and the fact that the derivative of $\Theta'_n/\Theta_n$ is bounded for $y$ in $[t_n-3\sigma/4,t_n+\varepsilon_n]$ independently of $n$, as follows from \eqref{eq:logderiv}. Consequently, 
	\begin{equation}\label{eq:Psixilower} 
		\int_{\mathbb{R} \setminus I_n} \Psi_n(x)\,dx \geq \int_{t_n-3\sigma/4}^{t_n-\sigma/4} \Psi_n(x)\,dx \geq \frac{\sigma}{2} c_p \Psi_n(\xi_n). 
	\end{equation}
	If $\varepsilon_n$ is sufficiently small we obtain a contradiction from \eqref{eq:Psixi}, \eqref{eq:Psixiupper}, and \eqref{eq:Psixilower}. 
\end{proof}
We now introduce the notations $\mu_n := (t_n+t_{n+1})/2$ and $\tau_n := t_{n+1}-t_n$.
\begin{lemma}\label{lem:midpoint} 
	Fix $1 \leq p < \infty$, and suppose that $n\geq2$ and $\tau_n \geq \max(\tau_{n+1},\tau_{n-1})$. There is a positive constant $C_p$ such that
	\[\Psi_n(\mu_n) \geq \frac{C_p}{\tau_n^2} \Psi_n(\xi_n),\]
	where $\xi_n$ denotes the unique point where $\Psi_n$ attains its maximum in $(t_n,t_{n+1})$. 
\end{lemma}
\begin{proof}
	We may assume that $t_n < \xi_n \leq \mu_n$, the converse case being completely analogous. We set 
	\begin{equation}\label{eq:Upsilonm} 
		\Theta_n(x) := \Psi_n(x) \frac{|x-t_n|^{1-p} |x-t_{n+1}|^{1-p}}{|x-t_{n+2}|^{p}}. 
	\end{equation}
	By the assumption that $\xi_n \leq \mu_n$, the left-hand side of
	\[\frac{\Psi_n'(x)}{\Psi_n(x)} = \frac{p-1}{x-t_n} + \frac{p-1}{x-t_{n+1}} + \frac{p}{x-t_{n+2}} + \frac{\Theta_n'(x)}{\Theta_n(x)}\]
	is nonpositive for $x=\mu_k$, and we therefore have
	\[ \frac{\Theta_n'(\mu_n)}{\Theta_n(\mu_n)} \leq \frac{p}{t_{n+2}-\mu_n}.\]
	We also have that $\Theta'_n/\Theta_n$ is decreasing on $(t_{n-1},t_{n+3})$ by \eqref{eq:logderiv}. Making use of these observations and integrating, we deduce that 
	\begin{equation}\label{eq:Upsilonmpw} 
		\Theta_n(x) \leq \Theta_n(\mu_n) \,\exp\left({ \frac{p\,(x-\mu_n)}{t_{n+2}-\mu_n}}\right) \leq \Theta_n(\mu_n) e^p 
	\end{equation}
	for $\mu_n \leq x \leq t_{n+2}$. Combining \eqref{eq:Upsilonm} and \eqref{eq:Upsilonmpw}, we find that
	\[\Psi_n(x) \leq e^p \Psi_n(\mu_n) \left(\frac{\tau_n}{2}\right)^{2-2p} \frac{|x-t_n|^{p-1} |x-t_{n+1}|^{p-1} |x-t_{n+2}|^p}{|t_{n+2}-\mu_n|^p}\]
	on the same interval. We apply this estimate to obtain 
	\begin{multline*}
		\Psi_{n+1}(x) = \Psi_n(x) \frac{|x-t_n|}{|x-t_{n+2}|} \\
		\leq e^p \Psi_n(\mu_n) \left(\frac{\tau_n}{2}\right)^{2-2p} \left(\frac{x-t_n}{t_{n+2}-\mu_n}\right)^p |x-t_{n+1}|^{p-1} |x-t_{n+2}|^{p-1}. 
	\end{multline*}
	Using that $\mu_n \leq t_{n+1}$ and that $\tau_{n+1} \leq \tau_n$, we infer that
	\[\left(\frac{x-t_n}{t_{n+2}-\mu_n}\right)^p \leq \left(\frac{2\tau_n}{\tau_{n+1}}\right)^p.\]
	Combining the last two estimates and using again that $\tau_{n+1} \leq \tau_n$, we find that 
	\begin{equation}\label{eq:Ikplus1} 
		\int_{I_{n+1}} \Psi_{n+1}(x)\,dx \leq e^p 2^{3p-2} \B(p,p) \Psi_n(\mu_n) \tau_n. 
	\end{equation}
	By Lemma~\ref{lem:xiendpoints}, there exists an absolute constant $\varepsilon>0$ such that $\xi_n-t_n\geq \varepsilon$ and such that $t_1>\varepsilon$. We may therefore invoke the arithmetic--geometric mean inequality as in the proof of Theorem~\ref{thm:zeroset}~(b) and then Lemma~\ref{lem:local} (with $m=1$) to the function $\psi_j$ at the point $\xi_j$ to get 
	\begin{equation}\label{eq:Ikplus2} 
		\int_{\mathbb{R} \setminus I_{n+1}} \Psi_{n+1}(x)\,dx \geq \int_{\xi_n-\varepsilon/2}^{\xi_n+\varepsilon/2} \Psi_n(x) \frac{|x-t_n|}{|x-t_{n+2}|} \,dx \geq \frac{\Psi_n(\xi_n)}{\tau_n} \frac{\sqrt{3} \varepsilon}{2^{2p+2} \pi}. 
	\end{equation}
	Combining \eqref{eq:Ikplus1} and \eqref{eq:Ikplus2} with Lemma~\ref{lem:orthogonality}~(b), we obtain the asserted result. 
\end{proof}

\begin{proof}[Proof of Theorem~\ref{thm:zeroset}~(c): The case $1 < p < \infty$] 
	We will argue by contradiction. Hence we begin by assuming that there is a subsequence $(t_{n_k})_{k\geq1}$ of $(t_n)_{n\geq1}$ such that $\tau_{n_k} \geq 2$, $\tau_{n_k} \to \infty$ as $k \to \infty$ and $\tau_{n_k} \geq \max(\tau_{n_k-1},\tau_{n_k+1})$ (we can assume this since $\varphi_p$ has exponential type $\pi$). Consider the function
	\[f_k(z) := z^2 \sinc^{4}\left(\frac{\pi}{4}(z-\mu_{n_k})\right).\]
	Since $f_k(0)=0$ and $f_k \Phi_p$ does not change sign on $I_{n_k}$, we infer from Theorem~\ref{thm:prpk} that 
	\begin{equation}\label{eq:uplowme} 
		\int_{I_{n_k}} |f_k(x)| |\varphi_p(x)|^{p-1} \,dx \leq \int_{\mathbb{R}\setminus I_{n_k}} |f_k(x)| |\varphi_p(x)|^{p-1}\,dx. 
	\end{equation}
	Our plan is now to obtain a lower bound for the left-hand side and an upper bound for the right-hand side, which will contradict the assumption that $\tau_{n_k} \to \infty$. 
	
	To get the desired lower bound, we first restrict the integral over $I_{n_k}$ to the interval $[\mu_{n_k}-1/2,\mu_{n_k}+1/2]$, then use a pointwise lower bound for $|f_k|$ and finally combine the arithmetic--geometric mean inequality with Lemma~\ref{lem:local} applied to $\varphi_p$ in the point $\mu_{n_k}$ to obtain
	\[\int_{I_{n_k}} |f_k(x)| |\varphi_p(x)|^{p-1}\,dx \geq A_p \mu_{n_k}^2 |\varphi_p(\mu_{n_k})|^{p-1}.\]
	Recalling that the relationship between $\varphi_p$ and $\Psi_{n_k}$, we reformulate this as 
	\begin{equation}\label{eq:lowyou} 
		\int_{I_{n_k}} |f_k(x)| |\varphi_p(x)|^{p-1}\,dx \geq A_p \Psi_{n_k}^{\frac{p-1}{p}}(\mu_{n_k}) \mu_{n_k}^{\frac{2}{p}} \left(\frac{\tau_{n_k}}{2}\right)^{\frac{2(p-1)}{p}}. 
	\end{equation}
	
	To get an upper bound for the right hand side of \eqref{eq:uplowme}, we first use H\"older's inequality to obtain
	\[\int_{\mathbb{R}\setminus I_{n_k}} |f_k(x)| |\varphi_p(x)|^{p-1}\,dx \leq H_1^{\frac{p-1}{p}} H_2^{\frac{1}{p}},\]
	where 
	\begin{align*}
		H_1 &:= \int_{\mathbb{R}\setminus I_{n_k}}\Psi_{n_k}(x)\,dx, \\
		H_2 &:= \int_{\mathbb{R}\setminus I_{n_k}} x^2 \sinc^{4p}\left(\frac{\pi}{4}(x-\mu_{n_k})\right)|x-t_{n_k}|^{p-1} |x-t_{n_k+1}|^{p-1}\,dx. 
	\end{align*}
	We bound $H_1$ by using Lemma~\ref{lem:orthogonality}~(b) and Lemma~\ref{lem:midpoint} so that we get
	\[H_1 = \int_{I_{n_k}} \Psi_{n_k}(x)\,dx \leq \tau_{n_k} \Psi_{n_k}(\xi_{n_k}) \leq C_p^{-1} \tau_{n_k}^3 \Psi_{n_k}(\mu_{n_k}).\]
	To bound $H_2$, we first note that all factors except $x^2$ are symmetric about $I_{n_k}$. If $x\geq t_{n_{k+1}}$, then
	\[\frac{x^2}{(x-\mu_{n_k})^{4p}} (x-t_{n_k})^{p-1}(x-t_{n_k+1})^{p-1} \leq \frac{(\mu_{n_k}+1)^2}{(x-\mu_{n_k})^{2p}},\]
	by the assumption that $\tau_{n_k}\geq2$. Integrating, we find that $H_2 \leq B_p (\mu_{n_k}+1)^2 \tau_{n_k}^{1-2p}$ and consequently that 
	\begin{equation}\label{eq:upyou} 
		\int_{\mathbb{R}\setminus I_{n_k}} |f_k(x)| |\varphi_p(x)|^{p-1}\,dx \leq D_p \Psi_{n_k}^\frac{p-1}{p}(\mu_{n_k}) (\mu_{n_k}+1)^{\frac{2}{p}} \tau_{n_k}^{1-\frac{2}{p}} 
	\end{equation}
	for some constant $D_p>0$. 
	
	Inserting \eqref{eq:lowyou} and \eqref{eq:upyou} into \eqref{eq:uplowme}, we obtain a bound on $\tau_{n_k}$ that contradicts the assumption that ${\tau_{n_k}}\to \infty$. 
\end{proof}
It would be of interest to establish Theorem~\ref{thm:zeroset}~(b) in the full range $0<p<\infty$. Theorem~\ref{thm:traces}, to be established below, yields a suitable extension of Theorem~\ref{thm:prpk} when $1/2\le p<1$, but the nature of the problem is quite different in this case, since we need to control the size of the reciprocal of the extremal function $\varphi$ rather than $\varphi$ itself. In particular, a major challenge would be to control global variations in the size of $|\varphi|^{-1}$. This interesting problem may require a more refined analysis of the orthogonality relations.

\section{Numerical bounds on the separation of zeros of extremal functions} \label{subsec:zeroset24} 
To place the present section in context, we begin with a brief discussion of how we may optimize our usage of the orthogonality relations of Lemma~\ref{lem:orthogonality}. Let us consider the following model problem. Given a finite interval $I = (t,s)$ with $t>0$, set
\[w_I(x):=\frac{x^2}{|(x-t)(x-s)|}, \]
and let $PW^p(I)$ consist of those functions in $PW^p$ that vanish at $t$ and $s$. Let 
\begin{equation}\label{eq:lpPLS} 
	\lambda_p(I) := \sup_{f\in PW(I)} \left\{\int_I |f(x)|^p \,w_I(x)\, dx \,:\, \int_{-\infty}^{\infty} |f(x)|^p \,w_I(x)\, dx = 1 \right\}. 
\end{equation}
Recall from Lemma~\ref{lem:unique} that in the convex range, the extremal functions are of the form
\[\varphi_p(z) = \prod_{n=1}^\infty \left(1-\frac{z^2}{t_n^2}\right).\]
If we are able to establish that
\[\sup_{\substack{s-t\leq \delta \\
t\geq t_1}} \lambda_p(I)<\frac{1}{2},\]
then it would follow that $t_{n+1}-t_n > \delta$ for every $n\geq1$. This is essentially the problem we wish to solve, up to whatever constraints we may add based on our knowledge of admissible extremal functions. 

Unfortunately, our understanding of the extremal problem \eqref{eq:lpPLS} is rudimentary at best. Our approach is therefore to rely on the precise numerics known for the Pollak--Landau--Slepian problem discussed in Section~\ref{subsec:PS}. This forces us to make the restriction $p>2$, because we use H\"{o}lder's inequality to estimate $L^2$ norms in terms of $L^p$ norms. In addition, we are faced with the problem that our weight $w$ blows up at the endpoints of $I$. We will circumvent that obstacle by doubling the size of the interval and use a different approach close to the endpoints of $I$. In this analysis, we will also invoke the known properties of the extremal function. 

Our ultimate goal is to establish Theorem~\ref{thm:sep}, but we begin with the following weaker result, which holds in more generality.
\begin{theorem}\label{thm:sepall} 
	\mbox{} 
	\begin{enumerate}
		\item[(a)] If $0<p<\infty$, then $t_1 \geq \frac{1}{2}$. 
		\item[(b)] If $2<p<\infty$, then $t_{n+1}-t_n \geq \frac{3}{5}$ for every $n\geq1$. 
	\end{enumerate}
\end{theorem}
\begin{proof}
	Part (a) follows at once from Lemma~\ref{lem:hormander} with $\xi=0$. For some fixed $n\geq1$, we set
	\[\psi(x) := \frac{x \varphi_p(x)}{(x-t_n)(x-t_{n+1})} \qquad \text{and} \qquad \Psi(x) := |\psi(x)|^p \frac{|x-t_n|^{p-1}|x-t_{n+1}|^{p-1}}{|x|^{p-2}}\]
	and note that $\psi$ is in $PW^2$ since $\varphi_p$ is in $PW^\infty$ and $\varphi_p(t_n)=\varphi_p(t_{n+1})=0$. The orthogonality relation of Lemma~\ref{lem:orthogonality}~(b) takes the form 
	\begin{equation}\label{eq:35ortho} 
		\int_{\mathbb{R}\setminus I} \Psi(x)\,dx = \int_I \Psi(x)\,dx, 
	\end{equation}
	where $I:=[t_n,t_{n+1}]$. We also introduce the notations
	\[\mu := \frac{t_n+t_{n+1}}{2}, \qquad \delta := t_{n+1}-t_n \qquad\text{and}\qquad J := [\mu-\delta,\mu+\delta].\]
	We adhere to the overall strategy outlined above. Specifically, we use H\"older's inequality with the assumption that $p>2$ to conclude that 
	\begin{equation}\label{eq:sepalllower1} 
		\int_{\mathbb{R} \setminus I} \Psi(x)\,dx \geq \int_{\mathbb{R} \setminus J} \Psi(x)\,dx \geq \frac{\left(\displaystyle \int_{\mathbb{R} \setminus J} |\psi(x)|^2\,dx \right)^{\frac{p}{2}}}{\left(\displaystyle \int_{\mathbb{R}\setminus J} \frac{x^2\,dx}{\big(|x-t_n|\,|x-t_{n+1}|\big)^{\frac{2(p-1)}{p-2}}}\right)^{\frac{p-2}{2}}}. 
	\end{equation}
	We have chosen $J$ to be concentric to $I$ with $|J|=2|I|=2\delta$, so since $\psi$ is in $PW^2$ we get with the convention $\lambda_0$ of Section~\ref{subsec:PS} that 
	\begin{equation}\label{eq:PSest} 
		\int_{\mathbb{R} \setminus J} |\psi(x)|^2\,dx \geq \|\psi\|_2^2 \big(1-\lambda_0(\pi \delta)\big) \geq \|\psi\|_\infty^2 \big(1-\lambda_0(\pi \delta)\big), 
	\end{equation}
	where we in the final estimate used that $\mathscr{C}_2=1$. Making use of the symmetry of the domain of integration, we next carry out a substitution to write
	\[\int_{\mathbb{R}\setminus J} \frac{x^2\,dx}{\big(|x-t_n|\,|x-t_{n+1}|\big)^{\frac{2(p-1)}{p-2}}} = \frac{2\mu^2}{\delta^{\frac{3p-2}{p-2}}} \int_1^\infty \left(x^2-\frac{1}{4}\right)^{-\frac{2(p-1)}{p-2}} \left(1+\frac{\delta^2 x^2}{\mu^2}\right)\,dx.\]
	To estimate the integral on the right-hand side, we use that $x^2-1/4 \geq 3/4$ to conclude that 
	\begin{multline*}
		2\int_1^\infty \left(x^2-\frac{1}{4}\right)^{-\frac{2(p-1)}{p-2}} \left(1+\frac{\delta^2 x^2}{\mu^2}\right)\,dx \\
		\leq 2 \left(\frac{4}{3}\right)^{\frac{2}{p-2}} \int_1^\infty \frac{1+\frac{\delta^2 x^2}{\mu^2}}{\left(x^2-\frac{1}{4}\right)^2}\,dx \\
		= \left(\frac{4}{3}\right)^{\frac{2}{p-2}} \left(\frac{16}{3}-4\log{3} + \frac{\delta^2}{\mu^2} \frac{4+3\log{3}}{3}\right). 
	\end{multline*}
	Inserting these estimates into \eqref{eq:sepalllower1}, we find that 
	\begin{equation}\label{eq:sepalllower2} 
		\int_{\mathbb{R} \setminus I} \Psi(x)\,dx \geq \frac{\|\psi\|_\infty^p \,\mu^{2-p}\, \delta^{\frac{3p-2}{2}} \big(1-\lambda_0(\pi \delta)\big)^{\frac{p}{2}}}{\frac{4}{3} \left(\frac{16}{3}-4\log{3} + \frac{\delta^2}{\mu^2} \frac{4+3\log{3}}{3}\right)^{\frac{p-2}{2}}}. 
	\end{equation}
	For the upper bound of the integral over $I$, we estimate 
	\begin{equation}\label{eq:step1upper1} 
		\begin{split}
			\int_I \Psi(x)\,dx &\leq \|\psi\|_\infty^p \int_I (x-t_n)^{p-1} (t_{n+1}-x)^{p-1} x^{2-p}\,dx \\
			&= \|\psi\|_\infty^p \mu^{2-p} \delta^{2p-1} \int_{|x| \leq 1/2} \left(\frac{1}{4}-x^2\right)^{p-1} \left(1-\frac{\delta x}{\mu}\right)^{2-p}\,dx. 
		\end{split}
	\end{equation}
	Since trivially $0<\delta/\mu < 2$, we compute 
	\begin{equation}\label{eq:max} 
		\max_{|x|\leq 1/2} \frac{\frac{1}{4}-x^2}{1-\frac{\delta x}{\mu}} = \frac{1}{\sqrt{4-\frac{\delta^2}{\mu^2}}+2}, 
	\end{equation}
	which when employed in \eqref{eq:step1upper1} implies that 
	\begin{equation}\label{eq:step1upper2} 
		\int_I \Psi(x)\,dx \leq \|\psi\|_\infty^p \,\mu^{2-p}\, \delta^{2p-1} \frac{1}{6} \left(\frac{1}{\sqrt{4-\frac{\delta^2}{\mu^2}}+2}\right)^{p-2}. 
	\end{equation}
	Combining \eqref{eq:sepalllower2} and \eqref{eq:step1upper2} with the orthogonality relation \eqref{eq:35ortho}, then simplifying, we find that 
	\begin{equation}\label{eq:step1postest} 
		1-\lambda_0(\pi \delta) \leq \delta A^{2/p} B^{1-2/p}, 
	\end{equation}
	where $A = 2/9$ and
	\[B = \left(\frac{16}{3}-4\log{3} + \frac{\delta^2}{\mu^2} \frac{4+3\log{3}}{3}\right)\left(\frac{1}{\sqrt{4-\frac{\delta^2}{\mu^2}}+2}\right)^2\]
	It is clear that $B$ is increasing as a function of $\delta/\mu$. We assume that $\delta \leq 3/5$ and aim to obtain a contradiction. We first deduce from part (a) of the current theorem that
	\[\frac{\delta}{\mu} \leq \frac{\delta}{\frac{\sqrt{2}}{\pi}+\frac{\delta}{2}} \leq \frac{\frac{3}{5}}{\frac{\sqrt{2}}{\pi}+\frac{3}{10}}< \frac{4}{5}.\]
	It follows that
	\[B \leq \left(\frac{16}{3}-4\log{3} + \left(\frac{4}{5}\right)^2 \frac{4+3\log{3}}{3}\right)\left(\frac{1}{\sqrt{4-\left(\frac{4}{5}\right)^2}+2}\right)^2 = 0.169841\ldots\]
	Since $B \leq A$ we conclude from \eqref{eq:step1postest} that
	\[1-\lambda_0(\pi \delta) \leq \frac{2}{9} \delta.\]
	Here the left-hand side is decreasing as a function of $\delta$, while the right-hand is increasing as a function of $\delta$. We obtain a contradiction to our assumption that $\delta \leq 3/5$ from the fact that
	\[1-\lambda_0(3\pi/5) \geq 0.14,\]
	which can be extracted from Table~\ref{table:lambda0}. Consequently, $t_{n+1}-t_n \geq 3/5$. 
\end{proof}

It follows from Theorem~\ref{thm:sepall} that the separation constant of $\mathscr{Z}(\varphi_p)$ is at least $3/5$. To improve this, the idea is to argue similarly as in the proof of Lemma~\ref{lem:local} to estimate $\Psi$ on $J\setminus I$, an interval which we in the previous argument simply discarded. In this estimate, we will take into account the information we have already established on $\mathscr{Z}(\varphi_p)$. This allows us to iteratively improve our estimates.

Since $\mathscr{Z}(\psi) = (\mathscr{Z}(\varphi_p) \setminus \{t_n,t_{n+1}\})\cup\{0\}$, the estimate for $t_1$ and the estimate for $t_{n+1}-t_n$ for every $n\geq1$ both have an effect on the separation constant of $\mathscr{Z}(\psi)$. If $t_1 \geq \gamma$ and $t_{n+1}-t_n \geq \delta_0$ for every $n\geq1$, then $\sigma(\mathscr{Z}(\psi)) \geq \min(\gamma,\delta_0)$. Setting 
\begin{equation}\label{eq:Galpha} 
	\mathscr{G}_\alpha(x) = \frac{\cos{\pi \alpha x}}{1-4 (\alpha x)^2}, 
\end{equation}
the most important result of the present section reads as follows.
\begin{lemma}\label{lem:iterateme} 
	Fix $2 < p < \infty$. Suppose that $t_{n+1}-t_n \geq \delta_0$ for every $n\geq1$ and that $t_1 \geq \gamma \geq \delta_0$. If $\delta = t_{n+1}-t_n$ for some $n\geq1$ and if $\delta \leq 3\delta_0/2$, then
	\[1-\lambda_0(\pi \delta) \leq \delta \max(A,B)\]
	where 
	\begin{align*}
		A &:= \frac{4}{3}\int_0^{1/2} \left(\frac{1}{4}-x^2\right) \left(2-\min\left(1,2\left(\frac{3-2x}{1+2x}\right)^{p-1}\left(\frac{\mathscr{G}_\alpha(1-x)}{\mathscr{G}_\alpha(x)}\right)^p\right)\right)dx, \\
		B &:= \left(\frac{16}{3}-4\log{3} + \beta^2 \frac{4+3\log{3}}{3}\right)\bigg(\frac{1}{\sqrt{4-\beta^2}+2}\bigg)^2, 
	\end{align*}
	for $\alpha = \delta/\delta_0$ and $\beta = \delta/(\gamma+\delta/2)$. 
\end{lemma}
\begin{proof}
	The proof is an elaboration of Theorem~\ref{thm:sepall} and we retain the definitions $\psi$, $\Psi$, $I$, $\mu$ and $J$ from that argument. We also retain the lower bound and deduce from \eqref{eq:sepalllower1}, \eqref{eq:PSest}, \eqref{eq:sepalllower2} and the fact that $\delta/\mu \leq \beta$ that 
	\begin{equation}\label{eq:iteratelower} 
		\int_{\mathbb{R} \setminus J} \Psi(x)\,dx \geq \frac{\|\psi\|_\infty^p \,\mu^{2-p}\, \delta^{\frac{3p-2}{2}} \big(1-\lambda_0(\pi \delta)\big)^{\frac{p}{2}}}{\frac{4}{3} \left(\frac{16}{3}-4\log{3} + \beta^2 \frac{4+3\log{3}}{3}\right)^{\frac{p-2}{2}}}. 
	\end{equation}
	For the upper bound, we begin by rewriting 
	\begin{multline*}
		\int_I \Psi(x)\,dx - \int_{J\setminus I} \Psi(x)\,dx \\
		= \int_{\mu-\delta/2}^{\mu} \big(\Psi(x)-\Psi(2\mu-\delta-x)\big)\,dx + \int_{\mu}^{\mu+\delta/2} \big(\Psi(x)-\Psi(2\mu+\delta-x)\big)\,dx. 
	\end{multline*}
	Next we estimate
	\[\int_{\mu-\delta/2}^\mu \big(\Psi(x)-\Psi(2\mu-\delta-x)\big)\,dx \leq \int_{\mu-\delta/2}^\mu \Psi(x)\left(1-\min\left(1,\frac{\Psi(2\mu-\delta-x)}{\Psi(x)}\right)\right)\,dx\]
	and similarly for the integral from $\mu$ to $\mu+\delta/2$. The purpose of these estimates is that we can now obtain an upper bound by estimating $\Psi$ pointwise from above. If $x$ is in $I$, then 
	\begin{align*}
		\Psi(x) &= |\psi(x)|^p \frac{|x-t_n|^{p-1}|x-t_{n+1}|^{p-1}}{x^{p-2}} \\
		&\leq \|\psi\|_\infty \frac{|x-t_n|^{p-1}|x-t_{n+1}|^{p-1}}{(\mu-|x-\mu|)^{p-2}} =: \|\psi\|_\infty w(x). 
	\end{align*}
	The virtue of this estimate is that $w(x)=w(2\mu-x)$ for $x$ in $I$. Using this estimate and translating the resulting integral, we find that 
	\begin{multline*}
		\int_{\mu-\delta/2}^\mu \big(\Psi(x)-\Psi(2\mu-\delta-x)\big)\,dx \\
		\leq \|\psi\|_\infty \int_0^{\delta/2} w(\mu-x) \left(1-\min\left(1,\frac{\Psi(\mu-\delta+x)}{\Psi(\mu-x)}\right)\right)\,dx. 
	\end{multline*}
	As above, we estimate the the integral from $\mu$ to $\mu+\delta/2$ similarly. Using the symmetry of $w$ and the elementary inequality $\min(1,a)+\min(1,b) \geq \min(1,2\sqrt{ab})$, which is valid for $a,b\geq0$, we find that 
	\begin{multline*}
		\int_I \Psi(x)\,dx - \int_{J\setminus I} \Psi(x)\,dx \\
		\leq \|\psi\|_\infty \int_0^{\delta/2} w(\mu-x) \left(2-\min\left(1,2\sqrt{\frac{\Psi(\mu-\delta+x)\Psi(\mu+\delta-x)}{\Psi(\mu-x)\Psi(\mu+x)}}\right)\right)\,dx. 
	\end{multline*}
	By a substitution, we can rewrite the latter integral (denoted $\mathscr{J}$ in what follows) as
	\[\mathscr{J} = \delta \int_0^{1/2} w(\mu-\delta x) \left(2-\min\left(1,2\sqrt{\frac{\Psi(\mu-\delta(1-x))\Psi(\mu+\delta(1-x))}{\Psi(\mu-\delta x)\Psi(\mu+\delta x)}}\right)\right)\,dx.\]
	Using \eqref{eq:max} as in the proof of Theorem~\ref{thm:sepall}, we then estimate 
	\begin{equation}\label{eq:west} 
		w(\mu-\delta x) \leq \delta^{2p-2} \mu^{2-p} \left(\frac{1}{\sqrt{4-\beta^2}+2}\right)^{p-2} \left(\frac{1}{4}-x^2\right) 
	\end{equation}
	for $0 \leq x \leq 1/2$ and $\delta/\mu \leq \beta$. We now seek to estimate the function
	\[\mathscr{P}(x) = \sqrt{\frac{\Psi(\mu-\delta(1-x))\Psi(\mu+\delta(1-x))}{\Psi(\mu-\delta x)\Psi(\mu+\delta x)}}\]
	from below and set
	\[\frac{\psi(\mu-\delta(1-x))\psi(\mu+\delta(1-x))}{\psi(\mu-\delta x)\psi(\mu+\delta x)} = \prod_{t \in \mathscr{Z}(\psi)} \frac{1-\frac{\delta^2(1-x)^2}{(t-\mu)^2}}{1-\frac{\delta^2x ^2}{(t-\mu)^2}}.\]
	By the assumption that $\delta \leq 3\delta_0/2$ and $0 \leq x \leq 1/2$, we obtain a lower bound for each factor by choosing the minimal value of $|t-\mu|$ for each $t$. We know that $\mathscr{Z}(\psi) = (\mathscr{Z}(\varphi_p) \setminus \{t_n,t_{n+1}\})\cup\{0\}$. The distance from $\mu$ to $\mathscr{Z}(\psi)$ is at least $\delta/2+\delta_0 \geq 3\delta_0/2$ and following this the zeros will be separated by at least $\min(\gamma,\delta_0)=\delta_0$. Taking into account the symmetry around $\mu$, we conclude that
	\[\sqrt{\frac{\psi(\mu-\delta(1-x))\psi(\mu+\delta(1-x))}{\psi(\mu-\delta x)\psi(\mu+\delta x)}} \geq \prod_{n=1}^\infty \frac{1-\frac{\delta^2(1-x)^2}{\delta_0^2(1/2+n)^2}}{1-\frac{\delta^2x ^2}{\delta_0^2(1/2+n)^2}} = \frac{\mathscr{G}_\alpha(1-x)}{\mathscr{G}_\alpha(x)},\]
	where we recall the convention $\alpha = \delta/\delta_0$ and \eqref{eq:Galpha}. It follows that
	\[\mathscr{P}(x) \geq \left(\frac{\mathscr{G}_\alpha(1-x)}{\mathscr{G}_\alpha(x)}\right)^p \left(\frac{3-2x}{1+2x}\right)^{p-1} \left(\frac{\mu^2-\delta^2(1-x)^2}{\mu^2- \delta^2 x^2}\right)^{1-\frac{p}{2}}.\]
	Since $2<p<\infty$ and $0 \leq x \leq 1/2$, the final factor is bounded below by $1$. This is attained as $\mu \to \infty$. Consequently 
	\begin{equation}\label{eq:Pest} 
		\mathscr{P}(x) \geq \left(\frac{3-2x}{1+2x}\right)^{p-1}\left(\frac{\mathscr{G}_\alpha(1-x)}{\mathscr{G}_\alpha(x)}\right)^p. 
	\end{equation}
	Combining \eqref{eq:west} and \eqref{eq:Pest}, we find that 
	\begin{multline*}
		\mathscr{J} \leq \delta^{2p-1} \mu^{2-p} \left(\frac{1}{\sqrt{4-\beta^2}+2}\right)^{p-2} \times \\
		\int_0^{1/2} \left(\frac{1}{4}-x^2\right) \left(2-\min\left(1,2\left(\frac{3-2x}{1+2x}\right)^{p-1}\left(\frac{\mathscr{G}_\alpha(1-x)}{\mathscr{G}_\alpha(x)}\right)^p\right)\right)\,dx. 
	\end{multline*}
	Recalling the upper bound
	\[\int_I \Psi(x)\,dx - \int_{J\setminus I} \Psi(x)\,dx \leq \|\psi\|_\infty \mathscr{J}\]
	and the lower bound \eqref{eq:iteratelower}, we obtain the stated result after simplifying the resulting inequality (in view of Lemma~\ref{lem:orthogonality}~(b)), then finally using that $p>2$. 
\end{proof}

The following elementary result shows that it is sufficient to obtain a contradiction for one $p_1 > 2$ and one $\delta_1 > \delta_0$ from Lemma~\ref{lem:iterateme} to obtain the same conclusion for the intervals $2 \leq p \leq p_1$ and $\delta_0 \leq \delta \leq \delta_1$.
\begin{lemma}\label{lem:Galpha} 
	Fix $0 \leq x \leq 1/2$. The function
	\[(\alpha,p) \quad \mapsto \quad \left(\frac{3-2x}{1+2x}\right)^{p-1}\left(\frac{\mathscr{G}_\alpha(1-x)}{\mathscr{G}_\alpha(x)}\right)^p\]
	is decreasing as function of both $0 < p < \infty$ and $1 \leq \alpha \leq 3/2$. 
\end{lemma}
\begin{proof}
	We first fix $0<p<\infty$ consider the expression as a function of $\alpha$. The function $f(y)=\mathscr{G}_1(y)= \frac{\cos{\pi y}}{1-4y^2}$ satisfies $f(y)\geq0$, $f'(y)<0$ and $f''(y)\leq0$ for $0 \leq y \leq 3/2$. Then, 
	\begin{multline*}
		\frac{d}{d\alpha} \frac{f(\alpha(1-x))}{f(\alpha x)} = \frac{1-x}{f(\alpha x)} f'(\alpha(1-x)) - \frac{x f(\alpha(1-x))}{f^2(\alpha x)} f'(\alpha x) \\
		\leq \frac{x}{f(\alpha x)}\left(f'(\alpha (1-x))-f'(\alpha x)\right) \leq 0, 
	\end{multline*}
	where we first used that $f'(y)<0$ and that $1-x \geq x$ in combination with that $f(\alpha (1-x)) \leq f(\alpha x)$ and that $-f'(y)>0$, then finally that $f''(y)\leq0$.
	
	To establish that the expression is decreasing as a function of $0<p<\infty$ if $1 \leq \alpha \leq 3/2$ and $0 \leq x \leq 1/2$ fixed, it is sufficient to establish that 
	\begin{equation}\label{eq:lol} 
		\frac{3-2x}{1+2x}\frac{\mathscr{G}_1(1-x)}{\mathscr{G}_1(x)} \leq 1 
	\end{equation}
	by the previous claim. This follows at once from the amusing observation \eqref{eq:lol} is an equality for all $0 \leq x \leq 1$. 
\end{proof}

In order to iterate our way from $t_{n+1}-t_n \geq 3/5$ from Theorem~\ref{thm:sepall}~(b) to $t_{n+1}-t_n \geq 2/3$ in Theorem~\ref{thm:sep}~(b) we require a better estimate for $t_1$ than the one established in Theorem~\ref{thm:sepall}~(a).

\begin{proof}[Proof of Theorem~\ref{thm:sep}~(a)] 
	Fix $2<p\leq 4$. Let $t_1$ be the smallest positive zero of $\varphi_p$ and set
	\[\psi(x) := \frac{\varphi_p(x)}{x-t_1} \qquad \text{and} \qquad \Psi(x) := |\psi(x)|^p |x-t_1|^{p-1} |x|.\]
	Note that $\psi$ is in $PW^2$ since $\varphi_p$ is in $PW^\infty$ and $\varphi_p(t_1)=0$. The orthogonality relation of Lemma~\ref{lem:orthogonality}~(a) is 
	\begin{equation}\label{eq:ortho1} 
		\int_{\mathbb{R} \setminus I} \Psi(x)\,dx = \int_I \Psi(x)\,dx, 
	\end{equation}
	where $I := [0,t_1]$. Setting $J := [-t_1/2,3t_1/2]$, we rewrite \eqref{eq:ortho1} as 
	\begin{equation}\label{eq:ortho2} 
		\int_{\mathbb{R} \setminus J} \Psi(x)\,dx = \int_I \Psi(x)\,dx - \int_{J \setminus I} \Psi(x)\,dx. 
	\end{equation}
	Using H\"older's inequality as in the proof of Theorem~\ref{thm:sepall}~(b), we find that 
	\begin{equation}\label{eq:t1lower1} 
		\int_{\mathbb{R} \setminus J} \Psi(x)\,dx \geq \frac{\|\psi\|_\infty^p \big(1-\lambda_0(\pi t_1)\big)^{\frac{p}{2}}}{\left(\displaystyle \int_{\mathbb{R}\setminus J} |x|^{-\frac{2}{p-2}} |x-t_1|^{-\frac{2(p-1)}{p-2}}\,dx\right)^{\frac{p-2}{2}}}. 
	\end{equation}
	Moreover, 
	\begin{multline*}
		\int_{\mathbb{R}\setminus J} |x|^{-\frac{2}{p-2}} |x-t_1|^{-\frac{2(p-1)}{p-2}}\,dx \\
		\leq \sup_{x \in \mathbb{R} \setminus J} \left(|x|^{-1} |x-t_1|^{-1}\right)^{\frac{4-p}{p-2}} \int_{\mathbb{R}\setminus J} |x|^{-1} |x-t_1|^{-3}\,dx \\
		= t_1^{-\frac{p+2}{p-2}} \left(\frac{4}{3}\right)^{\frac{4-p}{p-2}} \left(2\log{3}-\frac{8}{9}\right), 
	\end{multline*}
	since $2 <p \leq 4$. Inserting this estimate into \eqref{eq:t1lower1}, we find that 
	\begin{equation}\label{eq:t1lower2} 
		\int_{\mathbb{R} \setminus J} \Psi(x)\,dx \geq \frac{\|\psi\|_\infty^p \,t_1^{\frac{p}{2}+1} \,\big(1-\lambda_0(\pi t_1)\big)^{\frac{p}{2}}}{\big(\frac{4}{3}\big)^{\frac{4-p}{2}}\big(2\log{3}-\frac{8}{9}\big)^{\frac{p-2}{2}}}. 
	\end{equation}
	For the upper bound of the right-hand side of \eqref{eq:ortho2}, we discard the contribution of the interval $(t_1,3t_1/2]$ to get 
	\begin{equation}\label{eq:t1upper1} 
		\int_I \Psi(x)\,dx - \int_{J \setminus I} \Psi(x)\,dx \leq \int_0^{t_1/2} \big(\Psi(x)-\Psi(-x)\big)\,dx + \int_{t_1/2}^{t_1} \Psi(x)\,dx. 
	\end{equation}
	For the second integral in \eqref{eq:t1upper1}, we estimate
	\[\int_{t_1/2}^{t_1} \Psi(x)\,dx \leq \|\psi\|_\infty^p \int_{t_1/2}^{t_1} (t_1-x)^{p-1}x\,dx = \|\psi\|_\infty^p t_1^{p+1} \int_{1/2}^1 (1-x)^{p-1}x\,dx.\]
	For the first integral in \eqref{eq:t1upper1}, we use the fact that $\varphi_p$ is even (from Lemma~\ref{lem:unique}) to estimate 
	\begin{align*}
		\int_0^{t_1/2} \big(\Psi(x)-\Psi(-x)\big)\,dx &= \int_0^{t_1/2} |\psi(x)|^p (t_1-x)^{p-1} \frac{2x^2}{t_1+x}\,dx \\
		&\leq \|\psi\|_\infty^p t_1^{p+1} \int_0^{1/2} (1-x)^{p-1} \frac{2x^2}{1+x}\,dx. 
	\end{align*}
	Inserting these two estimates into \eqref{eq:t1upper1}, we find that 
	\begin{equation}\label{eq:t1upper2} 
		\int_I \Psi(x)\,dx - \int_{J \setminus I} \Psi(x)\,dx \leq \|\psi\|_\infty^p t_1^{p+1} \int_0^1 (1-x)^{p-1} w(x)\,dx, 
	\end{equation}
	where $w(x) := \frac{2x^2}{1+x}$ for $0 \leq x \leq 1/2$ and $w(x):=x$ for $1/2<x\leq1$. By H\"older's inequality and the assumption that $2 < p \leq 4$, we deduce that 
	\begin{multline*}
		\int_0^1 (1-x)^{p-1} w(x)\,dx \leq \left(\int_0^1 (1-x) w(x)\,dx\right)^{\frac{4-p}{2}} \left(\int_0^1 (1-x)^3 w(x)\,dx \right)^{\frac{p-2}{2}} \\
		= \left(4 \log{\left(\frac{3}{2}\right)}-\frac{3}{2}\right)^{\frac{4-p}{2}} \left(16\log{\left(\frac{3}{2}\right)}-\frac{6203}{960}\right)^{\frac{p-2}{2}}. 
	\end{multline*}
	Inserting this estimate in \eqref{eq:t1upper2}, then combining what we get with the modified orthogonality relation \eqref{eq:ortho2} and the lower bound \eqref{eq:t1lower2} we simplify to find that 
	\begin{equation}\label{eq:t1postest} 
		1-\lambda_0(\pi t_1) \leq t_1 A^{\frac{4-p}{p}} B^{\frac{p-2}{p}}, 
	\end{equation}
	where
	\[A = \frac{16}{3} \log{\left(\frac{3}{2}\right)}-2 \qquad \text{and} \qquad B = \left(2\log{3}-\frac{8}{9}\right)\left(16\log{\left(\frac{3}{2}\right)}-\frac{6203}{960}\right).\]
	Since $A < \sqrt{B}$, we find that $A^{\frac{4-p}{p}} B^{\frac{p-2}{p}} \leq \sqrt{B}$. Inserting this estimate into \eqref{eq:t1postest}, we conclude that 
	\begin{equation}\label{eq:t1postAB} 
		1-\lambda_0(\pi t_1) \leq t_1 \sqrt{B}. 
	\end{equation}
	The left-hand side of \eqref{eq:t1postAB} is a decreasing function of $t_1$, while the right-hand side is an increasing function of $t_1$. Consequently, if \eqref{eq:t1postAB} fails for some $t_1 = s$, then we must have $t_1>s$. Setting $s = 2/\pi$, we obtain the desired contradiction since $2 \sqrt{B}/\pi \leq 0.118$ and 
	\begin{equation}\label{eq:PS2pi} 
		1-\lambda_0(2) \geq 0.119, 
	\end{equation}
	where the latter was extracted from Table~\ref{table:lambda0}. 
\end{proof}

It seems difficult to do much better than $t_1 \geq 2/\pi$ uniformly in the range $2 \leq p \leq 4$. We expect the first positive zero of $\varphi_4$ to be close to the first positive zero of the function $f_4$ defined in \eqref{eq:fp}, which is $t_1 = 0.76547\ldots$.

\begin{proof}[Proof of Theorem~\ref{thm:sep}~(b)] 
	We will use Lemma~\ref{lem:iterateme} twice. In view of Theorem~\ref{thm:sep}~(a) we may take $\gamma = 2/\pi$ and by Theorem~\ref{thm:sepall}~(b) we can set $\delta_0 = 3/5$. By Lemma~\ref{lem:Galpha}, we see that $A$ is increasing as a function of both $p$ and $\delta$. We set $\delta=2/\pi$ and $p=4$ to get a value for $A$ which works for all $3/5 \leq \delta \leq 2/\pi$ and $2<p \leq 4$. Computing numerically with the \texttt{integrate} package from SciPy we find that
	\[A = 0.1440\ldots \qquad \text{and} \qquad B = 0.1337\ldots\]
	We obtain a contradiction to the inequality $1-\lambda_0(\pi \delta) \leq \delta A$ for $\delta = 2/\pi$ by \eqref{eq:PS2pi}. It follows that $t_{n+1}-t_n \geq 2/\pi$ for every $n\geq1$ and for all $2 \leq p \leq 4$. 
	
	Repeating the same analysis with $\gamma=\delta_0 = 2/\pi$, $p=4$ and $\delta=2/3$ we find that
	\[A = 0.1387\ldots \qquad \text{and} \qquad B = 0.1388\ldots\]
	which is a contradiction to the inequality $1-\lambda(\pi \delta) \leq \delta B$ for $\delta = 2/3$ by
	\[1-\lambda_0(2\pi/3)>0.103,\]
	which can be extracted from Table~\ref{table:lambda0}. This means that $t_{n+1}-t_n \geq 2/3$ for every $n\geq1$ and for all $2 \leq p \leq 4$ as desired. 
\end{proof}

\section{An upper bound for \texorpdfstring{$\mathscr{C}_p$}{Cp} when \texorpdfstring{$2 < p < 4$}{2<p<4}} \label{sec:2p4} 
A corollary to Theorem~\ref{thm:intrep} will serve as our starting point for the proof of Theorem~\ref{thm:korevaar}. To state it, we fix $1 \leq p < \infty$ and recall from Lemma~\ref{lem:unique} the unique\footnote{The fact that the extremal function is unique (and consequently even) is technically only required in the proof of Theorem~\ref{thm:sep}, but it simplifies the exposition of the present section greatly.} solution of \eqref{eq:extremalproblem} is of the form
\[\varphi_p(z) = \prod_{n=1}^\infty \left(1-\frac{z^2}{t_n^2}\right),\]
where $(t_n)_{n\geq1}$ is a strictly increasing sequence of positive numbers. Recall also our convention that $t_0=0$. Consider the function $K \colon (0,\infty) \to \mathbb{R}$ defined by 
\begin{equation}\label{eq:K} 
	K(x) := \sum_{n=0}^\infty \chi_{(t_n,t_{n+1})}(x)\, \frac{\sin{\frac{p}{2}\pi(x-n)}}{\pi x} 
\end{equation}
and set $K_+(x) := \max(K(x),0)$. 
\begin{corollary}\label{cor:kplus} 
	Fix $1 \leq p < \infty$ and let $K_+$ be as above. Then
	\[\mathscr{C}_p \leq 2 \int_0^\infty K_+^2(x)\,dx.\]
\end{corollary}
\begin{proof}
	We apply Theorem~\ref{thm:intrep} with $f=\varphi_p$ and $q=p/2$, then invoke Lemma~\ref{lem:unique} to rewrite the formula as
	\[1=2\int_0^{\infty} |\varphi_p(x)|^{p/2} K(x) \,dx.\]
	Squaring both sides, then using the definition of $K_+$ and the Cauchy--Schwarz inequality we find that
	\[1\leq 4 \left(\int_0^{\infty} |\varphi_p(x)|^{p/2} K_+(x) \,dx\right)^2 \leq 2 \|\varphi_p\|_p^p \int_{0}^\infty K_+^2(x)\, dx. \qedhere\]
\end{proof}

Before proceeding, let us take a look at the situation for the two endpoint cases $p=2$ and $p=4$. 
\begin{example}\label{ex:2} 
	If $p=2$, then the expression for $K$ in \eqref{eq:K} simplifies to
	\[K(x) = \sum_{n=0}^\infty \chi_{(t_n,t_{n+1})}(x)\, (-1)^n \sinc{\pi x}.\]
	It is clear that the choice $t_n=n$ for $n\geq1$ maximizes the upper bound of Corollary~\ref{cor:kplus}, since in this case $K(x)=K_+(x)=|\sinc{\pi x}|$. This gives yet another proof that $\mathscr{C}_2=1$ and that the unique solution of \eqref{eq:extremalproblem} is $\varphi_2(x)=\sinc{\pi x}$. 
\end{example}
\begin{example} 
	If $p=4$, then the expression for $K$ in \eqref{eq:K} simplifies to
	\[K(x) = 2\sum_{n=0}^\infty \chi_{(t_n,t_{n+1})}(x)\, \sinc{2 \pi x}.\]
	Here the choice of $t_n$ for $n\geq1$ is irrelevant. The upper bound of Corollary~\ref{cor:kplus} recovers the result from Theorem~\ref{thm:PW4}. 
\end{example}

The reader may at this point notice that if we had known that 
\begin{equation}\label{eq:2p4} 
	n-2/p \leq t_n \leq n, 
\end{equation}
then the proof in the case $2<p<\infty$ would have been equally trivial as in these two examples. It seems likely that \eqref{eq:2p4} holds for $2<p<4$, but, unfortunately, we only have at our disposal the much weaker assertions of Theorem~\ref{thm:sep}. Our plan is now to deduce from that theorem that the following result applies. 
\begin{lemma}\label{lem:collection} 
	Fix $1 \leq p < \infty$ and let $K_+$ be as above. Suppose that there is a sequence $\mathscr{I} = (I_k)_{k\geq0}$ of bounded measurable subsets of $(0,\infty)$ that enjoy the following properties: 
	\begin{itemize}
		\item[(a)] $K_+(x)=0$ on $(0,\infty)\setminus \bigcup_{k\geq0} I_k$. 
		\item[(b)] Set $\xi_k = \inf(I_k)$. Then $\xi_0=0$, $\xi_1\geq1$ and $\xi_{k+1} \geq \xi_k + 2/p$ for all $k\geq1$. 
		\item[(c)] For every $k\geq0$,
		\[\int_{I_k} K_+^2(x)\,dx \leq \int_{\xi_k}^{\xi_k+2/p} \frac{\sin^2{\frac{p}{2}\pi(x-\xi_k)}}{\pi^2 x^2}\,dx.\]
	\end{itemize}
	Then
	\[2 \int_0^\infty K_+^2(x)\,dx \leq \frac{p}{2}\left(1- 2(p-2)\int_1^\infty (\sinc{\pi x})^2 \, \frac{4x+p-2}{(2x+p-2)^2}\,dx\right).\]
\end{lemma}
\begin{proof}
	Using (a) and (c), then combining (b) with the fact that the function $x \mapsto 1/x$ is decreasing on $(0,\infty)$ and periodicity, we find that 
	\begin{align*}
		2 \int_0^\infty K_+^2(x)\,dx &\leq 2\sum_{k=0}^\infty \int_{\xi_k}^{\xi_k+2/p} \frac{\sin^2{\frac{p}{2}\pi(x-\xi_k)}}{\pi^2 x^2}\,dx \\
		&\leq 2\left(\int_0^{2/p} \frac{\sin^2{\frac{p}{2}\pi x}}{\pi^2 x^2}\,dx + \int_1^\infty \frac{\sin^2{\frac{p}{2}\pi (x-1)}}{\pi^2 x^2}\,dx\right). 
	\end{align*}
	The final expression is equal to the one given in the statement. 
\end{proof}

By Corollary~\ref{cor:kplus} and Lemma~\ref{lem:collection}, we will have a proof of Theorem~\ref{thm:korevaar} if we can produce the sequence of sets described in the latter result. We begin by establishing some terminology. 
\begin{definition}
	For $n\geq0$, we refer to the components of
	\[\left\{x\in\mathbb{R}\,:\,\sin{\frac{p}{2}\pi(x-n)} > 0\right\}\]
	as the intervals at \emph{level} $n$. An interval at level $n$ contained in $(t_n,t_{n+1})$ is called a \emph{stationary interval}. For $n\geq 1$, we assign a \emph{sign} $(\delta,\varepsilon)$ to $t_n$ as follows. We set $\delta$ equal to $+$ if $t_n$ is in the interior of an interval at level $n-1$ and equal to $-$ otherwise. Similarly, $\varepsilon$ is set equal to $+$ if $t_n$ is in the interior of an interval at level $n$ and equal to $-$ otherwise. When the sign of $t_n$ is $(+,\pm)$, we will refer to the interval at level $n-1$ containing $t_n$ as a \emph{departure interval}. When the sign of $t_n$ is $(\pm,+)$, we will refer to the interval at level $n$ containing $t_n$ as an \emph{arrival interval}. 
\end{definition}

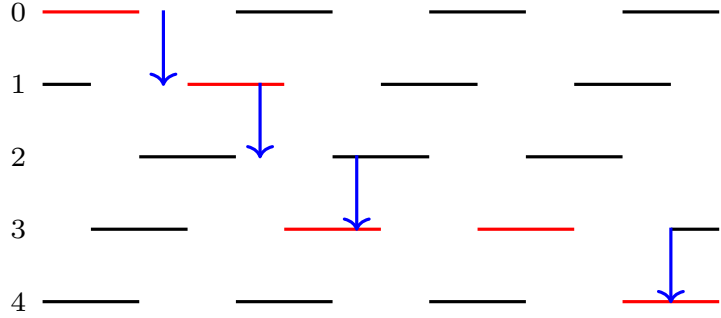
\begin{figure}
	\centering
	\begin{tikzpicture}[scale=1.5]
		\begin{axis}[
			axis equal image,
			axis lines = none,
			xmin = -2, 
			xmax = 30,
			ymin = -14, 
			ymax = 1]
			
			\node at (axis cs: -1,0) {$\scriptstyle 0$};
			\node at (axis cs: -1,-3) {$\scriptstyle 1$};
			\node at (axis cs: -1,-6) {$\scriptstyle 2$};
			\node at (axis cs: -1,-9) {$\scriptstyle 3$};
			\node at (axis cs: -1,-12) {$\scriptstyle 4$};
			
			\addplot[thick,solid,color=red] coordinates {(0,0) (4,0)};
			\addplot[thick,solid] coordinates {(8,0) (12,0)};
			\addplot[thick,solid] coordinates {(16,0) (20,0)};
			\addplot[thick,solid] coordinates {(24,0) (28,0)};
			
			\addplot[thick,solid] coordinates {(0,-3) (2,-3)};
			\addplot[thick,solid,color=red] coordinates {(6,-3) (10,-3)};
			\addplot[thick,solid] coordinates {(14,-3) (18,-3)};
			\addplot[thick,solid] coordinates {(22,-3) (26,-3)};
			
			\addplot[thick,solid] coordinates {(4,-6) (8,-6)};
			\addplot[thick,solid] coordinates {(12,-6) (16,-6)};
			\addplot[thick,solid] coordinates {(20,-6) (24,-6)};
			
			\addplot[thick,solid] coordinates {(2,-9) (6,-9)};
			\addplot[thick,solid,color=red] coordinates {(10,-9) (14,-9)};
			\addplot[thick,solid,color=red] coordinates {(18,-9) (22,-9)};
			\addplot[thick,solid] coordinates {(26,-9) (28,-9)};
			
			\addplot[thick,solid] coordinates {(0,-12) (4,-12)};
			\addplot[thick,solid] coordinates {(8,-12) (12,-12)};
			\addplot[thick,solid] coordinates {(16,-12) (20,-12)};
			\addplot[thick,solid,color=red] coordinates {(24,-12) (28,-12)};
			
			\addplot[thick,solid,color=blue,->] coordinates {(5,0+0.065) (5,-3-0.065)};
			\addplot[thick,solid,color=blue,->] coordinates {(9,-3+0.065) (9,-6-0.065)};
			\addplot[thick,solid,color=blue,->] coordinates {(13,-6+0.065) (13,-9-0.065)};
			\addplot[thick,solid,color=blue,->] coordinates {(26,-9+0.065) (26,-12-0.065)};
		\end{axis}
	\end{tikzpicture}
	\caption{Some intervals at levels $n=0,1,2,3,4$ for $p=3$. Here $t_1 = \frac{5}{6}$ has sign ${\color{blue}(-,-)}$, $t_2 = \frac{3}{2}$ has sign ${\color{blue}(+,-)}$, $t_3 = \frac{13}{6}$ has sign ${\color{blue}(+,+)}$ and $t_4=\frac{13}{3}$ has sign ${\color{blue}(-,+)}$. From left to right, the {\color{red}highlighted} intervals are stationary, departure interval for $t_2$, arrival interval for $t_3$, stationary, and arrival interval for $t_4$.}
	\label{fig:p3}
\end{figure}

We begin by pointing out that the interval at level $n$ are of the form $I=(\xi,\xi+2/p)$, where $\xi = n + k 4/p$ for $k$ in $\mathbb{Z}$. We also stress that a stationary interval can neither be arrival nor departure. See Figure~\ref{fig:p3} for an illustration of the definition. 

The idea is that the support of $K_+$ is the union of the intervals at level $n$ intersected with the interval $(t_n,t_{n+1})$. We traverse $(0,\infty)$ and jump from level $n$ to $n+1$ whenever we encounter $t_{n+1}$. As we go along, we will use an algorithm based on the signs of $t_n$ to construct the sequence $\mathscr{I}$. The following result is trivial, but we state it separately for ease of reference.
 
\begin{lemma}\label{lem:interval} 
	If $2 < p \leq 4$, then an interval at level $n$ intersects at most one interval at any other level. If $I=(\xi,\xi+2/p)$ is an interval at level $n$, then the interval at level $n+1$ intersecting $I$ is $(\xi+1-4/p,\xi+1-2/p)$. 
\end{lemma}

The key point of Lemma~\ref{lem:interval} is that $\xi+1-4/p \leq \xi$ since $2 < p \leq 4$. The effect of this is that arrival intervals are easier for us to deal with and that the most problematic case is $t_n$ of sign $(+,-)$. We need different arguments for the two cases $2 < p < 3$ and $3 \leq p \leq 4$. The easier latter case will be handled first. We invite the reader to consult Figure~\ref{fig:p3} during the proof. 

\begin{proof}[Proof of Theorem~\ref{thm:korevaar}: The case $3 \leq p \leq 4$] 
	The separation
	\[\sigma := \inf_{n\geq1} (t_{n+1}-t_n) \geq \frac{2}{3}\]
	established in Theorem~\ref{thm:sep}~(b) and the assumption that $p\geq3$ implies that no interval can be both an arrival interval and a departure interval. Let $\mathscr{I} = (I_k)_{k\geq0}$ be the sequence of intervals constructed iteratively for $n=0,1,2,\ldots$ as follows: 
	\begin{itemize}
		\item Include all stationary intervals at level $n$, ordered left to right. 
		\item If $t_{n+1}$ has sign $(+,-)$, include its departure interval in $\mathscr{I}$. 
		\item If $t_{n+1}$ has sign $(\pm,+)$, include its arrival interval (at level $n+1$) in $\mathscr{I}$. 
		\item If $t_{n+1}$ has sign $(-,-)$, do nothing. 
	\end{itemize}
	If $\xi_k = \inf(I_k)$, then Lemma~\ref{lem:interval} shows that $\mathscr{I}$ is ordered such that $\xi_k < \xi_{k+1}$. We need to check that $\mathscr{I}$ satisfies the three conditions (a)--(c) of Lemma~\ref{lem:collection}, which will establish the stated estimate by Corollary~\ref{cor:kplus}. It follows at once from Lemma~\ref{lem:interval} that (a) holds, since we always choose arrival intervals whenever possible. The only case of (c) which requires an argument is $(+,+)$, where we appeal to Lemma~\ref{lem:interval} again and use the fact $x \mapsto 1/x$ is decreasing on $(0,\infty)$ and periodicity. For (b), there are three cases to consider. 
	\begin{enumerate}
		\item[(i)] If $I_k$ and $I_{k+1}$ are on the same level, then $\xi_{k+1} = \xi_k + 4/p$. 
		\item[(ii)] If $I_k$ is at level $n-1$ and $I_{k+1}$ is at level $n$, then $\xi_{k+1} = \xi_k + 1$. We cannot have $I_{k+1} = (\xi_k+1-4/p,\xi_k+1-2/p)$ (see Lemma~\ref{lem:interval}), as this would contradict the fact that $\xi_k < \xi_{k+1}$. 
		\item[(iii)] If $I_k$ is at level $n-1$ and $I_{k+1}$ is at level $n-1+s$ for some $s\geq2$, then $s=2$ and $\xi_{k+1}=\xi_k+2-4/p$. The key point is that $t_n$ has sign $(\pm,-)$ and $t_{n+1}$ has sign $(+,+)$ (the latter due to the assumption that $I_{k+1}$ is not at level $n$ and that $\sigma \geq 2/3 \geq 2/p$ by Theorem~\ref{thm:sep}). Consequently, $s=2$ and $\xi_k+1 < t_{n+1} < \xi_k+1+2/p$. By Lemma~\ref{lem:interval}, the only interval at level $n+1$ which overlaps with this interval must be $I_{k+1}$. 
	\end{enumerate}
	Using the assumption that $p\geq 3$, we deduce from (i)--(iii) the general estimate $\xi_{k+1} \geq \xi_k+2/p$. If $t_1 \geq 2/p$, then $I_0 = (0,2/p)$ will be included in $\mathscr{I}$ as a stationary interval. Should $t_1<2/p$, we know by Theorem~\ref{thm:sep}~(a) that $t_1 \geq 2/\pi>1/2 \geq 1-2/p$ which means that the sign of $t_1$ is $(+,-)$ by Lemma~\ref{lem:interval} and so $I_0=(0,2/p)$ is included as a departure interval. Consequently, $\xi_0=0$. 
	
	For $\xi_1$, there are two cases to consider. If $t_2 \geq 2-2/p$, then we can exclude the case (iii) as the possible arrival interval at level $2$ is $I_1 = (2-4/p,2-2/p)$. Hence we have $\xi_1 \geq 1$. This means that (b) holds and we are done.
	
	In the case that $t_2 < 2-2/p$, we make the following adjustment to the algorithm. Since $t_2 > 2/\pi+2/3 \geq 5/4>2-4/p$ by Theorem~\ref{thm:sep}, it follows that the sign of $t_2$ is $(+,+)$. According to the algorithm, we should select the arrival interval $I_1=(2-4/p,2-2/p)$. We will instead select the departure interval $I_1 = (1,1+2/p)$ which has $\xi_1=1$ as desired. To justify this choice, we need to manually verify that condition (c) in Lemma~\ref{lem:collection} holds. This condition simplifies to
	\[\int_{t_2}^{2-2/p} \frac{\sin^2{\frac{p}{2}\pi (x-2)}}{\pi^2 x^2}\,dx \leq \int_{t_2}^{1+2/p} \frac{\sin^2{\frac{p}{2}\pi (x-1)}}{\pi^2 x^2}\,dx.\]
	We first note that $2-2/p \leq 1+2/p$ for $3 \leq p \leq 4$. Since $t_2 \geq 5/4$, it is therefore sufficient to establish that the pointwise estimate
	\[\sin^2{\frac{p}{2}\pi (x-2)} \leq \sin^2{\frac{p}{2}\pi (x-1)}\]
	holds for $5/4 \leq x \leq 2-2/p$. This pointwise estimate follows from the fact that both $\sin$-functions are positive on the interval and that the midpoint of $(1,2-2/p)$ is $3/2-1/p$ which is at most $5/4$. As we have now chosen the interval $I_1 = (1,1+2/p)$ instead of the interval $(2-4/p,2-2/p)$, we need to make sure that the next interval is not $(4-8/p,4-6/p)$ so that the requirement $\xi_2 \geq \xi_1+2/p$ still holds. The only way this interval could be included is as the arrival interval of $t_4$. However, by appealing to the separation, we find that
	\[t_4 \geq t_2 + 2 \sigma \geq \frac{5}{4} + \frac{4}{3} = \frac{31}{12}>\frac{5}{2} \geq 4-\frac{6}{p}\]
	for $3 \leq p \leq 4$, which shows that this cannot occur. 
\end{proof}

In the proof of Theorem~\ref{thm:korevaar} for $3 \leq p \leq 4$, the sets in $\mathscr{I}$ were all intervals of length $2/p$. In the proof of the case $2<p<3$, we will at times need to choose $I = (\xi,b)\cup (a,\eta)$ where $\xi$ is the left endpoint of some interval at level $n$ and $\eta$ is the right endpoint of some interval either at level $n$ or at level $n+2$. When checking condition (c) in Lemma~\ref{lem:collection}, the following trivial result will be helpful. 
\begin{lemma}\label{lem:doubleint} 
	If $0 \leq \xi < b < a < \eta$ satisfy $\eta \geq \xi+2/p$ and $\eta-a+b-\xi\leq 2/p$, then
	\[\int_\xi^b \frac{\sin^2{\frac{p}{2}\pi (x-\xi)}}{\pi^2 x^2}\,dx + \int_a^\eta \frac{\sin^2{\frac{p}{2}\pi (\eta-x)}}{\pi^2 x^2}\,dx \leq \int_\xi^{\xi+2/p} \frac{\sin^2{\frac{p}{2}\pi (x-\xi)}}{\pi^2 x^2}\,dx. \]
\end{lemma}
\begin{proof}
	This follows at once from the fact that $x \mapsto 1/x$ is decreasing on $(0,\infty)$. 
\end{proof}

\begin{figure}
	\centering
	\begin{tikzpicture}[scale=1.5]
		\begin{axis}[
			axis equal image,
			axis lines = none,
			xmin = -2, 
			xmax = 30,
			ymin = -14, 
			ymax = 1]
			
			\node at (axis cs: -1,0) {$\scriptstyle 0$};
			\node at (axis cs: -1,-3) {$\scriptstyle 1$};
			\node at (axis cs: -1,-6) {$\scriptstyle 2$};
			\node at (axis cs: -1,-9) {$\scriptstyle 3$};
			\node at (axis cs: -1,-12) {$\scriptstyle 4$};
			
			\addplot[thick,solid,color=red] coordinates {(0,0) (6,0)};
			\addplot[thick,solid] coordinates {(12,0) (18,0)};
			\addplot[thick,solid] coordinates {(24,0) (30,0)};
						
			\addplot[thick,solid] coordinates {(0,-3) (2,-3)};
			\addplot[thick,solid,color=red] coordinates {(8,-3) (14,-3)};
			\addplot[thick,solid] coordinates {(20,-3) (26,-3)};
			
			\addplot[thick,solid] coordinates {(4,-6) (10,-6)};
			\addplot[thick,solid,color=red] coordinates {(16,-6) (22,-6)};
			\addplot[thick,solid] coordinates {(28,-6) (30,-6)};
			
			\addplot[thick,solid] coordinates {(0,-9) (6,-9)};
			\addplot[thick,solid] coordinates {(12,-9) (18,-9)};
			\addplot[thick,solid,color=red] coordinates {(24,-9) (30,-9)};

			\addplot[thick,solid] coordinates {(0,-12) (2,-12)};
			\addplot[thick,solid] coordinates {(8,-12) (14,-12)};
			\addplot[thick,solid] coordinates {(20,-12) (25,-12)};
			\addplot[thick,solid] coordinates {(25,-12) (26,-12)};
			
			\addplot[thick,solid,color=blue,->] coordinates {(13,0+0.065) (13,-3-0.065)};
			\addplot[thick,solid,color=blue,->] coordinates {(18.4,-3+0.065) (18.4,-6-0.065)};
			\addplot[thick,solid,color=blue,->] coordinates {(24.3,-6+0.065) (24.3,-9-0.065)};
			\addplot[thick,solid,color=blue,->] coordinates {(29.6,-9+0.065) (29.6,-12-0.065)};
		\end{axis}
	\end{tikzpicture}
	\caption{Some intervals at levels $n=0,1,2,3,4$ for $p=\frac{8}{3}$. Here $t_1 = \frac{13}{8}$ has sign ${\color{blue}(+,+)}$, $t_2 = \frac{23}{10}$ has sign ${\color{blue}(-,+)}$, $t_3 = \frac{73}{24}$ has sign ${\color{blue}(-,+)}$ and $t_4=\frac{89}{24}$ has sign ${\color{blue}(+,-)}$. From left to right, we have applied {\color{red} Rule 1}, {\color{red}Rule 4.3}, {\color{red}Rule 2.2}, {\color{red}Rule 2.2}, and {Rule 3.1}.}
	\label{fig:p83a}
\end{figure}
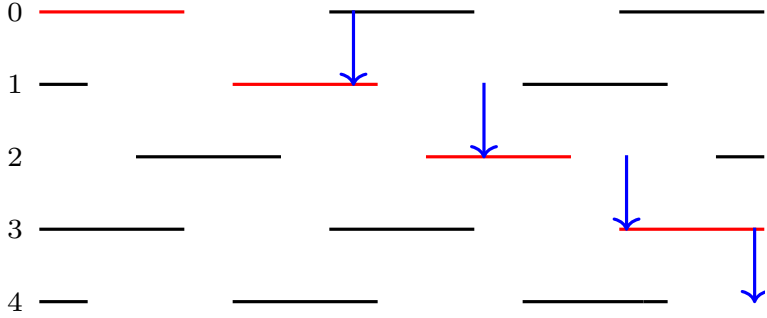
 
\begin{proof}[Proof of Theorem~\ref{thm:korevaar}: The case $2<p<3$] 
	In contrast to the previous case, it is now possible that an interval $I = (\xi,\xi+2/p)$ is both the arrival interval for $t_n$ and the departure interval for $t_{n+1}$. By Theorem~\ref{thm:sep}~(b), this means that
	\[\xi<t_n < t_n + 2/3 \leq t_{n+1} < \xi+2/p\]
	and hence $t_n$ has sign $(-,+)$ and $t_{n+1}$ has sign $(+,-)$ in view of Lemma~\ref{lem:interval}. We will choose the sequence $\mathscr{I}$ iteratively for $n=0,1,2,\ldots$ according to the following four rules. See Figure~\ref{fig:p83a}, Figure~\ref{fig:p83b} and Figure~\ref{fig:p83c} for illustrations of the rules. 
	\subsubsection*{Rule 1} All stationary intervals at level $n$ will be included in $\mathscr{I}$. 
	\subsubsection*{Rule 2} The sign of $t_{n+1}$ is $(-,+)$ and the arrival interval is $(\xi,\xi+2/p)$. 
	\begin{enumerate}
		\item[2.1:] If $I_{k-1} = (\xi+4/p-2,t_n)\cup(t_{n+1},\xi+2/p)$ is already included in $\mathscr{I}$ by either \emph{Rule 3.3} or \emph{Rule 4.1} below, then we do nothing and proceed. 
		\item[2.2:] If such a set is not included in $\mathscr{I}$, then we include the arrival interval $I_k = (\xi,\xi+2/p)$ in $\mathscr{I}$. 
	\end{enumerate}
	\subsubsection*{Rule 3} The sign of $t_{n+1}$ is $(+,-)$ and the departure interval is $(\xi,\xi+2/p)$. 
	\begin{enumerate}
		\item[3.1:] If $I_{k-1}=(\xi,\xi+2/p)$ is already included in $\mathscr{I}$ by \emph{Rule 2.2}, then we do nothing and proceed. 
		\item[3.2:] If $I_{k-1}=(\xi,\xi+2/p)$ is not included in $\mathscr{I}$ and if $t_{n+2} \geq \xi+2-2/p$, then we include the departure interval $I_k = (\xi,\xi+2/p)$ in $\mathscr{I}$. 
		\item[3.3:] If $I_{k-1}=(\xi,\xi+2/p)$ not included in $\mathscr{I}$ and if $t_{n+2}<\xi+2-2/p$, then we include $I_k = (\xi,t_{n+1})\cup(t_{n+2},\xi+2-2/p)$ in $\mathscr{I}$. 
	\end{enumerate}
	\subsubsection*{Rule 4} The sign of $t_{n+1}$ is $(+,+)$ and the arrival interval is $(\xi,\xi+2/p)$. By Lemma~\ref{lem:interval}, the departure interval is $(\xi+4/p-1,\xi+6/p-1)$. 
	\begin{enumerate}
		\item[4.1:] If $I_{k-1} = (\xi+4/p-2,t_n) \cup (t_{n+1},\xi+2/p)$ is already included in $\mathscr{I}$ either by \emph{Rule 3.3} or by the present rule and if $t_{n+2}<\xi+2/p+1$, then we include the set $I_k = (\xi+4/p-1,t_{n+1})\cup(t_{n+2},\xi+2/p+1)$ in $\mathscr{I}$. 
		\item[4.2:] If $I_{k-1} = (\xi+4/p-2,t_n) \cup (t_{n+1},\xi+2/p)$ is already included in $\mathscr{I}$ either by \emph{Rule 3.3} or by \emph{Rule 4.1} and if $t_{n+2} \geq \xi+2/p+1$, then we include\footnote{The second interval is only included to preserve the overall dichotomy that either $I = (\xi,\xi+2/p)$ or $I = (\xi,b)\cup(a,\eta)$. It could be dropped without affecting conditions (a)--(c) of Lemma~\ref{lem:collection}.} the set $I_k=(\xi+4/p-1,t_{n+1})\cup(\xi+2/p,\xi+6/p-1)$ in $\mathscr{I}$. 
		\item[4.3:] If $I_{k-1} = (\xi+4/p-2,t_n) \cup (t_{n+1},\xi+2/p)$ is not included in $\mathscr{I}$, then we include the arrival interval $I_k = (\xi,\xi+2/p)$ in $\mathscr{I}$. 
	\end{enumerate}
	
	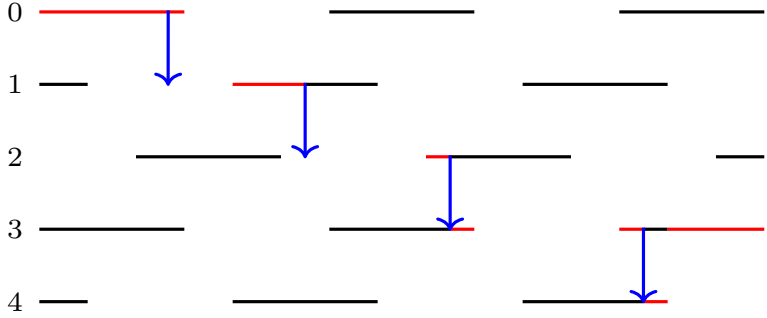
\begin{figure}
		\centering
		\begin{tikzpicture}[scale=1.5]
			\begin{axis}[
				axis equal image,
				axis lines = none,
				xmin = -2, 
				xmax = 30,
				ymin = -14, 
				ymax = 1]
	
				\node at (axis cs: -1,0) {$\scriptstyle 0$};
				\node at (axis cs: -1,-3) {$\scriptstyle 1$};
				\node at (axis cs: -1,-6) {$\scriptstyle 2$};
				\node at (axis cs: -1,-9) {$\scriptstyle 3$};
				\node at (axis cs: -1,-12) {$\scriptstyle 4$};
	
				\addplot[thick,solid,color=red] coordinates {(0,0) (6,0)};
				\addplot[thick,solid] coordinates {(12,0) (18,0)};
				\addplot[thick,solid] coordinates {(24,0) (30,0)};
				
				\addplot[thick,solid] coordinates {(0,-3) (2,-3)};
				\addplot[thick,solid,color=red] coordinates {(8,-3) (11,-3)};
				\addplot[thick,solid] coordinates {(11,-3) (14,-3)};
				\addplot[thick,solid] coordinates {(20,-3) (26,-3)};
	
				\addplot[thick,solid] coordinates {(4,-6) (10,-6)};
				\addplot[thick,solid,color=red] coordinates {(16,-6) (17,-6)};
				\addplot[thick,solid] coordinates {(17,-6) (22,-6)};
				\addplot[thick,solid] coordinates {(28,-6) (30,-6)};
	
				\addplot[thick,solid] coordinates {(0,-9) (6,-9)};
				\addplot[thick,solid] coordinates {(12,-9) (17,-9)};
				\addplot[thick,solid,color=red] coordinates {(17,-9) (18,-9)};
				\addplot[thick,solid,color=red] coordinates {(24,-9) (25,-9)};
				\addplot[thick,solid] coordinates {(25,-9) (26,-9)};
				\addplot[thick,solid,color=red] coordinates {(26,-9) (30,-9)};

				\addplot[thick,solid] coordinates {(0,-12) (2,-12)};
				\addplot[thick,solid] coordinates {(8,-12) (14,-12)};
				\addplot[thick,solid] coordinates {(20,-12) (25,-12)};
				\addplot[thick,solid,color=red] coordinates {(25,-12) (26,-12)};

				\addplot[thick,solid,color=blue,->] coordinates {(5.333,0+0.065) (5.33,-3-0.065)};
				\addplot[thick,solid,color=blue,->] coordinates {(11,-3+0.065) (11,-6-0.065)};
				\addplot[thick,solid,color=blue,->] coordinates {(17,-6+0.065) (17,-9-0.065)};
				\addplot[thick,solid,color=blue,->] coordinates {(25,-9+0.065) (25,-12-0.065)};
			\end{axis}
		\end{tikzpicture}
		\caption{Some intervals at levels $n=0,1,2,3,4$ for $p=\frac{8}{3}$. Here $t_1 = \frac{2}{3}$ has sign ${\color{blue}(+,-)}$, $t_2 = \frac{11}{8}$ has sign ${\color{blue}(+,-)}$, $t_3 = \frac{17}{8}$ has sign ${\color{blue}(+,+)}$, $t_4=\frac{25}{8}$ has sign ${\color{blue}(+,+)}$, and $t_5 \geq \frac{39}{8}$. From left to right, we have applied {\color{red}Rule 3.2}, {\color{red}Rule 3.3}, {\color{red}Rule 4.1} and {\color{red}Rule 4.2}.}
		\label{fig:p83b}
	\end{figure}
	
	\subsubsection*{Rule 5} If the sign of $t_{n+1}$ is $(-,-)$, then do we nothing and proceed. 
		
	\medskip \noindent We now need to check that conditions (a)--(c) of Lemma~\ref{lem:collection} are satisfied. It is clear that (a) holds. It is also not difficult to check that (c) holds. For \emph{Rule 2.1} and \emph{Rule 3.1} there is nothing to do. That (c) holds for \emph{Rule 1} and \emph{Rule 3.2} is trivial. For \emph{Rule 2.2} and \emph{Rule 4.3} we need to use that $x \mapsto 1/x$ is decreasing on $(0,\infty)$ and Lemma~\ref{lem:interval} as in the proof of the case $3 \leq p \leq 4$. For \emph{Rule 3.3}, \emph{Rule 4.1} and \emph{Rule 4.2} we use the separation $\sigma \geq 2/3$ and the assumption that $2<p<3$ to show that the conditions of Lemma~\ref{lem:doubleint} are satisfied. For (b), there are three cases to consider. Note first that $\xi_k = \inf(I_k)$ will always correspond to the left endpoint of an interval at some level. We will therefore think of $\xi_k$ as being on this level. 
	\begin{enumerate}
		\item[(i)] If $\xi_k$ and $\xi_{k+1}$ are on the same level, then $\xi_{k+1} = \xi_k + 4/p$. 
		\item[(ii)] If $\xi_k$ is at level $n-1$ and $\xi_{k+1}$ is at level $n$, then $\xi_{k+1} = \xi_k + 1$. This is clear since the rules avoid the interval $(\xi_k+1-4/p,\xi_k+1-2/p)$. 
		\item[(iii)] If $\xi_k$ is at level $n-1$ and $\xi_{k+1}$ at level $n-1+s$ for some $s \geq2$, then either $s=2$ and $\xi_k = \xi_k+2$ or $s=3$ and $\xi_k = \xi_k+3-4/p$. The key point is that the interval $(\xi_k+1,\xi_k+1+2/p)$ at level $n$ has been avoided. This is only possible by an application of \emph{Rule 3.3} or \emph{Rule 4.1} with $t_{n+1}<\xi_k+1$, such that the next rule applied is \emph{Rule 2.1}. We may reach level $n+2$ through either \emph{Rule 2.2} or \emph{Rule 4.3} applied to $t_{n+2}$. We refer to Figure~\ref{fig:p83c} for an illustration of the latter possibility. 
	\end{enumerate}
	In summary, we conclude that $\xi_{k+1} \geq \xi_k + 1$ for all $k\geq0$, which imply the weaker claim that $\xi_1 \geq1$ and $\xi_{k+1} \geq \xi_k + 2/p$ for $k\geq1$. It only remains to see that $\xi_0=0$. If $t_1 \geq 2/p$ this is evident, since the interval $I_0 = (0,2/p)$ will be stationary and included in $\mathscr{I}$. If $2/\pi \leq t_1 < 2/p$, then the sign of $t_1$ is $(+,-)$ by Lemma~\ref{lem:interval} and the fact that $2/\pi > 1/3 \geq 1-2/p$ for $2 \leq p \leq 3$. This means that $\xi_0=0$ either by \emph{Rule 3.2} or by \emph{Rule 3.3}, depending on $t_2$. 
\end{proof}

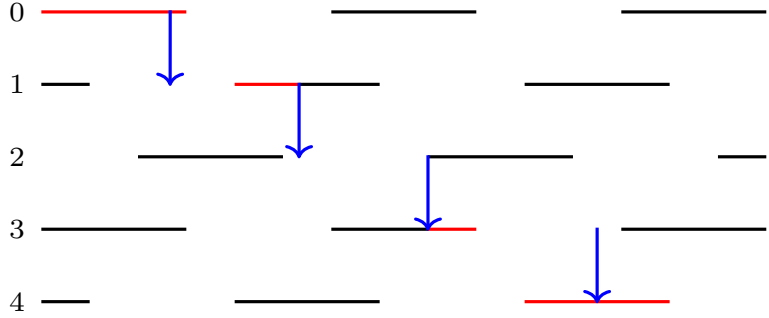
\begin{figure}
	\centering
	\begin{tikzpicture}[scale=1.5]
		\begin{axis}[
			axis equal image,
			axis lines = none,
			xmin = -2, 
			xmax = 30,
			ymin = -14, 
			ymax = 1]
	
			\node at (axis cs: -1,0) {$\scriptstyle 0$};
			\node at (axis cs: -1,-3) {$\scriptstyle 1$};
			\node at (axis cs: -1,-6) {$\scriptstyle 2$};
			\node at (axis cs: -1,-9) {$\scriptstyle 3$};
			\node at (axis cs: -1,-12) {$\scriptstyle 4$};
	
			\addplot[thick,solid,color=red] coordinates {(0,0) (6,0)};
			\addplot[thick,solid] coordinates {(12,0) (18,0)};
			\addplot[thick,solid] coordinates {(24,0) (30,0)};
		
			\addplot[thick,solid] coordinates {(0,-3) (2,-3)};
			\addplot[thick,solid,color=red] coordinates {(8,-3) (10.666,-3)};
			\addplot[thick,solid] coordinates {(10.666,-3) (14,-3)};
			\addplot[thick,solid] coordinates {(20,-3) (26,-3)};
	
			\addplot[thick,solid] coordinates {(4,-6) (10,-6)};
			\addplot[thick,solid] coordinates {(16,-6) (22,-6)};
			\addplot[thick,solid] coordinates {(28,-6) (30,-6)};
	
			\addplot[thick,solid] coordinates {(0,-9) (6,-9)};
			\addplot[thick,solid] coordinates {(12,-9) (16,-9)};
			\addplot[thick,solid,color=red] coordinates {(16,-9) (18,-9)};
			\addplot[thick,solid] coordinates {(24,-9) (30,-9)};			
	
			\addplot[thick,solid] coordinates {(0,-12) (2,-12)};
			\addplot[thick,solid] coordinates {(8,-12) (14,-12)};
			\addplot[thick,solid,color=red] coordinates {(20,-12) (26,-12)};

			\addplot[thick,solid,color=blue,->] coordinates {(5.333,0+0.065) (5.33,-3-0.065)};
			\addplot[thick,solid,color=blue,->] coordinates {(10.666,-3+0.065) (10.666,-6-0.065)};
			\addplot[thick,solid,color=blue,->] coordinates {(16,-6+0.065) (16,-9-0.065)};
			\addplot[thick,solid,color=blue,->] coordinates {(23,-9+0.065) (23,-12-0.065)};
		\end{axis}
	\end{tikzpicture}
	\caption{Some intervals at levels $n=0,1,2,3,4$ for $p=\frac{8}{3}$. Here $t_1 = \frac{2}{3}$ has sign ${\color{blue}(+,-)}$, $t_2 = 
\frac{4}{3}$ has sign ${\color{blue}(+,-)}$, $t_3 = 2$ has sign ${\color{blue}(-,+)}$, and $t_4=\frac{23}{8}$ has sign ${\color{blue}(-,+)}$. From left to right, we have applied {\color{red}Rule 3.2}, {\color{red}Rule 3.3}, {Rule 2.1}, and {\color{red}Rule 4.3}.}
	\label{fig:p83c}
\end{figure}

\section{Asymptotics for \texorpdfstring{$\mathscr{C}_p$}{Cp}} \label{sec:asymp} 
We have now come to the problems of estimating $\mathscr{C}_p$ as $p\to0^+$ and as $p\to \infty$. The main goals of the present section are to prove Theorem~\ref{thm:Cp0} and to prove Theorem~\ref{thm:Cpinfty}. We begin with some preliminary estimates on a family of entire functions of exponential type $\pi$.

\subsection{A family of entire functions} \label{subsec:gamma} For $\alpha>\frac{1}{2}$, consider the function
\[g_\alpha(z) := (2\alpha-1) \, 2^{2\alpha-2} \B(\alpha,\alpha) \int_{-\pi}^\pi \left(\cos{\frac{\xi}{2}}\right)^{2\alpha-2} e^{i z \xi}\,\frac{d\xi}{2\pi}.\]
The normalizing factor is chosen so that $g_\alpha(0)=1$. Note also that $g_1(z) = \sinc{\pi z}$. By a calculation attributed to Ramanujan (see \cite[Section~7.6]{Titchmarsh86}), we deduce that 
\begin{equation}\label{eq:galpha} 
	g_\alpha(z) = \frac{\Gamma^2(\alpha)}{\Gamma(\alpha-z)\Gamma(\alpha+z)} = \prod_{n=1}^\infty \left(1-\frac{z^2}{(n+\alpha-1)^2}\right). 
\end{equation}
The formula \eqref{eq:galpha} defines an entire function of exponential type $\pi$ for all $\alpha>0$. In particular, we see that $g_{1/2}(z) = \cos{\pi z}$.

Note the similarity of $\widehat{g}_\alpha$ and $\widehat{f}_p$ from \eqref{eq:fp}. From this one would expect that choosing $\alpha = \frac{1}{2}+\frac{1}{p}$ could yield a reasonable lower bound for $\mathscr{C}_p$. We have numerically verified that in the range $1 \leq p \leq 4$ the lower bounds obtained by $f_p$ are better. However, the virtue of $g_\alpha$ is that we know exactly the location of its zeros. This makes it possible (see e.g.~\cite{Gorbachev05}) to compute
\[\mathscr{C}_1 \geq \left( \pi \int_0^1 \sinc{\pi x}\,dx\right)^{-1} = 0.5399751567\ldots\]
by testing \eqref{eq:pointeval} with the function $g_{3/2}(z) = \frac{\cos{\pi z}}{1-4z^2}$. This is only slightly worse than the numerical lower bound obtained by H\"{o}rmander and Bernhardsson in \eqref{eq:hb}.

We choose to work with $g_\alpha$ for precisely the same reason when obtaining lower asymptotic estimates for $\mathscr{C}_p$ as $p\to \infty$ and $p \to 0^+$ by letting $\alpha$ go to respectively $\frac{1}{2}$ and $ \infty$. We will rely on the following estimates.
\begin{lemma}\label{lem:galphaest} 
	\mbox{} 
	\begin{enumerate}
		\item[(a)] If $0 \leq x \leq \alpha$, then
		\[g_\alpha(x) \leq \left(1-\frac{x}{\alpha}\right)^{-(\alpha-1/2-x)} \left(1+\frac{x}{\alpha}\right)^{-(\alpha-1/2+x)}.\]
		\item[(b)] If $1/2 \leq \alpha \leq x$, then
		\[|g_\alpha(x)| \leq \alpha^{2\alpha-1} \frac{(1-\alpha+x)^{1/2-\alpha+x}}{(\alpha+x)^{\alpha-1/2+x}} |\sin{\pi(\alpha-x)}|.\]
		\item[(c)] If $\alpha \geq 1/2$ and $0 \leq x \leq \alpha$, then
		\[|g_\alpha(x)| \leq \cos\left(\frac{\pi x}{2\alpha}\right).\]
	\end{enumerate}
\end{lemma}
\begin{proof}
	The well known formula 
	\begin{equation}\label{eq:loggamma} 
		\log{\Gamma(z)} = \left(z-\frac{1}{2}\right) \log{z} - z + \frac{1}{2}\log{2\pi } + 2\int_0^{\infty} \frac{\arctan \frac{t}{z}}{e^{2\pi t}-1} \,dt 
	\end{equation}
	is valid for $\mre{z}>0$. We begin with (a). If $0 < x \leq \alpha$ and $t \geq0$, then it follows from convexity that
	\[\arctan{\frac{t}{\alpha-x}} -2\arctan{\frac{t}{\alpha}} + \arctan{\frac{t}{\alpha+x}} \geq 0.\]
	We therefore obtain from three applications of \eqref{eq:loggamma} that
	\[\log{g_\alpha(x)} \leq 2\left(\alpha-\frac{1}{2}\right)\log{\alpha}- \left(\alpha-x-\frac{1}{2}\right)\log(\alpha-x) - \left(\alpha+x-\frac{1}{2}\right)\log(\alpha+x),\]
	which simplifies to the stated result (a). For the proof of (b), we first use the reflection formula for the gamma function to obtain 
	\begin{equation}\label{eq:galphareflect} 
		g_\alpha(x) = \frac{\Gamma^2(\alpha)}{\pi} \frac{\Gamma(1-\alpha+x)}{\Gamma(\alpha+x)} \sin{\pi(\alpha-x)}. 
	\end{equation}
	Fix $t>0$ and consider
	\[F(\alpha,x) := 2 \arctan{\frac{t}{\alpha}}+\arctan{\frac{t}{1-\alpha+x}}-\arctan{\frac{t}{\alpha+ x}}.\]
	We verify by differentiation that if $\alpha \geq 1/2$, then $x \mapsto F(\alpha,x)$ is decreasing. Another differentiation reveals that $\alpha \mapsto F(\alpha,\alpha)$ is also decreasing. It follows that
	\[2 \arctan{\frac{t}{\alpha}}+\arctan{\frac{t}{1-\alpha+x}}-\arctan{\frac{t}{\alpha+ x}} \leq 2 \arctan{2t}\]
	for all $x \geq \alpha \geq 1/2$ and $t\geq 0$. From this estimate and three applications of \eqref{eq:loggamma}, we find that 
	\begin{multline*}
		\log{\frac{\Gamma^2(\alpha) \Gamma(1-\alpha+x)}{\Gamma(\alpha+x)}} \leq 2\left(\alpha-\frac{1}{2}\right)\log{\alpha}+\left(\frac{1}{2}-\alpha+x\right)\log(1-\alpha+x) \\
		- (\alpha-1/2+x)\log(\alpha+x) + \log(2\pi) -1 + 4 \int_0^\infty \frac{\arctan{2t}}{e^{2\pi t}-1}\,dt. 
	\end{multline*}
	The integral evaluates to $(1-\log{2})/4$. Inserting the resulting estimate into \eqref{eq:galphareflect}, we obtain (b). The proof of (c) is the easiest one. If $\alpha \geq 1/2$ and $n\geq1$, then
	\[(n-1+\alpha) \leq \left(n-\frac{1}{2}\right) 2\alpha.\]
	The stated estimate follows after inserting this estimate into \eqref{eq:galpha} for $0 \leq x \leq \alpha$. 
\end{proof}

\subsection{Asymptotics as \texorpdfstring{$p\to\infty$}{p to infty}}  It is not difficult to see that an upper bound of the form $\mathscr{C}_p \leq C p$ is not asymptotically sharp as $p\to\infty$. Suppose that $f$ in $PW^p$ is real, even, and satisfies $f(0)=\|f\|_\infty$. Using Lemma~\ref{lem:hormander}, we find that
\begin{equation} \label{eq:fcos}
	\int_{-\infty}^\infty |f(x)|^p \,dx \geq \int_{-\frac{1}{2}}^{\frac{1}{2}} \|f\|_\infty^p \left(\cos{\pi x}\right)^p \,dx.
\end{equation}
Combining this with the well known asymptotic expansion
\begin{equation} \label{eq:betacos}
	\int_{-\frac{1}{2}}^{\frac{1}{2}} (\cos{\pi x})^p \,dx = \frac{\B((p+1)/2,1/2)}{\pi} = \sqrt{\frac{2}{\pi p}} + O\left(\frac{1}{p^{\frac{3}{2}}}\right)
\end{equation}
as $p \to \infty$, we conclude from \eqref{eq:extremalproblem} and Lemma~\ref{lem:unique} that
\[\mathscr{C}_p \leq \sqrt{\frac{\pi p}{2}} + O\left(\frac{1}{\sqrt{p}}\right)\]
as $p\to \infty$. In Theorem~\ref{thm:Cpinfty} we will sharpen this upper bound and provide a matching lower bound. In addition to \eqref{eq:fcos} and \eqref{eq:betaasymp}, we require the Riesz interpolation formula (see e.g.~\cite[Section~11.3]{Boas54}), which states that if $f$ is in $PW^\infty$, then 
\begin{equation}\label{eq:RIF} 
	f'(x_0)=\frac{4}{\pi} \sum_{n=-\infty}^{\infty}(-1)^n \frac{f(x_0+n+1/2)}{(2n+1)^2}. 
\end{equation}
We will also use Bernstein's inequality (see e.g.~\cite[Lecture~28]{Levin96}), which states that
\begin{equation}\label{eq:bernstein} 
	\|f'\|_\infty \leq \pi \|f\|_\infty 
\end{equation}
for every $f$ in $PW^\infty$. The inequality is attained if and only if $f(z) = a \cos{\pi z} + b \sin{\pi z}$.

\begin{proof}[Proof of Theorem~\ref{thm:Cpinfty}: Upper bound] 
	Fix $1<p<\infty$. In view of Lemma~\ref{lem:unique}, we may assume without loss of generality that $f$ is real, even, and satisfies $\|f\|_\infty=f(0)=1$. Since $f$ is a nontrivial function in $PW^p$ for $p<\infty$, we know that \eqref{eq:bernstein} is not attained. We may consequently assume that $\|f''\|_{\infty} =\pi^2-\delta \pi$ for some $0<\delta<\pi$. There are now two cases to consider.
	
	The first case is that $\delta>c \frac{\log{p}}{p}$ for some positive constant $c$ to be chosen later. Using Taylor's theorem at the origin, we obtain the estimate
	\[|f(x)| \geq 1-\frac{\|f''\|_\infty}{2} x^2 = 1-\frac{\pi^2-\delta \pi}{2} x^2.\]
	Restricting the domain of integration to the interval $|x| \leq \frac{\sqrt{2}}{\pi} \left(1-\frac{\delta}{\pi}\right)^{-\frac{1}{2}}$ and using this estimate, we find that 
	\begin{equation}\label{eq:substitution} 
		\|f\|_p^p = \int_{-\infty}^\infty |f(x)|^p \,dx \geq \frac{\sqrt{2}}{\pi} \left(1-\frac{\delta}{\pi}\right)^{-\frac{1}{2}} \int_{-1}^1 \left(1-x^2\right)^p\,dx. 
	\end{equation}
	Combining the expansion
	\begin{equation}\label{eq:betaasymp} 
		\int_{-1}^1 \left(1-x^2\right)^p \,dx = \B(p+1,1/2) = \sqrt{\frac{\pi}{p}} + O\left(\frac{1}{p^{\frac{3}{2}}}\right) 
	\end{equation}
	with the estimate 
	\[\left(1-\frac{\delta}{\pi}\right)^{-\frac{1}{2}} \geq 1 + \frac{\delta}{2\pi}\]
	which holds for $0<\delta<\pi$, we deduce from \eqref{eq:substitution} that
	\begin{equation}\label{eq:ppest} 
		\|f\|_p^p \geq \sqrt{\frac{2}{\pi p}} + \sqrt{\frac{2}{\pi}} \frac{c}{2\pi} \frac{\log{p}}{p^{\frac{3}{2}}} + O\left(\frac{1}{p^{\frac{3}{2}}}\right). 
	\end{equation}
	as $p\to\infty$, since $\delta > c \frac{\log{p}}{p}$.
	
	The second case is that $\delta \leq c \frac{\log{p}}{p}$. Using \eqref{eq:fcos} and \eqref{eq:betacos} as above, we find that
	\begin{equation}\label{eq:substitution2} 
		\int_{-\frac{1}{2}}^{\frac{1}{2}} |f(x)|^p \,dx \geq \sqrt{\frac{2}{\pi p}} + O\left(\frac{1}{p^{\frac{3}{2}}}\right). 
	\end{equation}
	Our goal is now to analyze the contribution to the integral of $|f(x)|^p$ for $|x|\geq \frac{1}{2}$. By Bernstein's inequality \eqref{eq:bernstein} in the form $\|f''\|_\infty \leq \pi \|f'\|_\infty$, our assumption that $\|f''\|_\infty = \pi^2-\delta \pi$ implies that
	\[\|f'\|_\infty \geq \pi-\delta.\]
	Let us assume that $x_0$ is a point such that $|f'(x_0)|=\pi-\delta$. Now let $\mathscr{M}_{\delta}$ be the set of those integers $n$ such that $|f(x_0+n+1/2)|\leq (1-2\delta)\|f\|_{\infty} =(1-2\delta)$. From the Riesz interpolation formula \eqref{eq:RIF} and the assumption that $\|f\|_\infty=1$, we find that
	\[\pi-\delta = |f'(x_0)| \leq \frac{4}{\pi}\sum_{n\in \mathbb{Z}} \frac{|f(x_0+n+1/2)|}{(2n+1)^2} \leq \pi - \frac{8\delta}{\pi}\sum_{n \in \mathscr{M}_\delta} \frac{1}{(2n+1)^2},\]
	where we also used that $\sum_{n \in \mathbb{Z}} (2n+1)^{-2} = \frac{\pi^2}{4}$. It follows that $\sum_{n \in \mathscr{M}_\delta} \frac{1}{(2n+1)^2} \leq \frac{\pi}{8}$, and consequently that 
	\begin{equation}\label{eq:complement} 
		\sum_{n \not \in \mathscr{M}_\delta} \frac{1}{(2n+1)^2} \geq \frac{\pi^2}{4} - \frac{\pi}{8} = 2.07470\ldots 
	\end{equation}
	The maximum of $(2n+1)^{-2}$ for integers $n$ is $1$, and this is attained only at $n=0$ and $n=-1$. Hence \eqref{eq:complement} implies that the complement of $\mathscr{M}_{\delta}$ consists of at least three integers. This means that there exists a point $y=x_0+n+1/2$ with $|y|\geq1$ such that $|f(y)|\geq 1-2\delta$. Let $y_0$ denote the closest local maximum of $|f(x)|$ to $y$. There are now two subcases to consider. 
	
	First, if $|y-y_0| \geq \frac{1}{\sqrt{p}}$, then there is an interval $I$ of length $\frac{1}{\sqrt{p}}$ which contains $y$ where $|f(x)| \geq 1-2\delta$ for every $x$ in $I$. Since $|y|\geq1$, we can always choose $p$ so large that $I$ contains no points $x$ with $|x| \leq \frac{1}{2}$. This means that there is no overlap between the integral \eqref{eq:substitution2} and the integral 
	\begin{equation}\label{eq:Iint} 
		\int_I |f(x)|^p \,dx \geq \frac{(1-2\delta)^p}{\sqrt{p}}. 
	\end{equation}
	Since $\delta < c \frac{\log{p}}{p}$, we see that 
	\begin{equation}\label{eq:explog} 
		(1-2\delta)^p \geq \left(1- 2c\frac{\log{p}}{p}\right)^p = \exp\left( -2c \log{p}+O\left(\frac{\log^2{p}}{p}\right)\right) \geq C_1 p^{-2c} 
	\end{equation}
	for an absolute constant $C_1>0$. Inserting \eqref{eq:explog} into \eqref{eq:Iint}, we find that
	\[\int_I |f(x)|^p \,dx \geq C_1 p^{-\frac{1}{2}-2c}.\]
	
	Second, if $|y-y_0| < \frac{1}{\sqrt{p}}$, then we choose $p$ so large that there are no numbers $x$ which satisfy both $|x| \leq \frac{1}{2}$ and $|x-y_0| \leq \frac{1}{2}$, which is possible since $|y|\geq 1$. By Taylor's theorem and Bernstein's inequality \eqref{eq:bernstein} twice, we find that 
	\[|f(x)| \geq (1-2\delta) \left(1 - \frac{\pi^2}{2(1-2\delta)} (x-y_0)^2\right),\]
where we used that $|f(y_0)| \geq (1-2\delta)$.	This shows that
	\[\int_{|x-y_0| \leq \frac{\sqrt{2}}{\pi}(1-2\delta)^{1/2}} |f(x)|^p \,dx \geq \frac{\sqrt{2}}{\pi} (1-2\delta)^{\frac{1}{2}+p} \int_{-1}^1 \left(1-x^2\right)^p\,dx.\]
	By \eqref{eq:betaasymp} and \eqref{eq:explog}, we conclude that
	\[\int_{|x-y_0| \leq \frac{\sqrt{2}}{\pi}(1-2\delta)^{1/2}} |f(x)|^p \,dx \geq C_2 p^{-\frac{1}{2}-2c}\]
	for an absolute constant $C_2>0$. 
	
	Combining what we have found from either subcase with \eqref{eq:substitution2}, we conclude that there is a positive constant $C_3$ such that
	\[\|f\|_p^p \geq \sqrt{\frac{2}{\pi p}} + C_3 p^{-\frac{1}{2}-\min(1,2c)}.\]
	If we choose any $c<1/2$, then the lower bound obtained from the second case $\delta \leq c \frac{\log{p}}{p}$ is larger than the lower bound obtained from the first case $\delta > c \frac{\log{p}}{p}$. Hence we conclude that the lower bound from \eqref{eq:ppest} applies in both cases. This implies that
	\[\mathscr{C}_p \leq \sqrt{\frac{\pi p}{2}} - B \frac{\log p}{\sqrt{p}} + O\left(\frac{1}{\sqrt{p}}\right)\]
	as $p\to \infty$, which gives the stated estimate with $B = \frac{c}{2\pi} \sqrt{\frac{\pi}{2}}$ for any $c<1/2$. 
\end{proof}

For the proof of the lower bound in Theorem~\ref{thm:Cpinfty}, we shall rely on estimates for the functions $g_\alpha$ discussed in Section~\ref{subsec:gamma}.

\begin{proof}[Proof of Theorem~\ref{thm:Cpinfty}: Lower bound] 
	Fix $1<p<\infty$ and let $\alpha> \frac{1}{2}$ be a parameter which depends on $p$ to be chosen later. We need an upper bound for $\|g_\alpha\|_p^p$ as $p\to \infty$. We first recall from Lemma~\ref{lem:galphaest}~(c) that
	\[|g_\alpha(x)| \leq \cos\left(\frac{\pi x}{2\alpha}\right)\]
	whenever $\alpha \geq 1/2$ and $0 \leq x \leq \alpha$. For $x>\alpha$, we deduce from Lemma~\ref{lem:galphaest}~(b) that
	\[|g_\alpha(x)| \leq \left(1+\frac{x}{\alpha}\right)^{-(2\alpha-1)} |\sin{\pi(\alpha-x)}|.\]
	Combining these estimates with the fact that $g_\alpha$ is even yields that 
	\begin{multline*}
		\|g_\alpha\|_p^p \leq \int_{-\alpha}^\alpha \cos^p\left(\frac{\pi x}{2\alpha }\right)\,dx + 2 \sum_{j=1}^\infty \int_{\alpha+j-1}^{\alpha+j} \left(1+\frac{x}{\alpha}\right)^{-(2\alpha-1)p} |\sin{\pi(\alpha-x)}|^p\,dx \\
		\leq \left(2\alpha + 2 \sum_{j=1}^\infty \left(1+\frac{\alpha+j-1}{\alpha}\right)^{-(2\alpha-1)p} \right) \int_{-\frac{1}{2}}^{\frac{1}{2}} (\cos{\pi x})^p \,dx. 
	\end{multline*}
	Choosing $\alpha = \frac{1}{2} + \frac{1}{\log{4}}\frac{\log{p}}{p}$, we find that the total contribution from the sum is $O(1/p)$. Using \eqref{eq:betacos}, we see that
	\[\sqrt{\frac{\pi p}{2}} -\frac{1}{\log{2}}\sqrt{\frac{\pi}{2}} \frac{\log p}{\sqrt{p}} + O\left(\frac{1}{\sqrt{p}}\right) \leq \mathscr{C}_p,\]
	as $p \to \infty$. This implies the stated lower bound with $A > \frac{1}{\log{2}} \sqrt{\frac{\pi}{2}}$. 
\end{proof}

\subsection{Asymptotics as \texorpdfstring{$p\to 0^+$}{p to 0}}  The first goal of the present section is to prove Theorem~\ref{thm:Cp0}~(a). We begin with the following consequence of Jensen's formula. 
\begin{lemma}\label{lem:jensen} 
	Let $f$ be a function in $PW^\infty$ that satisfies $\|f\|_\infty = |f(0)|\neq0$ and let $0<t_1\leq t_2 \leq \cdots $ be the positive zeros of $f$. 
	\begin{enumerate}
		\item[(a)] It holds that $t_n \geq \frac{n}{2e}$. 
		\item[(b)] If additionally $f$ is even, then $t_n \geq \frac{n}{e}$. 
	\end{enumerate}
\end{lemma}
\begin{proof}
	We may assume without loss of generality that $f(0)=1$. If $f$ is entire and does not vanish at the origin and if $r>0$, then Jensen's formula states that 
	\begin{equation}\label{eq:jensen} 
		\log|f(0)| = \sum_{\substack{z \in f^{-1}(\{0\}) \\
		|z| \leq r}} \log\left(\frac{|z|}{r}\right) + \int_0^{2\pi} \log|f(r e^{i\theta})| \, \frac{d\theta}{2\pi}. 
	\end{equation}
	The left-hand side of \eqref{eq:jensen} is $0$ since $f(0)=1$. To bound the integral on the right-hand side, we use the well known estimate $|f(x+iy)| \leq e^{\pi |y|}$, which holds since $f$ is in $PW^\infty$ and $\|f\|_\infty=1$. Computing the resulting integral, we obtain from \eqref{eq:jensen} that
	\[0 \leq \sum_{\substack{z \in f^{-1}(\{0\}) \\
	|z| \leq r}} \log\left(\frac{|z|}{r}\right) + 2r.\]
	If $t_n \geq n/2$, then the statement of (a) follows. If not, then we choose $r=n/2$ to conclude from this that
	\[n \log(n/2)-n \leq \log(t_1 t_2 \cdots t_n) \leq n \log(t_n).\]
	Dividing by $n$ and exponentiating, we arrive at (a). The proof of (b) is similar. If now $t_n \geq n$, then (b) follows. If $t_n<n$, then we choose $r=n$. We use that the zero set is symmetric about $0$ and deduce the asserted estimate from Jensen's formula. 
\end{proof}

The following result is only of interest as $p \to 0^+$. For reasons that will become clear later, we presently make no effort at optimizing the constant.
\begin{lemma}\label{lem:limsupok} 
	For every $0<p<\infty$ it holds that $\mathscr{C}_p \leq \frac{25}{18} p$. 
\end{lemma}
\begin{proof}
	The plan is to use Theorem~\ref{thm:intrep} with $q=p/2$ as in the proof of Theorem~\ref{thm:korevaar}, replacing Theorem~\ref{thm:sep} with Lemma~\ref{lem:jensen} (a). By Lemma~\ref{lem:zeros}, the assumptions of Theorem~\ref{thm:intrep} are satisfied. We may assume without loss of generality that the integral over $(0,\infty)$ is larger than the integral over $(-\infty,0)$. If not, simply replace $f$ by $g(z)=f(-z)$. It follows from this that 
	\begin{equation}\label{eq:precauchy} 
		|f(0)|^{p/2} \leq 2 \sum_{n=0}^\infty \int_{t_n}^{t_{n+1}} |f(x)|^{p/2} \frac{\sin\left(\frac{p}{2} \pi (x-n)\right)}{\pi x} \,dx. 
	\end{equation}
	We now claim that 
	\begin{equation}\label{eq:Kclaim} 
		\sum_{n=0}^\infty \chi_{(t_n,t_{n+1})}(x) \frac{\sin\left(\frac{p}{2} \pi (x-n)\right)}{\pi x} \leq 
		\begin{cases}
			\frac{\sin\left(\frac{p}{2}\pi x\right)}{\pi x}, & \text{if }\, 0 < x \leq \frac{1}{2p}; \\
			\frac{1}{\pi x}, & \text{if }\, \frac{1}{2p} < x < \infty. 
		\end{cases}
	\end{equation}
	Suppose that the claim is true. We then square both sides of \eqref{eq:precauchy}, insert \eqref{eq:Kclaim} and then use the Cauchy--Schwarz inequality to the effect that 
	\begin{multline*}
		|f(0)|^p \leq 4 \|f\|_p^p \left(\int_0^{1/(2p)} \frac{\sin^2\left(\frac{p}{2}\pi x\right)}{(\pi x)^2}\,dx + \int_{1/(2p)}^\infty \frac{1}{(\pi x)^2}\,dx\right) \\
		= 4 \|f\|_p^p \left( \frac{p}{2} \int_0^{1/4} \sinc^2{\pi x}\,dx + \frac{2p}{\pi^2}\right) \leq \frac{25}{18} p \|f\|_p^p, 
	\end{multline*}
	where we in the final estimate used that $\sinc$ is bounded by $1$ and that $\pi > 3$. In the first estimate we tacitly extended the integral of $|f|^p$ to the whole real line. It remains to demonstrate that the claim \eqref{eq:Kclaim} holds. To do this, it is sufficient to prove that if $t_n \leq x \leq \frac{1}{2p}$, then
	\[\sin\left(\frac{p}{2}\pi (x-n)\right) \leq \sin\left(\frac{p}{2} \pi x\right).\]
	By the well known periodicity and monotonicity properties of the sine function, this estimate will follow if we can establish that $n \leq \frac{2}{p}+2x$ whenever $t_n \leq x \leq \frac{1}{2p}$. This follows from Lemma~\ref{lem:jensen}~(a), which gives that
	\[n \leq 2e t_n \leq 2(e-1) t_n + 2x \leq \frac{e-1}{p}+2x < \frac{2}{p} + 2x. \qedhere\]
\end{proof}

We are now ready to proceed with the proof of Theorem~\ref{thm:Cp0}~(a), which relies on the just established Lemma~\ref{lem:limsupok} and on the power trick of Lemma~\ref{lem:powertrick}.

\begin{proof}[Proof of Theorem~\ref{thm:Cp0}~(a)]
	The supremum
	\begin{equation} \label{eq:c0sup}
		c_0 := \sup_{k \geq 1} 2k \mathscr{C}_{1/k}
	\end{equation}
	is finite by Lemma~\ref{lem:limsupok}. It follows from the power trick (Lemma~\ref{lem:powertrick}) that if $k_1$ divides $k_2$, then $2k_1 \mathscr{C}_{1/k_1} \leq 2k_2 \mathscr{C}_{1/k_2}$. Consequently, there is a sequence of integers $(k_j)_{j\geq1}$ which are all strictly greater than $1$ such that if $p_1=1$ and $p_{j+1}=p_j/k_j$ for $j\geq1$, then 
	\[c_0 = \lim_{j \to \infty} \frac{2}{p_j} \mathscr{C}_{p_j}.\]
	Fix $\varepsilon>0$. There is a positive integer $j$ such that if $q = p_j$, then 
	\begin{equation}\label{eq:c0eps} 
		c_0-\varepsilon \leq \frac{2}{q} \mathscr{C}_q. 
	\end{equation}
	Let us now estimate $\frac{2}{p} \mathscr{C}_p$ for $0<p \leq q/2$. We first let $k$ be the smallest positive integer such that $k \geq q/p$. Using that $p \mapsto \mathscr{C}_p$ is increasing, the power trick, and the lower bound \eqref{eq:c0eps}, we find that
	\[\frac{2}{p} \mathscr{C}_p \geq \frac{2}{p} \mathscr{C}_{q/k} \geq \frac{2}{p} \frac{\mathscr{C}_q}{k} \geq (c_0-\varepsilon) \frac{q}{kp} \geq (c_0-\varepsilon) \frac{q}{q+p},\]
	where we in the final inequality used that $k \leq q/p+1$. This shows that 
	\begin{equation}\label{eq:liminf} 
		\liminf_{p \to 0^+} \frac{2}{p} \mathscr{C}_p \geq (c_0-\varepsilon). 
	\end{equation}
	We next let $k$ be the largest positive integer such that $k \leq q/p$. Using that $p \mapsto \mathscr{C}_p$ is increasing in combination with \eqref{eq:c0sup}, we find that
	\[\frac{2}{p} \mathscr{C}_p \leq \frac{2}{p}\mathscr{C}_{q/k} \leq \frac{q}{pk} c_0 \leq c_0 \frac{q}{q-p},\]
	where we in the final inequality used that $k \geq q/p-1$. This shows that 
	\begin{equation}\label{eq:limsup} 
		\limsup_{p \to 0^+} \frac{2}{p} \mathscr{C}_p \leq c_0. 
	\end{equation}
	The claim follows from \eqref{eq:liminf} and \eqref{eq:limsup} and the fact that $\varepsilon>0$ was arbitrary. 
\end{proof}

Our next goal is to estimate precisely the constant $c_0$ of Theorem~\ref{thm:Cp0}~(a). From Lemma~\ref{lem:limsupok} it follows that $c_0 \leq \frac{25}{9}$, while the lower bound $c_0\geq1$ can be deduced by considering positive integers $k$ and
\[f(x) = \sinc^{2k}\left(\frac{\pi}{2k}x\right)\]
as a function in $PW^{1/k}$. We do not know if there exists an even extremal function in $PW^p$ when $0<p<1$, and this fact caused us some extra work in Lemma~\ref{lem:jensen} and Lemma~\ref{lem:limsupok} above. However, to compute the constant $c_0$, it suffices to restrict to even functions. Let us therefore introduce
\[\frac{1}{\mathscr{C}_{p,\operatorname{e}}} := \inf_{f \in PW^p} \left\{\|f\|_p^p \,:\, f(0) = 1 \,\text{ and }\, f \text{ is even}\right\}\]
and establish the following consequence of Theorem~\ref{thm:Cp0}~(a).
\begin{corollary}\label{cor:evenok} 
	Let $c_0$ be the constant appearing in Theorem~\ref{thm:Cp0}~(a). Then
	\[\lim_{p\to 0^+} \frac{2}{p} \mathscr{C}_{p,\operatorname{e}} = c_0.\]
\end{corollary}
\begin{proof}
	We trivially have $\mathscr{C}_{p,\operatorname{e}} \leq \mathscr{C}_p$. Suppose that $\varphi_{2p}$ is an extremal for $\mathscr{C}_{2p}$. Consider the even function $f(x) := \varphi_{2p}(x/2) \varphi_{2p}(-x/2)$, which is in $PW^p$ and which satisfies $f(0)=1$. Moreover, by the Cauchy--Schwarz inequality, we find that
	\[\|f\|_p^p = \int_{-\infty}^\infty |f(x)|^p \,dx \leq \int_{-\infty}^\infty |\varphi_{2p}(x/2)|^{2p} \,dx = 2\|\varphi_{2p}\|_{2p}^{2p}.\]
	This implies that
	\[\mathscr{C}_{p,\operatorname{e}} \geq \frac{|f(0)|^p}{\|f\|_p^p} \geq \frac{|\varphi_{2p}(0)|^{2p}}{2\|\varphi_{2p}\|_{2p}^{2p}} = \frac{\mathscr{C}_{2p}}{2}.\]
	Dividing by $p$ and taking the limit $p\to 0^+$ in the estimates $\mathscr{C}_{2p} \leq 2 \mathscr{C}_{p,\operatorname{e}} \leq 2 \mathscr{C}_p$, then appealing to Theorem~\ref{thm:Cp0}~(a) implies the stated result. 
\end{proof}

Since Corollary~\ref{cor:evenok} allows us to restrict our attention to even functions, we are able to improve the argument used in Lemma~\ref{lem:limsupok} in several ways to obtain the following result. 
\begin{theorem}\label{thm:Cp0upper} 
	Let $c_0$ be the constant appearing in Theorem~\ref{thm:Cp0}~(a). Then
	\[c_0 \leq \inf_{q > 1} \frac{2^q}{q} \left(\int_0^{1/2} (\sinc{\pi x})^{q^\ast} \,dx + \frac{2^{q^\ast-1}}{\pi^{q^\ast}(q^\ast-1)}\right)^{q-1},\]
	where $q^\ast$ denotes the H\"older conjugate of $q$. 
\end{theorem}
\begin{proof}
	By Corollary~\ref{cor:evenok} we consider an even function in $PW^p$ which does not vanish at the origin and whose zeros are all real. We fix $q>1$ and use Theorem~\ref{thm:intrep} with $q=p/q$ to obtain 
	\begin{equation}\label{eq:preholder} 
		|f(0)|^{p/q} = 2 \sum_{n=0}^\infty \int_{t_n}^{t_{n+1}} |f(x)|^{p/q} \frac{\sin\big(\frac{p}{q} \pi (x-n)\big)}{\pi x} \,dx. 
	\end{equation}
	As in the proof of Lemma~\ref{lem:limsupok}, we now claim that 
	\begin{equation}\label{eq:Kbetaclaim} 
		\sum_{n=0}^\infty \chi_{(t_n,t_{n+1})}(x) \frac{\sin\big(\frac{p}{q} \pi (x-n)\big)}{\pi x} \leq 
		\begin{cases}
			\frac{\sin(\frac{p}{q}\pi x)}{\pi x}, & \text{if }\, 0 < x \leq \frac{q}{2p}; \\
			\frac{1}{\pi x}, & \text{if }\, \frac{q}{2p} < x < \infty. 
		\end{cases}
	\end{equation}
	To establish \eqref{eq:Kbetaclaim}, it is sufficient to prove that $n \leq \frac{q}{p} + 2x$ whenever $t_n \leq x \leq \frac{q}{2p}$. In this case, we can use Lemma~\ref{lem:jensen}~(b) to prove that
	\[n \leq e t_n \leq (e-2) t_n + 2x \leq (e-2) \frac{q}{2p} +2x < \frac{q}{p} + 2x.\]
	We next raise \eqref{eq:preholder} to the power $q$, then apply \eqref{eq:Kbetaclaim} and H\"older's inequality on the right-hand side to conclude that 
	\begin{multline*}
		|f(0)|^p \leq 2^{q-1} \|f\|_p^p \left(\int_0^{q/(2p)} \left(\frac{\sin(\frac{p}{q}\pi x)}{\pi x}\right)^{q^\ast}\,dx + \int_{q/(2p)}^\infty (\pi x)^{-q^\ast}\,dx\right)^{\frac{q}{q^\ast}} \\
		= \frac{p}{2} \|f\|_p^p \frac{2^q}{q} \left(\int_0^{1/2} (\sinc{\pi x})^{q^\ast}\,dx + \frac{2^{q^\ast-1}}{\pi^{q^\ast}(q^\ast-1)}\right)^{q-1}, 
	\end{multline*}
	which implies the stated result. 
\end{proof}

Optimizing in $q$, using the \texttt{integrate} and \texttt{optimize} packages from SciPy, we found that the choice $q=1.784$ in Theorem~\ref{thm:Cp0upper} yields the upper bound
\[c_0 \leq 1.1481785,\]
which is the upper bound for $c_0$ stated in Theorem~\ref{thm:Cp0}~(b).

To obtain a lower bound for $c_0$, we will test \eqref{eq:pointeval} with the functions $g_\alpha$ discussed in Section~\ref{subsec:gamma}.
\begin{theorem}\label{thm:Cp0lower} 
	Let $c_0$ be the constant appearing in Theorem~\ref{thm:Cp0}~(a). Then
	\[\frac{1}{c_0} \leq \inf_{\gamma>1/2} \gamma\left(\int_0^1 \frac{\,dx}{(1-x)^{\gamma(1-x)}(1+x)^{\gamma(1+x)}}+\int_1^\infty \frac{dx}{(x-1)^{\gamma(1-x)}(1+x)^{\gamma(1+x)}}\right).\]
\end{theorem}
\begin{proof}
	We will choose $\alpha=\alpha(p)\geq \frac{1}{2}$ later, obtaining a lower bound for $c_0$ as follows
	\[\frac{1}{c_0} \leq \liminf_{p \to 0^+} \frac{p}{2} \|g_\alpha\|_p^p.\]
	Since $g_\alpha$ is even, it is sufficient to consider the $L^p$ integral over $x\geq0$. We will obtain an upper bound for this integral using Lemma~\ref{lem:galphaest}, so we split it at $x=\alpha$. For the first part, we use Lemma~\ref{lem:galphaest}~(a) to the effect that 
	\begin{align*}
		\int_0^\alpha |g_\alpha(x)|^p \,dx &\leq \int_0^\alpha \left(1-\frac{x}{\alpha}\right)^{-(\alpha-1/2-x)p} \left(1+\frac{x}{\alpha}\right)^{-(\alpha-1/2+x)p}\,dx \\
		&= \alpha \int_0^1 \frac{dx}{(1-x)^{(\alpha-1/2-\alpha x)p}(1+x)^{(\alpha-1/2+\alpha x)p}} \\
		&\leq \alpha \int_0^1 \frac{dx}{(1-x)^{(\alpha-1/2)p (1-x)}(1+x)^{(\alpha-1/2)p(1+x)}}, 
	\end{align*}
	where we in the final inequality used that $(1-x)^{px/2} (1+x)^{-px/2} \leq 1$. Note that this inequality is attained in the limit $p\to0^+$. For the second part of the integral, we deduce from Lemma~\ref{lem:galphaest}~(b) and $|\sin{\pi(\alpha-x)}|\leq1$ that 
	\begin{align*}
		\int_\alpha^\infty |g_\alpha(x)|^p \,dx &\leq \int_\alpha^\infty \alpha^{p(2\alpha-1)} \frac{(1-\alpha+x)^{(1/2-\alpha+x)p}}{(\alpha+x)^{(\alpha-1/2+x)p}} \,dx \\
		&=\alpha \int_1^\infty \frac{\left(1-\frac{\alpha-1}{\alpha x} \right)^{(1/2-\alpha+\alpha x)p}}{\left(1+\frac{1}{x}\right)^{(\alpha-1/2+\alpha x)p}} \,\frac{dx}{x^{{p(2\alpha-1)}}} \\
		&\leq \alpha \int_1^\infty \frac{\left(1-\frac{\alpha-1}{\alpha x} \right)^{(1/2-\alpha)p(1-x)}}{\left(1+\frac{1}{x}\right)^{(\alpha-1/2)p(1+x)}}\frac{dx}{x^{{p(2\alpha-1)}}}, 
	\end{align*}
	where we in the final inequality used that $\left(1-\frac{\alpha-1}{\alpha x}\right)^{px/2} \left(1+\frac{1}{x}\right)^{-px/2} \leq 1$. We will now choose $\alpha = \frac{1}{2}+\frac{\gamma}{p}$ for some fixed $\gamma>1/2$, which ensures that $g_\alpha$ is in $PW^p$. With this choice of $\alpha$, the upper bound for the first part of the integral becomes
	\[p\int_0^\alpha |g_\alpha(x)|^p \,dx \leq \left(\frac{p}{2}+\gamma\right)\int_0^1 \frac{\,dx}{(1-x)^{\gamma(1-x)}(1+x)^{\gamma(1+x)}}.\]
	Taking the limit $p\to 0^+$, we obtain the first part contribution to the stated upper bound. The choice of $\alpha$ means that our upper bound for the second part of the integral becomes
	\[p\int_\alpha^\infty |g_\alpha(x)|^p \,dx \leq \left(\frac{p}{2}+\gamma\right) \int_1^\infty \frac{dx}{\left(x-\frac{2\gamma-p}{2\gamma+p}\right)^{\gamma(1-x)}\left(1+x\right)^{\gamma(1+x)}}.\]
	Taking the limit $p\to0^+$, we obtain the second part of the contribution to the stated upper bound. 
\end{proof}

Optimizing the parameter $\gamma$ of Theorem~\ref{thm:Cp0lower} numerically using the \texttt{integrate} and \texttt{optimize} packages from SciPy, we find that $\gamma=0.935$ gives the lower bound
\[c_0 \geq 1.1393830,\]
which provides the lower bound for $c_0$ stated in Theorem~\ref{thm:Cp0}~(b).

\pagebreak 

\section{Epilogue: Duality and orthogonality revisited} \label{sec:epilogue}

This section presents some afterthoughts on the notions of duality and orthogonality, as discussed and employed in the preceding treatise. The first question to be addressed, is whether Theorem~\ref{thm:prpk}, which seems to be tied to the usual duality between $L^p$ and $L^{p/(1-p)}$, may nevertheless carry over to the range $0<p<1$. The second question is whether there is a natural Hilbert space induced by the orthogonality relations of Lemma~\ref{lem:orthogonality} (b). An interesting point is that both questions are intimately related to the same question: What can be said about the decay of the extremal functions $|\varphi(x)|$ when $x\to \infty$ and $x$ is bounded away from the zeros of $\varphi$? 

\subsection{Traces of duality} \label{subsec:traces} We begin with the problem of extending the reproducing formula of Theorem~\ref{thm:prpk} to the range $0<p<1$. Since H\"{o}lder's inequality is no longer available, we may not expect the formula to hold for all functions $f$ in $PW^p$. Instead, we will aim for a formula that applies to functions in $PW^2$ that belong to the Schwartz space $\mathscr{S}$ on $\mathbb{R}$.

We will rely on Theorem~\ref{thm:zeroset} (b), and we will therefore only succeed in making the desired extension in the range $1/2\leq p<1$. Retaining the notation $\Phi = |\varphi|^{p-2} \varphi/\| \varphi \|_p^p$ from \eqref{eq:holderfriend} in Section~\ref{subsec:convex}, we may state our next result as follows. 
\begin{theorem}\label{thm:traces} 
	Suppose that $1/2\le p<1$ and let $\varphi$ be a solution of the extremal problem \eqref{eq:extremalproblem}. Then $\Phi$ is a tempered distribution and 
	\begin{equation}\label{eq:dualp} 
		f(0)=\int_{-\infty}^{\infty} f(x) \, \Phi(x) dx 
	\end{equation}
	for every $f$ in $\mathscr{S} \cap PW^2$. 
\end{theorem}
We will use the following two lemmas, the first of which is standard. Its proof relies on a classical argument of Plancherel and P\'{o}lya \cite{PP37} and follows from \cite[Theorem~6.7.15]{Boas54}, either by using that the differentiation operator acts boundedly on $PW^p$ or by employing a Cauchy estimate in the proof of \cite[Theorem~6.7.15]{Boas54}.
\begin{lemma}\label{lem:PP} 
	Let $\Lambda$ be a uniformly discrete set of real numbers and assume that $0<p<\infty$. Then there exists a constant $C$ depending only on the separation constant of $\Lambda$ and on $p$ such that
	\[ \sum_{\lambda\in \Lambda} |f'(\lambda)|^p \leq C \| f \|_p^p \]
	holds for every $f$ in $PW^p$. 
\end{lemma}
\begin{lemma}\label{lem:upper} 
	Fix $1/2\leq p<\infty$ and suppose that $\varphi$ is a solution of the extremal problem \eqref{eq:extremalproblem}. Then there exist positive constants $C$ and $\gamma$ such that
	\[ |\varphi(x)| \geq C \frac{\dist (x,\mathscr{Z}(\varphi))}{(1+|x|)^{\gamma}}\]
	for all $x$ in $\mathbb{R}$. 
\end{lemma}
\begin{proof}
	By symmetry, it is sufficient to establish the estimate for $x\geq0$. Let $\mathscr{Z}(\varphi)=(t_n)_{n \in \mathbb{Z}\setminus\{0\}}$ be the zero set of $\varphi$. For $n\geq1$, we set $\delta_n := t_{n+1}-t_n$. Since $\varphi$ has exponential type $\pi$, we know that $t_n/n \to 1$ as $n \to \infty$. It follows that there is some $N$ such that $\delta_n \leq t_n$ and $t_{n+2} \leq 2 t_n$ whenever $n \geq N$. We restrict our attention to these $n$ in what follows. Setting
	\[\psi_n(x) := \frac{\varphi(x)}{(x-t_n)(x-{t_{n+1}})} \quad \text{and} \quad \Psi_n(x) := |\psi_n(x)|^p \frac{x^2}{|(x-t_n)(x-t_{n+1})|^{1-p}},\]
	we appeal to Lemma~\ref{lem:orthogonality}~(b) as before to conclude that
	\[\int_{\mathbb{R}\setminus I_n} \Psi_n(x)\,dx = \int_{I_n} \Psi_n(x)\,dx,\]
	where $I_n := [t_n,t_{n+1}]$. We estimate the right-hand in the now familiar fashion,
	\[\int_{I_n} \Psi_n(x)\,dx \leq \max_{x \in I_n} |\psi_n(x)|^p t_{n+1}^2 |I_n|^{2p-1} \B(p,p) \leq \max_{x \in I_n} |\psi_n(x)|^p t_{n+1}^{2p+1} \B(p,p),\]
	where we used that $\delta_n \leq t_n \leq t_{n+1}$ for $n\geq N$ in the final estimate. To bound the left-hand side from below, we use Lemma~\ref{lem:hormander} along with the fact that $\|\varphi\|_\infty = \varphi(0) = 1$ to conclude that
	\[\int_{\mathbb{R}\setminus I_n} \Psi_n(x)\,dx \geq \int_0^{\frac{1}{3}} \Psi_n(x)\,dx = \int_0^{\frac{1}{3}} \frac{|\varphi(x)|^p x^2}{(t_n-x)(t_{n+1}-x)}\,dx \geq \frac{1}{2^p} \frac{t_{n+1}^{-2}}{3^4}.\]
	Combining what we have done, we find that 
	\begin{equation}\label{eq:maxest} 
		\max_{x \in I_n} |\psi_n(x)| \geq \frac{c}{t_{n+1}^{2+3/p}} \qquad \text{for} \qquad c = \frac{1}{2(3^4 \B(p,p))^{1/p}}, 
	\end{equation}
	whenever $n \geq N$. To parlay this maximal estimate into a pointwise estimate, we will combine it with that of a neighboring interval. The function $\psi_n$ has a unique critical point on the interval $(t_{n-1},t_{n+2})$ by Lemma~\ref{lem:monotonicity}, where we use the convention $t_0=0$ as usual. Since $I_n$ is a subinterval of that interval, is follows from this that the maximum of $|\psi_n|$ on $I_n$ will be attained at a unique point which we denote by $\xi_n$. We now set
	\[\widetilde{\psi}_n(x) := \frac{\varphi(x)}{(x-t_n)(x-t_{n+1})(x-t_{n+2})}.\]
	A similar argument shows that the minimum of $|\widetilde{\psi}_n|$ on the interval $[\xi_n,\xi_{n+1}]$, which is a subinterval of $(t_{n-1},t_{n+3})$, must be attained at one of the endpoints. Combining this with \eqref{eq:maxest}, we find that 
	\begin{equation}\label{eq:tildepsij} 
		|\widetilde{\psi}_n(x)| \geq \min\left(\frac{|\psi_n(\xi_n)|}{t_{n+2}-\xi_n},\frac{|\psi_{n+1}(\xi_{n+1})|}{\xi_{n+1}-t_n}\right) \geq \frac{c}{t_{n+2}^{3+3/p}} \geq \frac{c 2^{-3-3/p}}{x^{3+3/p}} 
	\end{equation}
	for $\xi_n \leq x \leq \xi_{n+1}$, where we in the final estimate used that $t_{n+2} \leq 2 t_n$ for $n \geq N$. We next estimate
	\[\frac{|x-t_n|\,|x-t_{n+1}|\,|x-t_{n+2}|}{\dist(x,\mathscr{Z}(\varphi))} \geq \frac{1}{2}\inf_{n\geq1} (t_{n+1}-t_n)^2 = \frac{\sigma^2}{2},\]
	and recall from Theorem~\ref{thm:zeroset}~(b) that $\sigma>0$. Combined with \eqref{eq:tildepsij}, this yields the stated estimate with
	\[C = \frac{2^{-5-3/p} \sigma^2}{(3^4 \B(p,p))^{1/p}} \qquad \text{and} \qquad \gamma = 3+\frac{3}{p}\]
	for $x\geq \xi_n$ and $n \geq N$. We adjust $C$ if necessary to take into account $0 \leq x \leq \xi_N$. 
\end{proof}

The assumption that $p\geq 1/2$ is only required in the argument above to conclude that $\mathscr{Z}(\varphi)$ is uniformly discrete by Theorem~\ref{thm:zeroset}~(b). If $p < 1/2$ and if $\mathscr{Z}(\varphi)$ is uniformly discrete, then $|I_n|^{2p-1}$ is bounded above by a constant, so in fact we get a better exponent $\gamma$.

Using Theorem~\ref{thm:zeroset}~(c) twice in the argument above, we obtain the sharper bound 
\begin{equation}\label{eq:presharper} 
	|\varphi_p(x)| \geq C \frac{\dist(x,\mathscr{Z}(\varphi_p))}{(1+|x|)^{4/p}} 
\end{equation}
for $1 \leq p < \infty$. Taking also into account Corollary~\ref{cor:x2infty}, we find that 
\begin{equation}\label{eq:sharper} 
	|\varphi_p(x) | = \frac{\dist (x,\mathscr{Z}(\varphi_p))}{o\big((1+|x|)^{4/p}\big)} 
\end{equation}
as $x \to \infty$. In the next subsection, we will improve this restriction on the decay of $|\varphi_p|$ when $p>1$.

\begin{proof}[Proof of Theorem~\ref{thm:traces}] 
	We begin by noting that $\Phi$ is a tempered distribution because
	\[\int_{y}^{y+1} |\varphi(x)|^{p-1}\, dx = O\big(|y|^{(1-p)\gamma}\big)\]
	by Theorem~\ref{thm:zeroset}~(b) and Lemma~\ref{lem:upper}.
	
	To establish \eqref{eq:dualp}, we begin by picking an arbitrary function $f$ in $\mathscr{S} \cap PW^2$. We may assume that $f$ is real entire since $PW^p$ is closed under complex conjugation. We define
	\[ F(\varepsilon):= \int_{-\infty}^{\infty} \left|\varphi(x)+\varepsilon \big(f(x)-f(0)\varphi(x)\big) \right|^p dx. \]
	Since $F(\varepsilon)\geq F(0)$ for every $\varepsilon$ by the extremality of $\varphi$, it suffices to show that $F$ is differentiable at $0$ with 
	\begin{equation}\label{eq:Fder} 
		F'(0)=p \int_{-\infty}^{\infty} f(x) |\varphi(x)|^{p-2}\varphi(x) \,dx-p f(0) \| \varphi\|_p^p . 
	\end{equation}
	To this end, we note initially that 
	\begin{multline*}
		F(\varepsilon) = (1-\varepsilon f(0))^p \int_{-\infty}^{\infty} \left|\varphi(x)+\frac{\varepsilon f(x)}{1-\varepsilon f(0)}\right|^p \,dx \\
		= \int_{-\infty}^{\infty} \left|\varphi(x)+\frac{\varepsilon f(x)}{1-\varepsilon f(0)}\right|^p dx - p \varepsilon f(0) \|\varphi\|_p^p + O(\varepsilon^{p+1}), 
	\end{multline*}
	Here we used that
	\[ |a|^p-|b|^p \leq |a+b|^p\leq |a|^p+|b|^p \]
	holds for arbitrary complex numbers $a$ and $b$ when $0<p\leq 1$ to get the term $O(\varepsilon^{p+1})$. We now use Lemma~\ref{lem:PP} to obtain
	\[F(\varepsilon)=\int_{\dist(x,\mathscr{Z}(\varphi))\geq \varepsilon} \left|\varphi(x)+\frac{\varepsilon f(x)}{1-\varepsilon f(0)} \right|^p \,dx - p \varepsilon f(0) \|\varphi\|_p^p+O(\varepsilon^{p+1}).\]
	This is justified because if $|x-t_n|\leq \varepsilon$, then $|\varphi(x)|\leq \varepsilon |\varphi'(\lambda)|$ for some $\lambda$ between $t_n$ and $x$. Using Lemma~\ref{lem:upper} and our assumption that $f$ is in $\mathscr{S}$, we infer from this that 
	\begin{align*}
		F(\varepsilon) & =\int_{\dist(x,\mathscr{Z}(\varphi))\geq \varepsilon} |\varphi(x)|^{p} \left|1+\varepsilon \frac{f(x)}{(1-\varepsilon f(0))\varphi(x)} \right|^p \,dx \\
		& \qquad\qquad\qquad\qquad\qquad\qquad\qquad\qquad\qquad - p \varepsilon f(0) \|\varphi\|_p^p + O(\varepsilon^{p+1}) \\
		& = \int_{\dist(x,\mathscr{Z}(\varphi))\geq \varepsilon} \left( |\varphi(x)|^{p} +p \varepsilon f(x) |\varphi(x)|^{p-2} \varphi(x) \right)\,dx \\
		& \qquad\qquad\qquad\qquad\qquad\qquad\qquad\qquad\qquad -p \varepsilon f(0) \|\varphi\|_p^p+O(\varepsilon^{p+1}). 
	\end{align*}
	Now using Lemma~\ref{lem:PP} and Lemma~\ref{lem:upper} a second time, we find that
	\[ F(\varepsilon)=\int_{-\infty}^{\infty} \left(|\varphi(x)|^{p}+p \varepsilon f(x) |\varphi(x)|^{p-2} \varphi(x) \right)\,dx -p \varepsilon f(0) \|\varphi\|_p^p+O(\varepsilon^{p+1}), \]
	which yields \eqref{eq:Fder}. 
\end{proof}

\subsection{A family of de Branges spaces} \label{subsec:deBranges} We turn to our investigation of the Hilbert space structure induced by the orthogonality relation of Lemma~\ref{lem:orthogonality} (b). This will lead us to a de Branges space associated with the extremal functions $\varphi_p$ in the strictly convex range. An immediate application of this study will be an improvement of \eqref{eq:sharper} for $p>1$. 

We begin by recalling (from e.g.~\cite[Problem 50]{deBranges68}) that a Hilbert space $H$ of entire functions is a de Branges space if the following conditions are met: 
\begin{enumerate}
	\item[(H1)] Whenever $f$ is in $H$ and has a nonreal zero $w$, the function $f(z)(z-\overline{w})/(z-w)$ is in $H$ and has the same norm as $f$. 
	\item[(H2)] For every nonreal $w$, the linear functional defined on $H$ by $f \mapsto f(w)$ is continuous. 
	\item[(H3)] The function $f^\ast(z) := \overline{f (\overline{z})}$ belongs to $H$ whenever $f$ is in $H$, and $f^\ast$ has the same norm as $f$. 
\end{enumerate}
For $1 \leq p < \infty$, let $m_p$ denote the measure defined on $\mathbb{R}$ by
\[ dm_p(x) = |\varphi_p(x)|^{p-2} \,dx\]
and let $L^2(m_p)$ denote the corresponding $L^2$ space of measurable functions on $\mathbb{R}$. We furthermore declare $B_p$ to be the closure of $PW^p \cap L^2(m_p)$ in $L^2(m_p)$ and endow $B_p$ with the norm and inner product of $L^2(m_p)$. It is clear that $\varphi_p$ is in $B_p$, since 
\begin{equation}\label{eq:varphiinBp} 
	\|\varphi_p\|_{B_p}^2 = \int_{-\infty}^\infty |\varphi_p(x)|^2 \,dm_p(x) = \int_{-\infty}^\infty |\varphi_p(x)|^p \,dx = \|\varphi_p\|_p^p. 
\end{equation}
We begin with the following generalization of Corollary~\ref{cor:x2infty} which will be important in our investigations of $B_p$.
\begin{lemma}\label{lem:zerotype} 
	If $1 \leq p < \infty$ and if $\omega$ is a nonconstant entire function of $0$ exponential type, then 
	\begin{equation}\label{eq:unbounded} 
		\int_{-\infty}^{\infty} |\omega(x)|^2 |\varphi_p(x)|^p dx = \infty. 
	\end{equation}
\end{lemma}
\begin{proof}
	If $\omega$ is a polynomial, the statement follows at once from Corollary~\ref{cor:x2infty}. We will therefore assume that $\omega$ is an entire function of $0$ exponential type which is not a polynomial and that the integral in \eqref{eq:unbounded} is finite. Our goal is to show that this premise leads to a contradiction. 
	
	By \eqref{eq:presharper}, our assumption implies that
	\[\int_{-\infty}^{\infty} |\omega(x)|^2 \frac{\dist^p(x,\mathscr{Z}(\varphi_p))}{(1+|x|)^4}\, dx < \infty. \]
	Since $\omega$ is not a polynomial, it has an infinite number of zeros. Dividing out two zeros of $\omega$, we may instead assume that 
	\begin{equation}\label{eq:intcond} 
		\int_{-\infty}^{\infty} |\omega(x)|^2 \dist^p(x,\mathscr{Z}(\varphi_p))\, dx < \infty. 
	\end{equation}
	By Theorem~\ref{thm:zeroset}~(b), we know that $\mathscr{Z}(\varphi_p)$ has a positive separation constant $\sigma$. Consider the sequence of intervals $I_n:=[\sigma n, \sigma (n+1)]$ for $n\geq1$, and let $x_n$ be a point in $I_{2n}$ such that
	\[|\omega(x_n)|=\max\big\{|\omega(x)|\,:\, x\in I_{2n} \,\text{ and }\, \dist(x,\mathscr{Z}(\varphi_p))\geq \sigma/2\big\}.\]
	The condition \eqref{eq:intcond} implies that $\sup_{n\geq1} |\omega(x_n)|<\infty$. Now a classical theorem of Duffin and Schaeffer (see \cite{DS45} or \cite[Corollary~10.5.4]{Boas54}) implies that $\omega$ is bounded on $\mathbb{R}$. Since $\omega$ is assumed to be of $0$ exponential type, this implies that $\omega$ is a constant function. This contradicts our assumption that $\omega$ is not a polynomial. 
\end{proof}

By the singularities of the weight $|\varphi_1|^{-1}$, every function in $B_1$ must be of the form $\omega\varphi_1$ with $\omega$ of $0$ exponential type, and so $B_1$ consists only of scalar multiples of $\varphi_1$. We therefore restrict our attention to the nontrivial case $1<p<\infty$ in what follows. To state our next theorem, we set $e_0(z) := \varphi_p(z)$ and
\[e_n(z) = \frac{z}{z - t_n} \varphi_p(z)\]
for $n$ in $\mathbb{Z} \setminus\{0\}$ with the convention that $t_n = -t_{-n}$ in view of Lemma~\ref{lem:unique}. For $p>1$, it is easy to check that $e_n$ is in $B_p$ for every $n$ as in \eqref{eq:varphiinBp}.
\begin{theorem}\label{thm:deBranges} 
	If $1<p<\infty$, then $B_p$ is a de Branges space with orthogonal basis $\mathscr{B} = (e_n)_{n\in\mathbb{Z}}$. 
\end{theorem}
\begin{proof}
	We have already seen that $\mathscr{B}$ is a subset of $B_p$. It follows at once from Lemma~\ref{lem:orthogonality}~(a) that if $n\neq0$, then $\langle e_n, e_0 \rangle_{B_p}=0$. When $m \neq n$ and neither $m$ nor $n$ is zero, we get that $\langle e_m, e_n \rangle_{B_p} = 0$ directly from Lemma~\ref{lem:orthogonality} (b). It follows that $\mathscr{B}$ is an orthogonal basis for a subspace of $B_p$. 
	
	To show that this subspace consists of entire functions of exponential type $\pi$, it suffices to show that 
	\begin{equation}\label{eq:sumbasis} 
		\sum_{n=1}^{\infty} \frac{1}{\| e_n\|_{B_p}^2 t_n^2}<\infty. 
	\end{equation}
	Indeed, writing
	\[ S_N(x):=\sum_{n=-N}^N a_n \frac{e_n(x)}{\| e_n \|_{B_p}} = x \varphi_p (x) \sum_{n=-N}^N \frac{a_n}{(x-t_n)\|e_n\|_{B_p}}\]
	for an arbitrary $\ell^2$ sequence $(a_n)_{n \in \mathbb{Z}}$, we see that \eqref{eq:sumbasis} ensures that the sequence $S_N$ converges uniformly on compact subsets of $\mathbb{C}$ and that the limit function will be of exponential type $\pi$. (Here we tacitly used the fact that $t_{-n}=-t_n$.) 
	
	We may in fact give an explicit bound for the sum of the series in \eqref{eq:sumbasis}. To this end, setting
	\[f_N(x):=\sum_{n=1}^N \frac{1}{\| e_n\|_{B_p} t_n}\, \frac{e_n(x)}{\| e_n \|_{B_p}}, \]
	we have 
	\begin{equation}\label{eq:Bpnorm} 
		\|f_N\|_{B_p}^2=\sum_{n=1}^N \frac{1}{\|e_n\|_{B_p}^2 t_n^2}. 
	\end{equation}
	On the other hand, when $0\leq x \leq t_1$, we also have
	\[|f_N(x)| =x \varphi_p(x) \sum_{n=1}^N \frac{1}{\| e_n \|_{B_p}^2t_n(t_n-x)} \geq x \varphi_p(x) \sum_{n=1}^N \frac{1}{\| e_n \|_{B_p}^2 t_n^2},\]
	which implies that
	\[\|f_N\|_{B_p}^2\geq \left(\sum_{n=1}^N \frac{1}{\| e_n \|_{B_p}^2 t_n^2}\right)^2 \int_0^{t_1} x^2 |\varphi_p(x)|^p \,dx.\]
	Combining this inequality with \eqref{eq:Bpnorm}, we conclude that
	\[ \sum_{n=1}^{\infty} \frac{1}{\|e_n\|_{B_p}^2 t_n^2} \leq \left(\int_0^{t_1} x^2 |\varphi_p(x)|^p\, dx \right)^{-1}. \]
	
	We prove next that the sequence $\mathscr{B}$ is complete in $B_p$. We begin by recalling that the zeros of $\varphi_p$ are simple, which means that $e_n(t_n)\neq 0$. Let $f$ be any function in $PW^p\cap L^2(m_p)$ and set
	\[g(z) := \left(f(z)-\frac{f(t_n)}{e_n(t_n)}e_n(z)\right) \frac{z}{z-t_n}.\]
	Since $g(0)=0$, Theorem~\ref{thm:prpk} yields
	\[\langle f, e_n \rangle_{B_p} = \langle g, e_{0} \rangle_{B_p} + \frac{f(t_n)}{e_n(t_n)} \|e_n\|_{B_p}^2 = \frac{f(t_n)}{e_n(t_n)} \|e_n\|_{B_p}^2.\]
	It follows that
	\[h:=f-\sum_{n=-\infty}^{\infty} \frac{\langle f, e_n \rangle_{B_p}}{\| e_n\|_{B_p}^2} e_n \]
	is an entire function of exponential type $\pi$ that vanishes at the zeros of $z\varphi_p(z)$, whence $h(z)=z \varphi_p(z) \omega(z)$ for an entire function $\omega$ of $0$ exponential type. By Lemma~\ref{lem:zerotype}, we must have $\omega\equiv 0$, and so $f$ must lie in the closure of the span of $\mathscr{B}$. This means that $B_p$ is a Hilbert space of entire functions.
	
	To finish the proof, we note that the axioms (H1) and (H3) are trivially satisfied, while (H2) holds because the functional of point evaluation at the point $w$ in $\mathbb{C}$ is
	\[ \left(\sum_{n=-\infty}^{\infty} \frac{|e_n(w)|^2}{\| e_n\|_{B_p}^2} \right)^{1/2},\]
	which is finite by \eqref{eq:sumbasis}. 
\end{proof}

The bound \eqref{eq:sumbasis} from the proof of Theorem~\ref{thm:deBranges} yields an interesting improvement of \eqref{eq:sharper}. Indeed, using Lemma~\ref{lem:orthogonality}~(b), Lemma~\ref{lem:midpoint} and Theorem~\ref{thm:zeroset}~(c), we find that $\|e_n\|_{B_p}^2$ is bounded above and below by an absolute constant times $|\varphi_p(\mu_n)|^p$. Hence \eqref{eq:sumbasis} yields the following result. 
\begin{theorem}\label{thm:midpointest} 
	Fix $1 < p< \infty$ and let $\varphi_p$ be the solution of \eqref{eq:extremalproblem}. Then 
	\begin{equation}\label{eq:strong} 
		\sum_{n=1}^\infty \frac{1}{|\varphi_p(\mu_n)|^p n^2} < \infty,
	\end{equation}
	where $\mu_n = (t_n+t_{n+1})/2$. 
\end{theorem}
This result is significantly stronger than \eqref{eq:sharper} which merely says that the terms in the series in \eqref{eq:strong} tend to $0$. We may view this enhancement as resulting from our use of orthogonality on a global scale, in contrast to our former local study of pairs of zeros. It would be interesting to see if the spaces $B_p$ could be used to extract more nontrivial information about the extremal functions.

\section{Conjectures and further open problems} \label{sec:conj}

This section summarizes a number of conjectures and open problems suggested by our work. We split our discussion into four subsections. 

\subsection{Problems about \texorpdfstring{$\mathscr{C}_p$}{Cp}} We begin by restating the monotonicity conjecture, which is the main challenge pertaining to $\mathscr{C}_p$. 
\begin{conjecture}
	The function $p \mapsto \frac{\mathscr{C}_p}{p}$ is strictly decreasing on $(0,\infty)$. 
\end{conjecture}
If we could verify this hypothesis, it seems likely that we should be able get a simpler and cleaner proof of Theorem~\ref{thm:korevaar}. We should also be able to prove the following. 
\begin{conjecture}\label{conj:lowerbound} 
	There is a positive constant $A$ such that
	\[\mathscr{C}_p \geq \frac{p}{2}\big(1-A(p-2)\big)\]
	for $1 \leq p \leq 2$. 
\end{conjecture}

The preceding problem may conceivably be solved by a more refined analysis of the example functions $g_{1/2+1/p}$ (see Section~\ref{subsec:gamma}). Similar numerical computations as done with Bessel functions (see Figure~\ref{fig:plot}) suggest that this would be a reasonable approach. We have performed a preliminary analysis of $g_{1/2+1/p}$ showing that the conjectured bound holds for $p$ very close to $2$. 

One could also ask for an upper bound for $\mathscr{C}_p$ in the range $1 \leq p \leq 2$ similar to that obtained for $2 \leq p \leq 4$ in Theorem~\ref{thm:korevaar} or for a lower bound for $\mathscr{C}_p$ in the range $2 \leq p \leq 4$ similar to the one conjectured in Conjecture~\ref{conj:lowerbound}. Related to these questions is the challenge of extending the numerical work of H\"{o}rmander and Bernhardsson \cite{HB93} to the range $1\leq p \leq 4$; see the discussion at the end of Section~\ref{subsec:convex}.

It remains an interesting problem to determine the constant $c_0$ of Theorem~\ref{thm:Cp0}. We have no evidence suggesting what one should expect that constant to be, beyond the fact that it is contained in the interval $[1.1393830,1.1481785]$.

\subsection{Point evaluation on the imaginary axis} Fix $0<p<\infty$ and $y>0$. A natural extension of our problem is that of finding the smallest constant $C$ such that the inequality $|f(iy)|^p \leq C \|f\|_p^p$ holds for every $f$ in $PW^p$. A variant of this problem was studied by Korevaar~\cite{Korevaar49} who was interested in finding the best constant $K$, say $\mathscr{K}_p$, such that the inequality 
\begin{equation}\label{eq:pointkor} 
	|f(iy)|^p \leq K \sinc (i \pi p y) \|f\|_p^p 
\end{equation}
holds for all $y\geq0$ and all $f$ in $PW^p$. Note that plainly $\mathscr{C}_p \leq \mathscr{K}_p$. Korevaar proved that if $1 \leq p \leq 2$, then $\frac{1}{2} \leq \mathscr{K}_p \leq 1$ and if $2 \leq p < \infty$, then $1 \leq \mathscr{K}_p \leq p$. He conjectured, presumably based on the power trick, that $\mathscr{K}_p=p/2$ for all $1 \leq p < \infty$. However, this conjecture is refuted by \eqref{eq:hb} and the bound $\mathscr{C}_1\leq \mathscr{K}_1$. We believe that the following adjustment should hold. 
\begin{conjecture}\label{cor:newkor} 
	We have
	\[\mathscr{K}_p = 
	\begin{cases}
		\mathscr{C}_p, & \text{if } 0 < p < 2; \\
		p/2, & \text{if } 2 \leq p < \infty. 
	\end{cases}
	\]
\end{conjecture}

To justify this conjecture, we will now prove that 
\begin{equation}\label{eq:korevaarnew} 
	\lim_{y\to \infty} \sup_{f\in PW^p} \frac{|f(iy)|^p y e^{-p\pi y}}{\| f \|_p^p}=\frac{1}{4 \pi }. 
\end{equation}
To this end, we note that if $f$ is in $PW^p$, then $e^{i \pi z} f(z)$ is in the $H^p$ space of the upper half-plane (see \cite[Chapter~2]{Garnett07}), which implies that
\[e^{-p \pi y} |f(iy)|^p \leq \frac{\| f \|_p^p}{4\pi y}\]
holds for all $f$ in $PW^p$. This establishes the upper bound in \eqref{eq:korevaarnew}. On the other hand, the function $k_y(z):=(y-iz)^{-2/p}$ satisfies 
\begin{equation}\label{eq:ky} 
	|k_y(iy)|^p=\frac{\|k_y\|_p^p}{4\pi y}. 
\end{equation}
The function $k_y$ can be represented using the Fourier transform as
\[k_y(z) = \frac{2\pi}{\Gamma(2/p)}\int_0^\infty \xi^{2/p-1} e^{-y\xi} e^{i z\xi}\, \frac{d\xi}{2\pi}.\]
The function $z \mapsto e^{-i \pi z} k_y(z)$ is not entire. Let therefore $b$ be a smooth bump function such that $b(\xi)=1$ for $0 \leq \xi \leq \pi$ and $b(\xi)=0$ for $2\pi \leq \xi < \infty$. Then
\[f_y(z) := e^{-i \pi z} \int_0^\infty \widehat{k}_y(\xi)\, b(\xi) \,e^{i z\xi}\,\frac{d\xi}{2\pi}\]
is in $PW^p$. We next estimate
\[\left|e^{\pi y} k_y(iy)-f_y(z)\right| \leq e^{\pi y} \frac{2\pi}{\Gamma(2/p)} \int_0^\infty \xi^{2/p-1} e^{-2 y\xi} \big(1-b(\xi)\big)\,\frac{d\xi}{2\pi} \leq C_p e^{-\pi y},\]
using that $1-b(\xi)=0$ for $0 \leq \xi \leq \pi$. A similar estimate shows that
\[\int_{-\infty}^\infty |e^{i \pi x} k_y(ix)-f_y(ix)|^p \,dx \leq D_p e^{-\pi y} \|k_y\|_p^p.\]
We see from \eqref{eq:ky} that $\|k_y\|_p^p = \pi/y$. Combining these estimates, we then find that
\[ \frac{|f_y(iy)|^pe^{-p\pi y}y}{\| f_y \|_p^p}=\frac{1+o(1)}{4\pi} \]
when $y\to \infty$ which implies the lower bound in \eqref{eq:korevaarnew}.

We interpret \eqref{eq:korevaarnew} as saying that Korevaar's original conjecture is asymptotically true when $y\to \infty$. Curiously, it seems that the best constant in \eqref{eq:pointkor} for a fixed $y$, say $\mathscr{K}_p(y)$, may increase with $y$ when $p>2$ and decrease with $y$ when $0<p<2$. If we could establish such monotonicity properties, then Conjecture~\ref{cor:newkor} would be verified. 

The asymptotic behavior of $\mathscr{K}_p(y)$ when $y\to \infty$ reflects that point evaluation in $PW^p$ at $iy$ looks increasingly as point evaluation in $H^p$, up to an exponential factor. Since zeros of functions in $H^p$ can be divided out by Blaschke products, it is immediate that the power trick for $H^p$ extends to all positive powers. This means that we know the norm of point evaluation for all $H^p$ spaces once we know it for a single $H^p$ space. 

\subsection{Problems about the extremal functions and their zero sets} We begin with the most basic problem that remains to be resolved. 
\begin{conjecture}
	The extremal problem \eqref{eq:extremalproblem} has a unique solution for $0 < p < 1$. 
\end{conjecture}
Clearly, if we were able to establish that there is a unique solution $\varphi_p$ to \eqref{eq:extremalproblem}, then $\varphi_p$ would have to be an even function, since also $\varphi_p(-x)$ would solve \eqref{eq:extremalproblem}. This observation leads to a presumably easier problem. 
\begin{conjecture}
	Any solution of \eqref{eq:extremalproblem} for $0<p<1$ is even. 
\end{conjecture}
We turn to the zero sets of solutions of \eqref{eq:extremalproblem}. The following conjecture has only been established for $p\geq 1/2$ in Theorem~\ref{thm:zeroset}~(b). 
\begin{conjecture}\label{conj:discrete} 
	The zero set of any solution of \eqref{eq:extremalproblem} is uniformly discrete for all $0<p<\infty$. 
\end{conjecture}
One could also conjecture that the zero set of any solution of \eqref{eq:extremalproblem} is uniformly dense, but we will offer a more precise conjecture (see Conjecture~\ref{conj:precise}) which also implies Conjecture~\ref{conj:discrete}. The reason for singling out Conjecture~\ref{conj:discrete} is that it implies the following. 
\begin{conjecture}
	Theorem~\ref{thm:traces} extends to the full range $0<p<\infty$. 
\end{conjecture}
We now take the uniqueness of the extremal function for granted. We retain the notation $\varphi_p$ for the unique solution of \eqref{eq:extremalproblem}, and we let $t_n=t_n(p)$, $n\geq 1$, be the positive zeros of $\varphi_p$. We have the following monotonicity conjecture for the zeros. 
\begin{conjecture}
	The function $p\mapsto t_n$ is strictly decreasing for all $n\geq 1$. 
\end{conjecture}
As suggested in the beginning of Section~\ref{sec:2p4}, we also believe that 
\begin{equation}\label{eq:pbigger2} 
	n-1/2+1/p \leq t_n \leq n 
\end{equation}
when $2\leq p<\infty$. It also seems reasonable to expect that we have
\[ n \leq t_n \leq n-1/2+1/p \]
when $0<p \leq 2$.

The next conjecture is based on the surmise that the zeros of any extremal function are regularly distributed. 
\begin{conjecture}\label{conj:precise} 
	We have $\displaystyle \lim_{n\to \infty} (t_{n+1}-t_n)=1$ for all $p>0$. 
\end{conjecture}
The result of this paper closest to giving evidence for the preceding conjecture is the bound $\limsup_{n\to \infty} (t_{n+1}-t_n)\leq 1.0805$ when $p=1$ (see Section~\ref{sec:uppersep}).

Our analysis also suggests the following asymptotic behavior of $t_1$.
\begin{conjecture}\label{con:pasymp} 
	We have $\displaystyle \lim_{p\to \infty} t_1=1/2$ and $\displaystyle \lim_{p\to 0^+} p t_1 >0$. 
\end{conjecture}

Intimately related to the distribution of the zeros $\mathscr{Z}(\varphi_p)$ is the decay of $\varphi_p$. When discussing this relation beyond what was found in \eqref{eq:sharper} and Theorem~\ref{thm:midpointest}, it is quite natural to think of $x\varphi_p(x)$ as the generating function of the set $ \mathscr{Z}(\varphi_p)\cup \{0\}$. (Note that $x \varphi_2(x)= \frac{1}{\pi} \sin{\pi x}$ and $\mathscr{Z}(\varphi_2)\cup \{0\}=\mathbb{Z}$.) In view of Lemma~\ref{lem:zerotype}, the following assertion has been verified for $p\geq 2$. 
\begin{conjecture}
	$\mathscr{Z}(\varphi_p)\cup\{0\}$ is a uniqueness set for $PW^p$ if and only if $p\geq 1$. 
\end{conjecture}
Although it remains to verify the above for $p<2$ and in particular in the most accessible range $1<p<2$, we believe much more to be true. The rationale for the above conjecture is that we suspect that $x\varphi_p(x)/\dist(x,\mathscr{Z}(\varphi_p)\cup \{0\})$ in absolute value to behave roughly as the function $(1+|x|)^{1-2/p}$, which is integrable when $p<1$. Although this precise relation may be false, it seems likely that the following may hold. 
\begin{conjecture}
	Suppose that $1<p<\infty$. The function
	\[\left(\frac{|x\varphi_p(x)|}{\dist(x,\mathscr{Z}(\varphi_p)\cup \{0\})}\right)^p\]
	is a Muckenhoupt $(A_p)$ weight. 
\end{conjecture}
Here we recall that a positive function $w$ is said to be a Muckenhoupt $(A_p)$ weight (see \cite[pp. 246--247]{Garnett07}) if
\[\sup_{I} \frac{1}{|I|} \int_I w(x) \,dx \left( \frac{1}{|I|}\int_I \left(\frac{1}{w(x)}\right)^{\frac{1}{p-1}} \,dx \right)^{p-1} < \infty, \]
where the supremum is taken over all finite intervals $I$. It is easy to verify that $(1+|x|)^{p-2}$ is an $(A_p)$ weight. In view of \cite{LS97}, this conjecture would imply the stronger property of $\mathscr{Z}_0(\varphi_p)\cup\{0\}$ that it be a complete interpolating sequence for $PW^p$. This would give the following stronger link between $PW^p$ and the de Branges space $B_p$ of the preceding section: The expansion
\[f(x)=\sum_{n=-\infty}^{\infty} f(t_n) \frac{e_n(x)}{e_n(t_n)} \]
would hold also for $f$ in $PW^p$, and $\big(\sum_{n\in\mathbb{Z}} |f(t_n)|^p\big)^{1/p}$ would define another norm on $PW^p$ equivalent to the $L^p$ norm.

There are also basic questions about the Fourier transform of the extremal functions that we are unable to answer. For example, if $0 < p \leq 1$, then every function in $PW^p$ is integrable and has a continuous Fourier transform. Since these Fourier transforms are compactly supported, they are also integrable. We believe our extremal functions in the strictly convex regime should enjoy the same properties.
\begin{conjecture}\label{conj:cont} 
	If $1<p<\infty$, then the Fourier transform $\widehat{\varphi}_p$ is a continuous integrable function on $(-\pi,\pi)$. 
\end{conjecture}

Based on our belief that the extremal functions are close to the functions $f_p$ defined in \eqref{eq:fp}, we are also led to the following. 
\begin{conjecture} 
	Let $\varphi$ be a solution of \eqref{eq:extremalproblem}. Then $\widehat{\varphi}$ is nonnegative. Moreover, 
	\begin{enumerate}
		\item[(a)] if $0<p<2$, then $\widehat{\varphi}$ is decreasing on $(0,\pi)$ and $\widehat{\varphi}(\xi) \to 0$ as $\xi \to \pi^-$; 
		\item[(b)] if $2<p<\infty$, then $\widehat{\varphi}$ is increasing on $(0,\pi)$ and $\widehat{\varphi}(\xi) \to \infty$ as $\xi \to \pi^-$. 
	\end{enumerate}
\end{conjecture}
By a result of P\'{o}lya \cite[p.~373]{Polya18}, part (b) would imply that
\[ \frac{2n-1}{2} < t_n < \frac{2n+1}{2} \]
for $2<p<\infty$. This is in particular consistent with the first part of Conjecture~\ref{con:pasymp}. If we in addition to (b) assume that $\widehat{\varphi}$ is strictly convex, then it follows from \cite[p.~373]{Polya18} that we have
\[\frac{2n-1}{2} < t_n < n, \]
in accordance with the more precise conjecture \eqref{eq:pbigger2}. Presumably, other and more precise consequences for $\mathscr{Z}(\varphi)$ could be deduced from properties that we expect $\widehat{\varphi}$ to have.

In the next section, we will discuss another problem related to the Fourier transforms of our extremal functions.

\subsection{A problem about \texorpdfstring{$n$-fold}{n-fold} convolutions} \label{subsec:conv}

\begin{figure}
	\centering
	\begin{tikzpicture}[scale=1.5]
		\begin{axis}
			[axis equal image,
			axis lines=middle,
			axis line style=thin,
			hide y axis,
			y post scale=3.14,
			xmin=-10,
			xmax=10,
			xtick={-9.42,-3.14,3.14,9.42},
			xticklabels={$\scriptstyle -3\pi$,$\scriptstyle -\pi$,$\scriptstyle \pi$,$\scriptstyle 3\pi$},
			ymin=0,
			ymax=2,
			every axis x label/.style={ at={(ticklabel* cs:1.025)}, 
				anchor=west,},
			every axis y label/.style={ at={(ticklabel* cs:1.025)}, 
				anchor=south,},
			axis line style={->}]
	
			\addplot[thin,densely dotted] coordinates {(-9.42,1) (94.1,1)};
			\node at (-9.75,1) {$\scriptstyle 1$};
			\input{convolution1.tex}
		\end{axis}
	\end{tikzpicture}
	\caption{Plot of $\widehat{\psi}\ast\widehat{\psi}\ast\widehat{\psi}$ for $\widehat{\psi}$ as in \eqref{eq:bessel4}.  See Figure~\ref{fig:conv2} for a close-up of the part of the plot on {\color{red} the interval $(-\pi,\pi)$}.}
	\label{fig:conv1}
\end{figure}

To conform with the normalization of the Fourier transform given in Section~\ref{subsec:convex}, we define the convolution of $\widehat{f}$ and $\widehat{g}$ as follows:
\[\widehat{f}\ast\widehat{g}\,(\eta) := \int_{-\infty}^\infty \widehat{f}(\xi)\,\widehat{g}(\eta-\xi)\,\frac{d\xi}{2\pi}.\]
We will restrict our attention to integrable functions $\widehat{\psi}$ that are supported on $[-\pi,\pi]$ and continuous on $(-\pi,\pi)$. Such $\widehat{\psi}$ will be referred to as admissible functions. We let the $n$-fold convolution operator $\operatorname{C}_n$ act on admissible functions $\widehat{\psi}$ so that
\[\operatorname{C}_0\widehat{\psi} := \widehat{\psi}, \qquad \operatorname{C}_1\widehat{\psi} := \widehat{\psi} \ast \widehat{\psi}, \qquad \operatorname{C}_2\widehat{\psi} := \widehat{\psi} \ast \widehat{\psi} \ast \widehat{\psi}, \qquad \operatorname{C}_3\widehat{\psi} := \widehat{\psi} \ast \widehat{\psi} \ast \widehat{\psi} \ast \widehat{\psi}, \qquad\ldots\]
We now consider the following problem: Is there an admissible function $\widehat{\psi}$ such that 
\begin{equation}\label{eq:conveq} 
	\operatorname{C}_n{\widehat{\psi}}(\xi) = 1 
\end{equation}
for every $-\pi\leq \xi \leq \pi$? If such a function exists, we will refer to it as an admissible solution of \eqref{eq:conveq}. In the case $n=0$, the unique admissible solution is $\widehat{\psi} = \chi_{[-\pi,\pi]}$. Consider next the function supported on $[-\pi,\pi]$ and defined on $(-\pi,\pi)$ by 
\begin{equation}\label{eq:bessel4} 
	\widehat{\psi}(\xi) = \left(\frac{1}{c}\right)^{\frac{1}{3}} \left(1-\frac{\xi^2}{\pi^2}\right)^{-\frac{1}{2}} 
\end{equation}
for $c=1.7400645117$. Note that $\widehat{\psi}$ is equal to the function $\widehat{f}_4$ from \eqref{eq:fp} multiplied by $c^{-1/3}$ for $c = \|f_4\|_4^4$. See Figure~\ref{fig:conv1} for a plot of $\operatorname{C}_2 \widehat{\psi}$ computed with the packages \texttt{special.ellipk} and \texttt{integrate} from SciPy. At first glance, it would seem that \eqref{eq:bessel4} is an admissible solution of \eqref{eq:conveq} for $n=2$. A closer look---see Figure~\ref{fig:conv2}---reveals that this is not the case.

\begin{figure}
	\centering
	\begin{tikzpicture}[scale=1.5]
		\begin{axis}
			[axis equal image,
			axis lines=middle,
			axis line style=thin,
			xmin=0,
			xmax=8.3,
			xtick={1,4.14,7.28},
			xticklabels={$\scriptstyle -\pi$,$\scriptstyle 0$,$\scriptstyle \pi$},
			ymin=0,
			ymax=2,
			ytick={0.30731,1,1.71448},
			yticklabels={$\scriptstyle 0.993$,$\scriptstyle 1$,$\scriptstyle 1.007$},
			every axis x label/.style={ at={(ticklabel* cs:1.025)}, 
				anchor=west,},
			every axis y label/.style={ at={(ticklabel* cs:1.025)}, 
				anchor=south,},
			axis line style={->}]
			\addplot[thin,densely dotted] coordinates {(1,1) (7.28,1)};
			\input{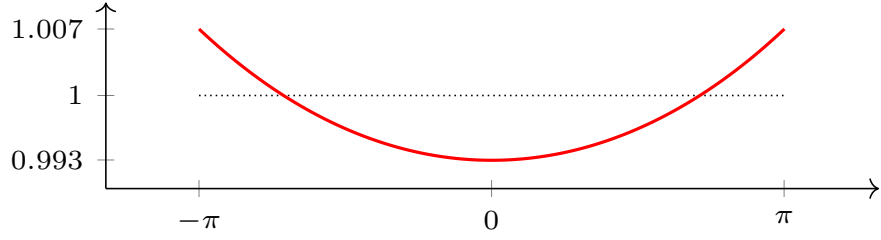}
		\end{axis}
	\end{tikzpicture}
	\caption{Close-up of the plot from Figure~\ref{fig:conv1} on {\color{red} the interval $(-\pi,\pi)$}.}
	\label{fig:conv2}
\end{figure}

We will next explain how this can be interpreted as evidence that $f_4$ is close to the extremal function $\varphi_4$. Let $p=2(k+1)$ for a nonnegative integer $k$ and let $\varphi_p$ denote the unique solution of our extremal problem \eqref{eq:extremalproblem}. As $\varphi_p$ is real by Lemma~\ref{lem:real}, it follows that
\[\Phi_p(x) = \frac{|\varphi_p(x)|^{p-2} \varphi_p(x)}{\|\varphi_p\|_p^p} = \frac{(\varphi_p(x))^{2k+1}}{\|\varphi_p\|_p^p}.\]
By Corollary~\ref{cor:otherside}, the convolution theorem for Fourier transforms, and Conjecture~\ref{conj:cont}, we find that if $p=2(k+1)$, then
\[\widehat{\psi}(\xi) := \|\varphi\|_p^{-\frac{p}{p-1}}\widehat{\varphi}_p(\xi)\]
is an admissible solution of \eqref{eq:conveq} for $n=2k$. On the other hand, if $\widehat{\psi}$ is an admissible solution of \eqref{eq:conveq} for $n=2k$, then
\[\psi(0) = \int_{-\infty}^\infty \psi^{2(k+1)}(x)\,dx.\]
If we also assume that $\widehat{\psi}$ is even so that $\psi$ is real, then this implies that $\psi$ is in $PW^{2(k+1)}$. It then follows from the fact that $\widehat{\psi}$ is a solution of \eqref{eq:conveq} and Theorem~\ref{thm:prpk} that $\psi$ be a multiple of $\varphi_p$ for $p=2(k+1)$.

We have just seen evidence that the equation \eqref{eq:conveq} has an admissible solution if $n$ is an even integer. We will next present evidence that the equation has no solution if $n$ is an odd integer. Suppose that $\widehat{\psi}$ solves the equation \eqref{eq:conveq} for $n=2k+1$ and has the desired properties. Since $\widehat{\psi}$ is integrable, it follows that $\psi$ is in $L^\infty(\mathbb{R})$. Consequently, $\psi^{2k+2}$ is also in $L^\infty(\mathbb{R})$. By the convolution theorem for Fourier transforms and the assumption that $\widehat{\psi}$ is a solution of \eqref{eq:conveq}, it follows that 
\begin{equation}\label{eq:psicontradictme} 
	f(0) = \int_{-\infty}^\infty f(x) \,\psi^{2k+2}(x)\,dx 
\end{equation}
for every $f$ in $PW^1$. If $\psi$ were real-valued on $\mathbb{R}$, then we would obtain a contradiction to \eqref{eq:psicontradictme} from the function
\[f(z):= \sinc^2\left(\frac{\pi}{2}(z+2)\right)\]
which is in $PW^1$, is nonnegative on $\mathbb{R}$, and satisfies $f(0)=0$. It follows from this analysis that the equation \eqref{eq:conveq} cannot have solutions $\widehat{\psi}$ that are real-valued and even. From this discussion, we are led to the following conjecture. 
\begin{conjecture}
	\mbox{} 
	\begin{enumerate}
		\item[(a)] There is exactly one real-valued admissible solution of \eqref{eq:conveq} when $n$ is an even integer, namely $\widehat{\psi}=\widehat{\varphi}_{n+2}/\|\varphi_{n+2}\|_{n+2}^{1+1/(n+1)}$. 
		\item[(b)] There are no real-valued admissible solutions of \eqref{eq:conveq} when $n$ is an odd integer. 
	\end{enumerate}
\end{conjecture}

\bibliographystyle{amsplain} 
\bibliography{pwpeval}

\end{document}